\documentclass[11pt,a4paper,reqno]{amsart}
\usepackage{geometry}
\usepackage{amsmath,amsfonts,amssymb,amsthm,mathrsfs}
\usepackage{graphicx,wrapfig}
\usepackage{hyperref}
\usepackage{tikz}
\usepackage{pgfplots}
\pgfplotsset{compat=1.17}
\usepgfplotslibrary{fillbetween}
\usepackage{dsfont}
\usepackage{enumerate}
\usepackage[shortlabels]{enumitem} 
\usepackage{comment}

\newtheorem{theorem}{Theorem}
\newtheorem{lemma}{Lemma}[section]
\newtheorem{claim}[lemma]{Claim}
\newtheorem{proposition}[lemma]{Proposition}
\newtheorem{corollary}[lemma]{Corollary}

\newtheorem{example}{Example}
\newtheorem{question}{Question}

\theoremstyle{remark}
\newtheorem{remark}{Remark}

\numberwithin{equation}{section}

\newcommand{\R}{\mathbb{R}}
\newcommand{\Z}{\mathbb{Z}}
\newcommand{\N}{\mathbb{N}}
\def\P{\mathbb{P}} 
\newcommand{\E}{\mathbb{E}}
\newcommand{\F}{\mathcal{F}}

\newcommand{\id}{\mathds{1}}
\newcommand{\var}{\text{\rm{Var}}}
\newcommand{\eps}{\varepsilon}
\newcommand{\capa}{\text{\textup{Cap}}}
\newcommand{\per}{ \text{\textup{Per}}}
\newcommand{\lp}{ \theta}
\newcommand{\enn}{n}
\newcommand{\ep}{\varepsilon}
\newcommand{\lm}{\lambda}
\newcommand{\al}{\alpha}
\newcommand{\cA}{\mathcal{A}}
\newcommand{\cB}{\mathcal{B}}

\newcommand{\vol}{\textrm{Vol}}

\newcommand{\norm}[1]{\left\Vert#1\right\Vert}

\definecolor{dark gray}{gray}{0.3}
\def\gray#1{{\color{dark gray} #1}}

\begin{document}

\title[Persistence and entropic repulsion of Gaussian fields]{Persistence and entropic repulsion of stationary Gaussian fields with spectral singularity at the origin}
\author{Naomi Feldheim}
\address{N. Feldheim, Department of Mathematics, Bar-Ilan University}
\email{naomi.feldheim@biu.ac.il}
\author{Ohad Feldheim}
\address{O. Feldheim, Einstein Institute of Mathematics, Hebrew University of Jerusalem}
\email{ohad.feldheim@mail.huji.ac.il}
\author{Stephen Muirhead}
\address{S. Muirhead, School of Mathematics, Monash University}
\email{stephen.muirhead@monash.edu}
\begin{abstract}
We compute the exact log-asymptotics of the persistence probability, and determine the entropic repulsion profile conditioned on persistence, for general $d$-dimensional stationary Gaussian fields with spectral singularity at the origin of order $\alpha \in [0,d)$. Under mild regularity conditions these are shown to be universal, depending only on $\alpha$ and $d$, and to have explicit formulations in terms of the capacity and equilibrium potential of the $\alpha$-Riesz kernel. This generalises a result of Bolthausen, Deuschel and Zeitouni on the Gaussian free field to a wide class of Gaussian fields with spectral singularity.
\end{abstract}
\thanks{}
\date{\today}
\maketitle

\section{Introduction}
Let $f$ be a continuous centred non-degenerate stationary Gaussian field (SGF), either on $\R^d$ or on $\Z^d$. Such a field is characterised by its \emph{covariance kernel} $K(x) = \E[f(0)f(x)]$, or equivalently by its \emph{spectral measure} $\mu=\mathcal F^{-1}[K]$, the unique non-negative, finite, symmetric measure whose Fourier transform is $K$, and we often write $f=f_\mu$. SGFs receive widespread attention because of their prevalence in physical systems: by the functional central limit theorem, a random field obtained as the sum of stationary nearly-independent infinitesimal contributions is an SGF.

A \emph{persistence}, or \emph{hard wall event}, is an event of the form $\{f|_D \ge \ell\}$ for some bounded typically large domain~$D$ and some level $\ell\in\R$. Persistence events are among the most studied rare events for SGFs, partly because they are fundamental and susceptible to analysis, but also because many physical systems behave, under suitable external conditions known as \emph{hard wall constraints}, like $f$ conditioned on persistence. Hence, in a sense, they are not rare. 

Several physical systems undergo, in response to a hard wall constraint, a predetermined macroscopic translation away from the wall. This phenomenon, known in the literature as
\emph{entropic repulsion}, has been established rigorously for a handful of models, and in particular for certain stationary Gaussian lattice models such as the \emph{Gaussian free field (GFF)}~\cite{bdz95,dg99} and the \emph{membrane model}~\cite{kur07}. It was conjectured, however, that this phenomenon is universal among strongly correlated fields~\cite{Vel06}. 

A key step in establishing entropic repulsion is the development of precise estimates, as $T \to \infty$, for the persistence probability on the log-scale
\[ \lp^\ell_{\mu}(T) := -\log \mathbb{P} \big[ f_\mu(x) \ge \ell \text{ for all } x \in B(T) \big],\]
 where $B(T)$ is the Euclidean ball of radius $T$ about the origin. For the models considered in \cite{bdz95,dg99,kur07}, this was achieved using a specific connection to the lattice Laplacian.

The theory of persistence for general SGFs was slower to develop, following its emergence in the classical papers of Rice \cite{ric45} and Slepian \cite{sle62}. The first special case to be addressed was that of non-negatively correlated processes, where one can exploit the sub-additivity of~$\lp^\ell_{\mu}(T)$. In this case Newell and Rosenblatt \cite{nr62} obtained the growth of $\lp_{\mu}(T):=\lp^0_{\mu}(T)$ up to a $\log T$ factor. More recently this was improved by Dembo and Mukherjee~\cite{dm17}, who recovered $\lp_{\mu}(T)$ up to a constant factor. 

In recent years, an intimate relation between the persistence probability and the behaviour of the spectral measure near the origin has surfaced. The coarse picture is the following \cite{ff15,ffn21}: for non-singular $\mu$ satisfying $\mu[B(\delta)]\asymp \delta^{\alpha}$ as $\delta\to0$, $\alpha > 0$, we have
\[\lp_{\mu}(T)  \begin{cases}
\asymp T^\alpha \log T & \alpha<d  ,\\
    \asymp T^d & \alpha= d ,\\
    \lesssim T^d\log T & \alpha> d ,\\
\end{cases}\]
where the three regimes correspond respectively to a spectral singularity at the origin, a `neutral' spectrum at the origin, and spectral decay at the origin; if the support of $\mu$ does not contain the origin we expect $\lp_{\mu}(T)\asymp T^{2d}$ (see \cite{ffjnn20} for the case $d=1$). These results leverage \textit{spectral decompositions}, namely the fact that if
$\mu = \mu_1+\mu_2$ for spectral measures $\mu,\mu_1,\mu_2$, then $f_\mu \stackrel{d}{=} f_{\mu_1}\oplus f_{\mu_2}$ where $f_{\mu_1}$ and $f_{\mu_2}$ are independent SGFs. The governing paradigm is that persistence events typically occur as a balance between (i) raising the Gaussian coefficient for a single low frequency component of the field, and (ii) the suppression of the remaining components of the field.

Recently, \cite{ffm25} established finer asymptotics in the `neutral' regime $\alpha=d$. 
%Denoting \[ \lp^\ell_{\mu}(T) := -\log \mathbb{P} \big[ f_\mu(x) \ge \ell \text{ for all } x \in B(T) \big],\]
It was shown that, if $\mu$ is non-singular with finite and positive spectral density at the origin (i.e.\ satisfying $\displaystyle\mu'(0)=\lim_{\delta\to 0}\mu[B(\delta)]\delta^{-d}\in (0,\infty)$), then there exists a \textit{persistence exponent} $\lp^\ell_\mu \in (0,\infty)$ for all $\ell\in\R$, in the sense that
  \[ \lp^\ell_\mu(T) = \lp^\ell_\mu \cdot T^{d} \big(1+o(1)\big)  .\]
Nevertheless, even the precise value of $\theta_\mu:=\theta^0_\mu$ remains elusive, and can be  computed only for a handful of one-dimension examples, such as $K(x)= e^{-|x|}$ (the Ornstein-Uhlenbeck process) \cite{ric58} and $K(x) =(\cosh(x))^{-1}$ \cite{ps18}. Generally  $\theta_\mu$ is predicted to have complicated dependence on the spectral measure, and even efficient approximation methods are unknown. The subtle behaviour, in this regime, of fields conditioned to persist seems entirely out-of-reach at the moment; heuristically the conditioned field remain constant on average and entropic repulsion does not occur. 
  
In this paper we draw our attention to fields with a spectral singularity at the origin, corresponding to $\alpha < d$. In this regime it is less unlikely for the low frequency component to be raised, and so a milder suppression of the remaining components is required for persistence. This 
has two implications. First, the milder suppression yields a remarkable universality in the persistence probability decay rate -- for a wide class of fields, after proper normalisation, it depends only on the order $\alpha$ of the singularity and the dimension~$d$. This is a striking example of \textit{spectral condensation}: persistence is essentially governed by a one-dimensional statistic of the spectral singularity. Second, the conditioned process is raised further and further from zero by the change in the low frequency component, resulting in entropic repulsion. In contrast to the GFF case, the predetermined macroscopic translation need not be a simple flat shift in general, and its shape can be recovered as a solution to a minimisation problem. This work is dedicated to establishing these implications for SGFs in broad generality.

We conjecture a third, finer implication
of the mild suppression of the remainder: after proper re-centring, 
the conditioned process should become closer and closer in shape to the original process, converging to it in distribution. This implication, which would generalise a phenomenon known to occur for the GFF, is left as an open problem and is further discussed in Section~\ref{ss:oq}. 

\smallskip
Our approach is largely inspired by a line of works studying large deviation events involving excursion sets of the GFF \cite{sni15, nit18,sni19,cn20,grs21}. In \cite{ms22} some of these results were generalised to a wider class of strongly correlated SGFs. A key difference in our work is that we incorporate the spectral analysis from \cite{ffn21} (and its further development in \cite{ffm25}) and crucially exploit the \textit{convexity} of the persistence event. This gives us access to potent \textit{smoothing} techniques, which allows us to treat a very wide class of SGFs. 

\smallskip
\textbf{Notation.} We make use of standard Landau notation. In addition, for positive functions $f$ and $g$ we use $f \asymp g$, $f \sim g$, and $f \ll g$ to indicate respectively that $f/g$ is bounded above and away from zero, that $f/g \to 1$, and that $f/g \to 0$. For a measure $\mu$ on $\R^d$, we denote by $ \|\mu\|$ its total mass (or total variation if $\mu$ is signed).
  
\subsection{Persistence and entropic repulsion}
We begin by presenting our results for fields with \textit{spectral singularity at the origin of order} $\alpha \in [0,d)$, meaning that 
\begin{equation}
\label{e:blowup}
 \lim_{\delta\to 0}\frac{ \log \mu[B(\delta)]}{\log \delta} =\alpha.
 \end{equation}
 Although for simplicity we consider only persistence above the mean level ($\ell = 0$), the results hold in general for persistence above arbitrary \textit{fixed} level $\ell \in \R$ (see Section \ref{ss:gen2}).

 Our first result gives precise log-asymptotics for the persistence probability. To state this we extract from $\mu$ an ideal low frequency component, raising which is the most efficient way to promote persistence. More precisely, we seek to decompose $f_\mu=\zeta h \oplus g$, where $h$ is a deterministic function satisfying $h|_D \ge 1$, $\zeta$ is a centered Gaussian variable with as large a variance as possible, and $g$ is a remainder process. Asymptotically, the probabilistic cost of making $\zeta$ sufficiently large will determine the probability of persistence at logarithmic scale, and $h$ will serve as the shape of the entropic repulsion.
 
 The optimum choice $h = h_f(D)$ is called the \textit{equilibrium potential} of $D$ w.r.t.\ $f$, and $\capa_f(D) = 1/\var(\zeta)$ is known as the \emph{capacity}. Formally these are defined as
\begin{align}
\label{e:cap2}
\capa_f(D)  & := \min \big\{ \|h\|^2_H : h \in H, h \ge 1 \text{ on } D \big\},
\\
h_f(D)  & := \arg\min \big\{ \|h\|^2_H : h \in H, h \ge 1 \text{ on } D \big\}, \text{ if } \capa_f(D)<\infty, \notag    \end{align}
where $H$ is the reproducing kernel Hilbert space (RKHS) associated to $\mu$ (i.e.\ $H=\{\F[g\, d\mu]: \, g\in L^2_\mu \}$ and $\|g d\mu\|_H = \|g\|_{L^2_\mu}$). The existence and uniqueness of $h$, as well as the following dual representations, are given in Proposition~\ref{p:basiccap}:
 \begin{align}
   \label{e:cap1}
 \nonumber  \capa_f(D)  & := \Big( \min_{ \nu \in \mathcal{P}(D) } \int_D \int_D K(x-y) d\nu(x) d\nu(y) \Big)^{-1}, \\
 \nu_f &:= \arg \min\Big\{  \int_D \int_D K(x-y) d\nu(x) d\nu(y)  :   \nu \in \mathcal{P}(D) \Big\}, \\
   h_f(D)  &:= \capa_f(D) ( K * \nu_f)  \nonumber ,
 \end{align}
 where $\mathcal{P}(D)$ is the set of Borel probability measures on $D$. A probability measure $\nu_f$ which is a minimiser as in \eqref{e:cap1} is known as an \textit{equilibrium measure}.

Beyond \eqref{e:blowup} we impose two additional regularity conditions on the field. 
Write $\mu = \mu_{a.c.} + \mu_d + \mu_{s.c.}$ for the decomposition of a measure $\mu$ into an \textit{absolutely continuous} component, a \textit{discrete} component and a \textit{singular continuous} component respectively. We require:
\begin{flalign}
\text{(Non-singularity)} && \mu_{a.c.} \neq 0 \qquad \text{and} \qquad \mu_{s.c.} = 0, && \label{e:sing}
\end{flalign}\vspace{-20pt}
\begin{flalign}
\text{(Regular capacity growth)} && \forall \epsilon > 0, \, \capa_f \big(B(T + T^{1-\epsilon})\big) \sim \capa_f \big(B(T) \big) \text{ as } T \to \infty. && \label{e:capcon}
\end{flalign}

In particular, if $\mu = \mu_{a.c.}$ with density $\rho$, conditions \eqref{e:blowup}, \eqref{e:sing}, and \eqref{e:capcon} hold under the simpler assumption that either
\begin{equation}
    \label{e:pure}
\lim_{\lambda \to 0} \rho(\lambda) |\lambda|^{d-\alpha} = c  \qquad \text{or} \qquad \lim_{x \to \infty} K(x) |x|^{\alpha} = c , \quad \alpha \in (0,d) , c > 0.
\end{equation}
This is further discussed in Section~\ref{ss:uni} and established in Section~\ref{s:rv}.

\smallskip
Our first result gives log-asymptotics of the persistence probability in terms of  dimension, order of singularity, and capacity:

\begin{theorem}[Log-asymptotics for the persistence probability]
\label{t:main}
Let $\alpha \in [0,d)$ and let $f$ be an SGF on $\R^d$ or $\Z^d$ satisfying \eqref{e:blowup}, \eqref{e:sing} and \eqref{e:capcon}. Then, as $T \to \infty$,
\begin{equation}
\label{e:main}
 \lp_\mu(T) \sim  m (d-\alpha)   \capa_f\big(B(T)\big) \log T,
 \end{equation}
 where $m = \|\mu_{a.c.}\| > 0$.
\end{theorem}

The regularity conditions \eqref{e:sing} and \eqref{e:capcon} are discussed further in Sections \ref{ss:gen} and \ref{ss:rem} below. As the following example attests, some assumption on $\mu_{s.c.}$ is necessary for~\eqref{e:main}:

\begin{proposition}
    \label{p:counterlim}
    For every $\alpha \in (0,1)$ there exists a stationary Gaussian process on $\R$ satisfying \eqref{e:blowup}, $\mu_{a.c.} \neq 0$, and \eqref{e:capcon} such that
    \[     \liminf_{T \to \infty}
 \frac{ \lp_\mu(T) / (1-\alpha) }{  \capa_f\big(B(T)\big) \log T } = \|\mu_{a.c.}\|   <  \|\mu\|  = \limsup_{T \to \infty}
 \frac{ \lp_\mu(T)  / (1-\alpha) }{  \capa_f\big(B(T)\big) \log T }   .\]
\end{proposition}

\medskip
Our second result establishes the entropic repulsion of the field conditioned to persist. We require two additional conditions:
\begin{flalign}
\text{(Doubling)} && \exists c>0\  \forall \delta>0,\quad \mu[B(2\delta)] \le c \mu[B(\delta)], && 
\label{e:doub}
\end{flalign}\vspace{-20pt}
\begin{flalign}
\text{(Origin dominance)} &&\hspace{-8pt}\exists c>0\  \forall \delta>0,\quad  \sup_{u} \mu[u + B(\delta)]  \le c \mu[B(\delta)] .
&& 
\label{e:ds}
 \end{flalign}

In particular, \eqref{e:doub} and \eqref{e:ds} are satisfied if $\mu = \mu_{a.c.}$ has a density which is bounded outside a neighbourhood of the origin, and one of the conditions in \eqref{e:pure} holds (see Section \ref{ss:uni} and Proposition \ref{p:pos}).

\smallskip
To state the theorem we also define a suitable class of test functions. Denote by $\mathcal{M}$ the set of signed Radon measures on $\R^d$, and write $\mathcal{T} \subset \mathcal{M}$ for the set of signed measures of the form $\xi \id_D dx$, where $\xi \in C^0(\R^d)$ is Lipschitz, $D$ is a compact domain with piece-wise smooth boundary, and $dx$ denotes the Lebesgue measure. For $\eta \in \mathcal{T}$, we write $\eta_T = T^{-d} \eta(\cdot/T)$ for the rescaled signed measure of identical total mass, and $\langle g , \eta \rangle = \int g d\eta$ for the usual $L^2$-inner product, interpreting the integral as a sum over $\Z^d$ if $g$ is supported on $\Z^d$ and $\eta$ has a density. 

\smallskip
Denote the field $f$ conditioned on $\per(B(T))$ by $\tilde{f}_T$, and abbreviate $h_T = h_f(B(T))$, where $h_f(D)$ is the equilibrium potential defined in \eqref{e:cap1}. The following entropic repulsion result states that $\tilde{f}_T$ is, with high probability, close to a $\sqrt{2m(d-\alpha) \log T}$ multiple of~$h_T$ in the sense of macroscopic averages:

\begin{theorem}[Entropic repulsion]
\label{t:er}
Assume the conditions of Theorem \ref{t:main} hold, and suppose in addition that \eqref{e:doub} and \eqref{e:ds} are satisfied.  Then for every $\eta \in \mathcal{T}$, as $T \to \infty$, 
\[ \left| \frac{ \langle \tilde{f}_T, \eta_T \rangle }{ \sqrt{2m(d-\alpha) \log T }   }  - \langle h_T, \eta_T \rangle  \right| \to 0 \quad \text{in probability}. \]
\end{theorem}

In particular, plugging in $\eta = \id_{B(1)} dx \in \mathcal{T}$, and recalling that $h_T|_{B(T)} \ge 1$, we deduce repulsion of the typical spatially-averaged height of the conditioned process:

\begin{corollary}[Entropic repulsion of the typical average height]
Assume the conditions of Theorem \ref{t:er} hold. Then for every $\delta > 0$, as $T \to \infty$,
\[ \P \Big[ \frac{1}{\normalfont{\vol}\big(B(T)\big)} \int_{B(T)} f(x) \,dx \ge  \sqrt{2m(d-\alpha - \delta) \log T } \: \Big| \: \per\big(B(T)\big) \Big] \to 1 .\]
\end{corollary}

For the reader familiar with the GFF, we emphasise that the equilibrium potential $h_T$ need not be asymptotically flat on $B(T)$ ($\sup_{x \in B(T)} h_T(x) = 1 + o(1))$ as it is for the GFF. For instance, for the membrane model, $\lim_{T \to \infty} h_T(0) \in (1,\infty)$ whereas $\lim_{T \to \infty} h_T(x) = 1$ for $x \in \partial B(T)$.

\smallskip
Theorems~\ref{t:main} and ~\ref{t:er} were previously known only for a handful of special fields related to the lattice Laplacian, such as the GFF~\cite{bdz95, dg99} ($\alpha=d-2$), the membrane model \cite{kur07} ($\alpha=d-4$), and certain generalisations of these~\cite{sak03} ($\alpha=d-2k$, $k\in \mathbb{N}$).
In Section~\ref{ss:oq} we discuss stronger entropic repulsion results that have been established for the GFF, which we conjecture to hold in much greater generally.

\subsection{Universality for regularly varying fields}
\label{ss:uni}
As an extension of Theorems \ref{t:main} and \ref{t:er}, we provide refined results for a wide class of `regularly varying' fields, where a stronger notion of universality governs both persistence and entropic repulsion.

\smallskip
Recall the $\alpha$\textit{-Riesz kernel}, $\alpha \in [0,d)$,
\[ k_\alpha(x) = B_{\alpha,d} \, |x|^{-\alpha} \ ,  \quad  B_{\alpha,d} = \frac{ \Gamma(1+ \alpha/2) \Gamma(d/2) }{ \pi^{\alpha} \Gamma((d-\alpha)/2) } , \] 
which is the canonical kernel on $\R^d$ with spectral singularity of order $\alpha$ at the origin, and coincides with the classical \textit{Coulomb-Newton kernel} if $\alpha = d-2$. We associate to $k_\alpha$ a spectral measure $\mu_\alpha$ which satisfies $k_\alpha = \mathcal{F}[\mu_\alpha]$: if $\alpha \in (0,d)$, $\mu_\alpha$ is the absolutely continuous measure with density
\[ \rho_\alpha(\lambda) := A_{\alpha,d}|\lambda|^{\alpha-d} \ , \quad A_{\alpha,d} = \frac{\alpha \Gamma(\tfrac d 2)}{2\pi^{d/2}}, \]
while for $\alpha=0$ we have $\mu_0 := \delta_0$, a Dirac mass at $0$. In particular $\mu_\alpha[B(\delta)]  = \delta^\alpha$, which motivates the choice of the normalising constant $B_{\alpha,d}$.

\smallskip 
Although $k_\alpha$ is singular (i.e., $\|\mu_\alpha\|=\infty$ for $\alpha\neq 0$), it is still possible to define the notions of capacity and equilibrium potential with respect to $k_\alpha$, via \eqref{e:cap2}--\eqref{e:cap1} (see \cite{lan72}). 
Let $c_\alpha = c_{\alpha,d}\in (0,\infty)$ and $  h_\alpha(\cdot) = h_{\alpha,d}(\cdot)\in (0,
\infty)$ denote the capacity and equilibrium potential of the unit ball $B(1)$ with respect to $k_\alpha$. Due to symmetry, one can compute $c_\alpha$ and $h_\alpha$ explicitly (see Appendix \ref{a:explicit}); in particular $c_0 = 1$. We note that $h_\alpha$ is continuous, strictly-positive, and decays at infinity. We also note a phase transition at the Coulomb-Newton point $\alpha = d-2$: if $\alpha \ge d-2$ then $h_\alpha|_{B(1)} = 1$, whereas if $\alpha \in (0,d-2)$ then $h_\alpha|_{\mathbb{S}^{d-1}} = 1$ but $h_\alpha > 1$ on the interior of $B(1)$.

\smallskip
Recall that a function $w: \mathbb{R}^+ \to \mathbb{R}^+$ is said to be \textit{slowly varying} as $x\to x_0$ if, for all $u>0$, $\tfrac{w(u x)}{w(x)}\to 1$ as $x\to x_0$. A real-valued function $W$ on $\R^d$ or $\Z^d$ is \textit{slowly varying} as $x \to x_0$ if there exists a slowly varying function $w:\R_+\to\R_+$ such that $W(x) \sim w(|x|)$ as $x\to x_0$.

\smallskip
Our next results establish a stronger sense of universality for both persistence probability (Theorem~\ref{t: univ pro}) and entropic repulsion (Theorem~\ref{t: univ ent}) for fields which are asymptotically described by the $\alpha$-Riesz kernel in the sense of regular variation. 
We show that for such fields both 
the exponential order of the persistence probability and the shape of the entropic repulsion depend only the order $\alpha$ of the singularity and the dimension~$d$.
We emphasise that these results do not assume that $K \ge 0$ and can be applied to both discrete and continuous fields.

\begin{theorem}[Universality of persistence probability]
\label{t: univ pro}
Let $\alpha \in [0,d)$, let $f$ be an SGF on $\R^d$ or $\Z^d$, non-singular as per \eqref{e:sing}, and suppose that one of the following hold:
\begin{enumerate}[{\rm (i)}]
\item\label{item:1} 
\[ \frac{K(x)}{k_\alpha(x)} \text{ is slowly varying, as } x\to\infty. \]
\item\label{item:2} $\alpha>0$ and $\mu$ has density $\rho$ in a neighbourhood of the origin such that
\[ \frac{\rho(\lambda)}{\rho_\alpha(\lambda)} \text{ is slowly varying, as }\lambda \to 0.\] 

\item\label{item:3} $\mu$ is radial (i.e.\ the field $f$ is isotropic), and
\[ \frac{\mu[B(\delta)]}{\mu_\alpha[B(\delta)]} \text{ is slowly varying, as } \delta\to 0. \]

\item\label{item:4} $\alpha=0$ and $\mu$ has an atom at the origin.

\end{enumerate}
Then
\begin{equation}\label{get:univ1}
\theta_\mu(T) \sim c_{\alpha,d}\, m(d-\alpha) \frac{\log T}{\mu[B(\tfrac 1 T)]}, \quad \text{ as }T\to\infty,
\end{equation}
where $m=\|\mu_{a.c.}\|$.
\end{theorem}

\begin{remark}
\label{r:k}
In case~\ref{item:1} one can rewrite~\eqref{get:univ1} in terms of $K$, namely
 \begin{equation}
 \label{get:univ1_K}
 \theta_\mu(T) \sim  c_{\alpha,d} \, B_{\alpha,d} \, m  (d-\alpha)  \frac{ \log T}{K(T)} , \end{equation}
 since in that case $\mu[B(1/T)] \sim K(T) / B_{\alpha,d}$ by Proposition \ref{p:taub} below. Similarly, in case \ref{item:2} one can rewrite~\eqref{get:univ1} in terms of $\rho$ as
 \begin{equation}
 \label{get:univ1_rho}
 \theta_\mu(T) \sim  c_{\alpha,d} \, A_{\alpha,d} \,  m  (d-\alpha)  \frac{T^d \log T}{\rho(\tfrac 1 T)} , \end{equation}
 since in that case $\mu[B(\tfrac 1 T)] \sim T^{-d} \rho(\tfrac 1 T)  / A_{\alpha,d}$ by Proposition \ref{p:taub}.
\end{remark}

\begin{theorem}[Universality of entropic repulsion]\label{t: univ ent}
Assume the conditions of Theorem~\ref{t: univ pro} hold, along with the origin dominance condition \eqref{e:ds}. Then 
 \begin{equation}
     \label{get:univ2} \frac{ \langle \tilde{f}_T, \eta_T \rangle }{ \sqrt{2m(d-\alpha) \log T }  }  \to \langle h_\alpha , \eta \rangle  \quad \text{in probability, \quad as } T\to\infty, \end{equation}
for every $\eta \in \mathcal{T}$ in cases \ref{item:1},\ref{item:2},\ref{item:4}, or for every radial $\eta \in \mathcal{T}$ in case~\ref{item:3}.
\end{theorem}

\begin{remark}
A particular case of~\eqref{get:univ2} is
\begin{equation}
     \label{e:er3}
   \frac{1}{\sqrt{\log T} \, \normalfont{\vol}(B(T)) } \int_{B(T)} \tilde{f}_T(x) \,dx  \to   \frac{\sqrt{2m(d-\alpha )} }{\normalfont{\vol}(B(1)) } \int_{B(1)}h_\alpha(x)
 \,dx \quad \text{in probability}.   
 \end{equation} 
If $\alpha \in \{0\} \cup [d-2,d)$ the limiting constant in \eqref{e:er3} is $\sqrt{2m(d-\alpha )}$, since in these cases $h_\alpha|_{B(1)} = 1$ (see Appendix \ref{a:explicit}).
\end{remark}

We deduce Theorems~\ref{t: univ pro} and \ref{t: univ ent} from Theorems~\ref{t:main} and~\ref{t:er} using certain scaling properties of the capacity and equilibrium potential stated in Section~\ref{s:rv}. To validate the conditions in Theorems~\ref{t:main} and~\ref{t:er} we make use of the following result:

\begin{proposition}
\label{p:cond}
Let $\alpha \in [0,d)$, and let $f$ be an SGF on $\R^d$ or $\Z^d$ satisfying one of the conditions~\ref{item:1}, \ref{item:2}, \ref{item:3} and \ref{item:4} in Theorem~\ref{t: univ pro}. Then:
\begin{enumerate}
    \item $\delta \mapsto \delta^{-\alpha} \mu[B(\delta)]$ is slowly varying as $\delta \to 0$; and
    \item $T \mapsto T^{-\alpha} \capa_f(B(T))$ is slowly varying as $T \to \infty$.
\end{enumerate}
In particular \eqref{e:blowup}, \eqref{e:capcon} and \eqref{e:doub} hold.
\end{proposition}

\begin{remark}
 Conditions \ref{item:1}, \ref{item:2} and \ref{item:3} are generally not equivalent, but have considerable overlap (see~\cite{lo13}). It would of  interest to fully characterise when \eqref{get:univ1} and \eqref{get:univ2} hold.
\end{remark}

\begin{remark}
\label{r:ds}
As we show in Proposition \ref{p:pos} below, \eqref{e:ds} is automatically satisfied in case~\ref{item:1} of Theorem \ref{t: univ pro}, and evidently also in case \ref{item:4}, but not in general.
\end{remark}

\subsection{Applications}
\label{s:examples}
In this section we apply our results to several examples, which demonstrate the prevalence of the `universal' persistence behaviour presented in Section \ref{ss:uni}.

\smallskip
We first consider a class of discrete fields related to the lattice Laplacian, which includes the GFF and membrane model:

\begin{example}[Discrete fields related to the Laplacian]
\label{e:lap}
{\rm
Let $1 \le k \le m$, $d \ge 2k + 1$, $q_k,\ldots,q_m \in \R$ with $q_k > 0$, and suppose that $\sum_{j=k}^m q_j r^j > 0$ for every $r \in (0,2]$. Then the operator 
\[ J = \sum_{j = k}^m q_j (-\Delta)^j\] 
is positive definite on $\Z^d$, where $\Delta$ is the lattice Laplacian. Let $f$ be the SGF on $\Z^d$ with covariance $K = J^{-1}$. It is known~\cite{sak03} that $f$ has a spectral density and $K(x) \sim (1/ \gamma) |x|^{2k-d}$, for a constant $\gamma = \gamma(k,q_k,d) > 0$. Then the conclusions of Theorems \ref{t: univ pro} and \ref{t: univ ent} hold with $\alpha = d - 2k$ and $m=K(0)$. Plugging this into Remark~\ref{r:k} (and noting that $K(0) = \|\mu\|$) we obtain
\[
\lp_\mu(T) \sim c T^{d-2k} \log T, \quad \text{where } c =  K(0) \, 2k \,   \gamma  \, c_{d-2k,d} \, B_{d-2k,d}   > 0.
\]}
\end{example}

This result was already known for the GFF ($k=m=1$) \cite{bdz95} and membrane model ($k=m=2$) \cite{sak03,kur07}, but in general only under a certain additional assumption which implies that $K \ge 0$ \cite{kur07}.

\smallskip
We further generalise these examples to a much wider class of discrete and continuous fields:

\begin{example}[General SGFs]
\label{e:smooth}
{\rm
$\,$
\begin{enumerate}[itemsep=4pt]
\item Let $\alpha \in (0,d)$ and consider an SGF on $\R^d$ or $\Z^d$ with $\mu_{s.c.} = 0$ and $K(x) \sim  (1/ \gamma) |x|^{-\alpha}$ for $\gamma > 0$. Then the conclusions of Theorems \ref{t: univ pro} and \ref{t: univ ent} hold. In particular \linebreak$\lp_\mu(T) \sim  c T^{\alpha} (\log T)$, where $c = K(0)  (d-\alpha)   \gamma  B_{\alpha,d}  c_{\alpha,d} > 0$.
\item Consider an SGF on $\R^d$ or $\Z^d$ satisfying $\mu_{s.c.} = 0$ and $K(x)  \sim (1/\gamma) (\log |x|)^{-\beta}$ for $\gamma,\beta> 0$. Then the conclusions of Theorems \ref{t: univ pro} and \ref{t: univ ent} hold for $\alpha = 0$. In particular $\lp_\mu(T) \sim K(0)  d  \gamma (\log T)^{1+\beta}$.
\item Consider an SGF on $\R^d$ or $\Z^d$ with $\mu_{s.c.} = 0$ and such that $\mu$ has a non-trivial atom at the origin. Then the conclusions of Theorems \ref{t: univ pro} and \ref{t: univ ent} hold for $\alpha = 0$. In particular $\P[ \per(B(T))  ] =  T^{ -  \|\mu_{a.c.}\|   d / \mu(0)   + o(1) }$.
 \end{enumerate}
 }\end{example}

For the validity of these examples see Remarks \ref{r:k} and \ref{r:ds}, and also recall that $B_{0,d} = c_{0,d} = 1$.

\subsection{General bounds on the persistence probability}
\label{ss:gen}
We deduce Theorem \ref{t:main} from upper and lower bounds on the persistence probability which are valid under weaker assumptions and are of independent interest. To state these we introduce constants
\[\alpha^- := \liminf_{\delta \to 0} \frac{\log  \mu[B(\delta)] }{\log \delta} \quad \text{and} \quad \alpha^+ :=  \limsup_{\delta \to 0}   \frac{\log  \mu[B(\delta)] }{\log \delta} \]
which bound the order of the spectral singularity from below and above respectively; in particular if \eqref{e:blowup} holds for $\alpha \in [0,\infty]$ then $\alpha^- = \alpha^+ = \alpha$, and vice versa. We say that $\mu$ is \textit{regular} if for every $\eta > 0$ there exists an $\eps > 0$ such that
\begin{equation}
    \label{e:rss}
    \limsup_{\delta \to 0 } \frac{ \log \mu[B(\delta)] - \log  \mu[B(\delta^{1-\eps})]}{\log \delta} < \eta  . 
    \end{equation}
This is satisfied, for instance, if \eqref{e:blowup} holds for $\alpha \in  [0, \infty)$ (take $\eps < \eta/\alpha$), or under \eqref{e:doub}.

\smallskip
We also introduce two constants
which control the scale of the moderate deviation regime for the infimum of $f$ once the spectral singularity has been removed:
\begin{align}
    \label{e:m-}
 & m^-  := \sup \Big\{ u \ge 0 : \exists  \delta_0  > 0 \ \text{such that, for all } \delta \in (0,\delta_0) \text{ and } \beta \in [0,d), 
 \text{ as } T \to \infty \\
 &  \nonumber \qquad   \qquad  - 
 \log \P \Big[ \inf_{x \in B(T)} f_{\mu|_{B(\delta)^c}}(x) \ge - \sqrt{2 u ( d - \beta)  \log T} \Big] \ge T^{\beta + \delta_0(d-\beta)  }  \text{ eventually} \Big\}
\end{align}
and
\begin{align}
    \label{e:m+}
 & m^+  := \inf \Big\{ u \ge 0 :  \exists  \delta_0  > 0 \ \text{such that, for all } \delta \in (0,\delta_0) \text{ and } \beta \in (0,d), \text{ as } T \to \infty  \\ 
&  \nonumber \qquad     \qquad  - \log \P \Big[ \inf_{x \in B(T)} f|_{\mu_{B(\delta)^c}}(x) \ge -  \sqrt{2 u (d-\beta)  \log T} \Big] \le T^{\beta - \delta_0(d-\beta)  }  \text{ eventually}   , \\
 & \nonumber \qquad \qquad \qquad \qquad \text{ and  } \P \Big[ \inf_{x \in B(T)} f|_{\mu_{B(\delta)^c}}(x) \ge -  \sqrt{2 u d  \log T} \Big]  \to 1 \Big\}  .
\end{align}
It is simple to check that if $f$ is an i.i.d.\ Gaussian field on $\Z^d$ then $m^- = m^+ = \|\mu\|$. In Corollary~\ref{c:m} we show that, in general,
\begin{equation}
    \label{e:mbounds}
m \le m^- \le m^+ \le m + \|\mu_{s.c.}\|, 
\end{equation} 
and in particular $m^- = m^+ = m$ if $\mu_{s.c.} = 0$.

 \begin{theorem}[Bounds on the persistence probability]
\label{t:bounds}
Let $f$ be an SGF on $\R^d$ or $\Z^d$. Then for every $\delta > 0$, as $T \to \infty$, eventually 
\[    \lp_\mu(T) \le (m^+  (d-\alpha^-)_+ + \delta) \log T\, \capa_f(B(T + T^\delta ) )  , \]
where $(x)_+ = \max\{0,x\}$. Moreover if $\mu$ is regular in the sense of \eqref{e:rss}, then for every $\delta > 0$ there exists an $\eps > 0$ such that, as $T \to \infty$, eventually 
\[ \theta_\mu(T)  \ge (m^- (d-\alpha^+ ) - \delta) \log T\,\capa_f(B(T-T^{1-\eps}))  .\]
\end{theorem}

\begin{proof}[Reduction of Theorem \ref{t:main} to Theorem \ref{t:bounds}]
Under the assumptions of Theorem~\ref{t:main} we have $\alpha^- = \alpha^+ = \alpha \in [0,d)$, $m^- = m^+ = m$ (by \eqref{e:mbounds}), and for every $\delta,\eps > 0$,
\[ \capa_f(B(T+T^\delta)) \sim \capa_f(B(T - T^{1-\eps})) \sim \capa_f(B(T)) .\]
Then Theorem \ref{t:main} follows from Theorem \ref{t:bounds} by letting $T \to \infty$ and then taking $\delta \to 0$.
\end{proof}

\begin{remark}
\label{r:weak}
Comparing the conclusions of Theorem \ref{t:bounds} with Proposition \ref{p:counterlim} shows that $m = m^- < m^+ = m + \|\mu_{s.c.}\|$ for the examples in Proposition \ref{p:counterlim} (which are regular in the sense of \eqref{e:rss}), and so \eqref{e:mbounds} is tight.
\end{remark}

By combining with \emph{a priori} bounds on the capacity, one can deduce from Theorem~\ref{t:bounds} very general bounds on persistence, albeit at lower precision than in Theorem \ref{t:main}. For instance, we show in Lemma \ref{l:capbounds} below that
\begin{equation}
    \label{e:aprioricap}
  T^{\alpha^- + o(1) } \le  \capa(B(T)) \le T^{\alpha^+ + o(1)} , 
\end{equation}
and also that the doubling condition \eqref{e:doub} implies that $\capa(B(T)) \asymp 1 / \mu[B(1/ T) ]$. As immediate consequences of Theorem \ref{t:bounds} we deduce:

\begin{corollary}
Let $f$ be an SGF on $\R^d$ or $\Z^d$ with spectral singularity $\alpha \in (0,d)$ as per \eqref{e:blowup} and $\mu_{a.c.} \neq 0$. Then, as $T \to \infty$, $\theta_\mu(T) =  T^{\alpha + o(1) }$.
\end{corollary}

\begin{corollary}
Let $f$ be an SGF on $\R^d$ or $\Z^d$ with $\alpha^+ < d$, $\mu_{a.c.} \neq 0$, which satisfies the doubling conditions \eqref{e:doub}. Then, as $T \to \infty$, $\theta_\mu(T) \asymp  (\log T)  / \mu[B(1/T)] .$
\end{corollary}

\subsection{Heuristics and outline of the proof}
\label{ss:heu}
Let us assume the conditions of Theorem \ref{t:main} and outline the main ideas in the proof that
\[  \theta_\mu(T) \sim  m (d-\alpha)  \log T\,\capa_f(B(T)).\]
As discussed, $\per(B(T))$ typically occurs as a balance between raising the Gaussian coefficient $\zeta$ in the decomposition $f_\mu = \zeta h_T \oplus g$, where $h_T= h_f(B(T))$ is the equilibrium potential, and the suppression of the remainder $g$. Let us show that this strategy is optimised by raising $\zeta$ to level
\[ \ell_T \approx \sqrt{2m(d-\alpha) \log T} . \]
 Since $\capa_f(B(T)) = 1/\textrm{Var}[\zeta]$, and recalling the \emph{a priori} bound on $\capa_f(B(T))$ in \eqref{e:aprioricap}, standard Gaussian tail bounds (see Claim \ref{c:gtb}) applied to $\zeta$ give that
\begin{equation}
    \label{e:heu1}
  -\log \mathbb{P}[\zeta \ge \ell_T] \approx   \frac{\ell_T^2 \capa_f(B(T))}{2} \approx  \ell_T^2 T^{\alpha +o(1)}  .  
      \end{equation}
To address the cost of suppressing $g$ we observe that, since neither $\zeta h$ nor the spectral singularity contribute to the large deviations of the infimum, we may replace $g$ with $f_{\mu'}$ where $\mu' = \mu|_{B(\delta)^c}$ for $\delta > 0$ small. After a minor smoothing, we may also suppose that $K' = \mathcal{F}[\mu']$ is rapidly decaying. Putting these together, we obtain
\begin{equation}
    \label{e:heu3}
 \P \big[ \inf_{x \in B(T)} g(x)  \ge -\ell_T \big] \approx \P[ f_{\mu'}(0) \ge - \ell_T ]^{c_d T^d} \approx e^{ - e^{ - \ell_T^2 /(2m) } T^d } .  \end{equation}
Indeed the lower bound in \eqref{e:heu3} follows from the Gaussian correlation inequality, whereas we deduce the upper bound from the rapid decay of $K'$ using a general decoupling inequality for SGFs \cite{kls82}.

 \smallskip
Comparing \eqref{e:heu1} and \eqref{e:heu3} we wish to choose $\ell_T$ so that
\[\ell_T^2 T^{\alpha}  \asymp e^{ - \ell_T^2 /(2m) } T^d  \implies \ell_T \approx \sqrt{2m(d-\alpha) \log T} . \]
Moreover, if instead we choose  $\ell_{T  +\delta}=  \sqrt{(2m(d-\alpha) + \delta)  \log T}$ then \eqref{e:heu3} decays with $\delta$ much faster than \eqref{e:heu1} increases. Hence, by taking $\delta \to 0$, we are led to the estimate
\[  \theta_\mu(T) \approx \frac{\ell_T^2 \capa_f(B(T))}{2}  \approx m(d-\alpha) \log T\, \capa_f(B(T))  . \]
This also suggests that the conditioned field $\tilde{f}_T$ is close to $\ell_T h_T$, which is roughly the content of our entropic repulsion result.

\smallskip
It is not hard to turn this argument into a rigorous lower bound by estimating the probabilistic cost of this strategy. On the the other hand, for the upper bound we are faced with the harder challenge of bounding the cost of \textit{all} strategies to persist. To overcome this we show 
that, with sufficiently high probability (in an appropriate moderate deviation sense governed by the constant $m^-$ defined in \eqref{e:m-}), the infimum of $f_{\mu'}$ over any mesoscopic ball $B(T^{1-\ep})$ remains close to its expectation. On the other hand, due to its low frequency, $f_\mu - f_{\mu'} \stackrel{d}{=} f_{\mu|_{B(\delta)}}$ attains values close to its infimum over an entire such ball. By bounding the fluctuations of $\sup(f'_\mu - f'_{\mu'})$ with extremely high probability we deduce that, in order for persistence to occur, $f_\mu-f_{\mu'}$ must persist above a level close to $\ell_T$, an event whose probability we can compare to the probability that $\zeta h_T$ persists above a similar level, and hence to the lower bound.

Key to this analysis is our ability to bound the capacity and persistence probabilities of the approximating measure with those of the original measure by varying slightly the domain (see Lemmas~\ref{l:captrunc} and \ref{l:trunc}), which exploits the \textit{convexity} of the persistence event. 

To elevate this argument and obtain entropic repulsion, we must also rule out the possibility that $f_{\mu'}$ incurs macroscopic shape changes at the scale of $\ell_T$. For this we rely on the stability of the equilibrium potential to a small change in scale (see Corollary \ref{c:regh}), which requires the additional assumptions of doubling \eqref{e:doub} and origin dominance~\eqref{e:ds}.

\smallskip
Observe that these heuristics are still valid even in the absence of a spectral singularity, as exploited in \cite{ffn21,ffm25}. 
The key difference in that case is that the optimum level $\ell_T$ is bounded. As a result the balance between raising the coefficient $\zeta$ and suppressing the remainder $g$ becomes more delicate: the field is repelled by merely a constant, universality does not ensue, and a parallel of Theorem~\ref{t:er} could only be expected when considering persistence above a high level tending to infinity.
 
\subsection{Further remarks and open problems}
\label{ss:rem}
\subsubsection{Sufficient conditions for regular capacity growth} We have already noted that regular capacity growth as in~\eqref{e:capcon} holds for `regularly varying' fields (Proposition \ref{p:cond}). We provide further sufficient conditions:

\begin{proposition}
\label{p:suffcapcon}
Suppose \eqref{e:blowup} holds for $\alpha \in (0,\infty]$. Then the following are sufficient conditions for \eqref{e:capcon}:
\begin{enumerate}
\item $\alpha > d-1$ and $K \ge 0$; or
\item $\alpha \in (0,d-1)$, and $K$ is radial and eventually non-negative.
\end{enumerate}
\end{proposition}

\noindent The following example shows that \eqref{e:capcon} is not true in full generality: 
\begin{proposition}
\label{p:counterexample2}
For every $\alpha \in (0,1)$ there exists a stationary Gaussian process on $\R$ satisfying \eqref{e:blowup} such that \eqref{e:capcon} fails.
\end{proposition}

\subsubsection{A sufficient condition for origin dominance} We provide a convenient way to verify the origin dominance condition~\eqref{e:ds}: 

\begin{proposition}
\label{p:pos}
Assume \eqref{e:doub}, $\liminf\limits_{\delta \to 0} \mu[B(\delta)] \delta^{-d} > 0$, and $K$ is eventually positive. Then~\eqref{e:ds} holds.
\end{proposition}

\subsubsection{Generalisations}
\label{ss:gen2}
Our results can be generalised in several directions with minimal change to the proofs.

\smallskip
\noindent \textit{General levels.} One can consider persistence $f|_{B(T)} \ge u_T$ above levels $u_T$ satisfying
\[     \frac{u_T}{\sqrt{\log T}} \to u' \in \big(  -\sqrt{2m(d- \alpha^+)} , \infty \big)   \quad \text{as } T \to \infty. \]
In that case one replaces $m(d-\alpha)$ in Theorems \ref{t:main} and \ref{t:er} with $(\sqrt{2m (d-\alpha)} + u')^2    /2$,
and similarly for $m^+ (d-\alpha^-)$ and $m^-(d-\alpha^+)$ in Theorem \ref{t:bounds}. See also \cite{bdz95} where an analogous generalisation was obtained for the GFF.

\smallskip
\noindent \textit{General domains.} One can consider persistence $f|_{TD} \ge 0$ on rescaled domains $TD$, where $D \subset \R^d$ is a compact domain. In that case one replaces $B(T + s)$ with the $s$-neighbourhood of $TD$ in the sup-norm, and Theorems \ref{t:main}, \ref{t:er} and \ref{t:bounds} remain valid. However our proof of Theorems \ref{t: univ pro} and \ref{t: univ ent} and Proposition \ref{p:suffcapcon} are not valid in full generality (for the former, one could take $D$ Lipschitz, but in the latter we use that $D = B(1)$ is radial).

\subsubsection{Open problems}
\label{ss:oq}

We propose two sets of open problems concerning persistence of processes with spectral singularity.

\smallskip
\noindent \textit{Relaxing the conditions.} It would be appealing to relax the conditions in our main results:

\begin{question}
Do the conclusions of Theorems \ref{t:main} and \ref{t:er} hold assuming only \eqref{e:blowup} and \eqref{e:sing}? 
\end{question}
\begin{question}
Does  $\lp_\mu(T) \asymp  \capa_f\big(B(T)\big) \log T$ hold assuming only that $\alpha^+ < d$?
\end{question}

It would also be interesting to investigate the case of marginal spectral singularity, that is if  $\mu[B(\delta)] \delta^{-d} \to \infty$ and $\mu[B(\delta)] \delta^{-d} = \delta^{o(1)}$, so that $\alpha^+ = d$. Suppose for concreteness that $\mu$ has a density and $\mu[B(\delta)] \delta^{-d} \sim (\log 1/\delta)^{\beta}$, $\beta > 0$. Then we expect that $\capa_f(B(T)) \sim c'_{\beta,d} T^d (\log T)^{-\beta}$, and by the same heuristics as in Section~\ref{ss:heu}, persistence is governed by entropic repulsion to level $\ell_T = \sqrt{2 m \beta \log \log T }$, where $m = \|\mu\|$. This suggests that
\begin{equation}
    \label{e:alphad}
  \theta_\mu(T) \sim c'_{\beta,d}  m \beta T^d (\log T)^{-\beta} (\log  \log T)   .
\end{equation}

\begin{question}
Are the asymptotics in \eqref{e:alphad} correct?
\end{question}

For technical reasons our proof does not extend to marginal singularities, since $\ell_T \ll \sqrt{\log T}$ affects our ability to do spectral approximations.

\smallskip
\noindent \textit{Local entropic repulsion.} A natural extension of Theorem \ref{t:er} would be to investigate the mean and fluctuations of the conditioned field $\tilde{f}_T$ on local scales (i.e.\ studying the point-wise behaviour, rather than averaging over test functions). This has been thoroughly studied in the case of the GFF, using tools specific to that setting: 

\begin{theorem}[\cite{bdz95,dg99}]
For the GFF, as $T \to \infty$,
\begin{equation}
    \label{e:erlocal}
\sup_{x \in B(T)} \Big| \frac{\E[ \tilde{f}_T(x)  ] }{ \sqrt{2m(d-\alpha)\log T}}  - h_T(x)  \Big| \to 0 
\end{equation}
(recall that in this case $h_T|_{B(T)} = 1$ and $\alpha = d-2$). Moreover,
\begin{equation}
    \label{e:fluct}
\tilde{f}_T - \E[\tilde{f}_T] \ \Rightarrow  \   f  \quad \text{ in law.} 
\end{equation}
\end{theorem}

We expect \eqref{e:erlocal} and \eqref{e:fluct} to be true in much wider generality:

\begin{question}
Does \eqref{e:erlocal} hold under the conditions of Theorem~\ref{t:er}? Does it hold assuming only \eqref{e:blowup} and \eqref{e:sing}? Can the domain in \eqref{e:erlocal} be extended to the entire space?
\end{question}

\begin{question}
Does \eqref{e:fluct} hold assuming only that $\mu$ has a density with a singularity at the origin (i.e.\ $\mu[B(\delta)]\delta^{-d} \to \infty$)?
\end{question}

\subsection{Outline of the paper}
In Section \ref{s:prelim} we collect preliminary results which underpin the proof, including properties of the capacity and of the infimum of stationary Gaussian processes. In Sections \ref{s:proof} and \ref{s:er} we complete the proof of our main results, namely the persistence bounds in Theorem~\ref{t:bounds} and the entropic repulsion estimates in Theorem \ref{t:er} respectively. In Section \ref{s:rv} we prove our universality results, namely Theorems \ref{t: univ pro} and \ref{t: univ ent}. In Section \ref{s:ca} we study the regularity of the capacity, in particular proving Propositions \ref{p:suffcapcon} and \ref{p:counterexample2}. In Section \ref{s:sc} we study the effect of singular continuous measures, establishing Proposition \ref{p:counterlim}. Finally, Appendix~\ref{a:capacity} contains a proof of the equivalence of the dual definitions \eqref{e:cap2} and \eqref{e:cap1}  of the capacity, Appendix~\ref{a:taub} gives sufficient conditions for a regularly varying spectral singularity, and Appendix \ref{a:explicit} records explicit forms for the Riesz capacity and equilibrium potential of the unit ball.

\subsection*{Acknowledgements}
N.F.\ and O.F.\ are supported by the Israel Science Foundation grant 3541/24. S.M.\ is supported by the Australian Research Council (ARC) Future Fellowship FT240100396. A large part of this work was carried out while S.M.\ was a Research Fellow at the University of Melbourne, supported by the ARC Discovery Early Career Researcher Award DE200101467. S.M.\ also acknowledges the hospitality of the Statistical Laboratory, University of Cambridge, where part of this work was carried out.

%%%%%%%%%%%%
%%%%%%%%%
 \medskip

 \section{Preliminary results}
 \label{s:prelim}

 In this section we collect useful preliminary results. As before, $f$ is a continuous centred non-degenerate SGF on either $\R^d$ or $\Z^d$; we refer to the former as the `continuous case' and the latter as the `discrete case'. The covariance kernel of $f$ is denoted $K$ and its spectral measure $\mu$. A function, measure or set on $\R^d$ or $\Z^d$ is \textit{Hermitian} if it is symmetric with respect to reflection through the origin. A function, measure or set on $\R^d$ is \textit{radial} if it depends only on the distance from the origin. The density of a (signed) measure on $\R^d$ or the torus is always taken with respect to the Lebesgue measure, whereas a density on $\Z^d$ is taken respect to the counting measure on $\Z^d$.  For a (signed) measure $\nu$ we write $\|\nu\|_\infty$ and $\|\nu\|_{L^k}$ to denote respectively the sup-norm and $L^k$-norm of its density, assuming it exists.
 
\subsection{The Fourier transform}
\label{ss:fp}

\subsubsection{Standard facts}
We use the convention 
\[ \mathcal{F}[ f ](\lambda) = \int e^{-2 \pi i \langle \lambda,x \rangle} f(x) \,dx \]
to define the Fourier transform $\mathcal{F}$ of a function $f$ on $\R^d$ or $\Z^d$. The following are well-established facts which we state here without proof.
 
 \begin{proposition}
[Low frequencies]
 \label{c:2}
 There exists a $c_d > 0$ such that, for every Hermitian probability measure $\nu \in \mathcal{P}(B(c_d))$, $\mathcal{F}[\nu](x) \ge 1/2$ for all $x \in B(1)$.
 \end{proposition}

 \begin{proposition}
     [Radial measures]
 \label{c:3}
Let $0  < T \le T'$, and let $\nu \in \mathcal{P}(B(T') \setminus B(T))$ be radial. Then there exists a $c_d > 0$ such that
\[  |\mathcal{F}[\nu](x)| \le  c_d (1 + |x|T)^{-(d-1)/2} ,  \]
and more generally if $k$ is a multi-index,
\[  |\partial^k \mathcal{F}[\nu](x)| \le c_d (T')^{|k|} ( 1 + |x|T)^{-(d-1)/2}    . \]
 Moreover, 
\[  |\mathcal{F}[\upsilon_{B(T)}](x)| \le  c_{d} (1 + |x|T)^{-(d+1)/2},\]
where $\upsilon_{B(T)}$ denotes the uniform measure on $B(T)$.
\end{proposition}

\subsubsection{Simultaneous truncation and mollification}
The following standard construction of \textit{simultaneous}  mollification and truncation in the spatial and spectral domains respectively is fundamental to our approach.
For $a,b>0$ consider the smooth radial `bump' function 
\[ \xi'_{a,b}(x)=\xi'(x) =  \exp \left( -\frac{b}{\big(1-4|x|^2 \big)^a} \right) \id_{B(1/2)}(x),\]
and let
\[\xi_{a,b} = \xi = (\xi' * \xi') \varphi,\]
where 
$\varphi$ is the standard Gaussian density, and $*$ denotes convolution.
Some properties of this construction are summarised in the next proposition, which is standard to verify (see~\cite{joh15}).

\begin{proposition}
 [Simultaneous truncation and mollification]
\label{c:1}
For every $v \in (0,1)$ there exist 
$a,b>0$ such that the function $\xi=\xi_{a,b}$ and its Fourier transform $\zeta=\F[\xi]$ obey the  following:
\begin{itemize}
\item $(\xi,\zeta)$ are a Fourier transform pair, i.e.\ $\zeta = \mathcal{F}[\xi]$ and $\xi = \mathcal{F}[\zeta]$;
\item  $\xi , \zeta \in L^1(\R^d) \cap C^\infty(\R^d)$;
\item $\xi$ and $\zeta$ are radial, non-negative, and satisfy 
$\xi(0) = \zeta(0) = 1$ (consequently, $\xi(t)\le 1$ and $\zeta(t)\le 1$ for all $t$);
\item $\xi$ is supported on $B(1)$;
\item $\zeta(x) > 0$, and for every multi-index $k$ there exists $c_k > 0$ such that $| \partial^k \zeta(x)| \le c_k e^{-|x|^v}$.
\end{itemize}
\end{proposition}

For $v\in [0,1]$, we fix $(\xi,\zeta)$ to be a pair of functions whose existence is given by Proposition~\ref{c:1}. The \textit{spectral truncation function} with parameter $v$ is defined by
 \begin{equation}\label{eq:stf} 
 \phi_{s,v}(\lm) = \phi_s (\lm)= 
\zeta(2s \lm). \end{equation}
 Note that
   \begin{equation}
   \label{e:stfsupp}
   \textrm{supp}(\mathcal{F}[\phi_{s}]) \subseteq B(s/2) 
   \end{equation} 
    and, for every multi-index $k$, there exists $c_k > 0$ such that, for all $s \ge 1$,
 \begin{equation}
     \label{e:stfdecay}
 |\partial^k \phi_{s}(x)|  \le   c_k e^{-(s|x|)^v} .
  \end{equation} 

\subsubsection{Poisson duality}
Let $t > 0$ and $g : \R^d \to \R$ be a Schwartz function. The $t$-\textit{discretisation} of $g$ is defined by
\[ \text{Dis}_t(g)(\cdot) = \sum_{z \in t\Z^d} g(\cdot) \delta_z(\cdot) ,\]
where $\delta_x(\cdot)$ denotes a Dirac mass at $x$. The $t$-\textit{periodisation} of $g$ is defined by
\[ \text{Per}_t(g)(\cdot) = \sum_{z \in t\Z^d} g(\cdot + z) = \sum_{z \in t\Z^d} (g * \delta_z)(\cdot) .\]
Note that $\text{Dis}_t(g)$ is naturally identified with the function $g|_{t \Z^d}$ on $t\Z^d$, and $\text{Per}_t(g)$ is naturally identified with its restriction to the torus $t \mathbb{T}^d$. In particular, when discussing the spectral measure of discrete fields, we will often identify a function on the torus $\mathbb{T}^d$ with its $1$-periodisation, and vice versa.

The Fourier duality between these objects is captured by the Poisson summation formula: 
\begin{proposition}[Poisson summation formula]
\label{c:4}
For every $t > 0$ and Schwartz function $g : \R^d \to \R$,
\[ \mathcal{F}[\text{Dis}_{1/t}(g)] = t\text{Per}_{t}(\mathcal{F}[g]) .  \]
\end{proposition}

In Section \ref{s:ca} we use this to control the difference between the spectra of discretisations of the same function $g$ in two scales:

\begin{corollary}
\label{cc:4}
Suppose $d=1$. Then for every $T > 0$ and Schwartz function $g : \R \to \R$,
\[
\mathcal{F}[\text{Dis}_{1/T}(g)] - \mathcal{F}[\text{Dis}_{2/T}(g)]= (T/2) \sum_{k \in \Z} (-1)^k (\mathcal{F}[g] * \delta_{k T/2})(x) .\]
\end{corollary}
\begin{proof}
The right-hand side can be written as
\[T \text{Per}_{T}(\mathcal{F}[g]) - (T/2) \text{Per}_{T/2}(\mathcal{F}[g])  , \]
which is equal to the left-hand side by Proposition~\ref{c:4}.
\end{proof}

In order to work with discrete fields we also define the \textit{periodisation} of the spectral truncation function in  \eqref{eq:stf}
\[ \tilde{\phi}_{s} = \textrm{Per}_1(\phi_s) = \sum_{z \in \Z^d} \zeta((2s) ( \cdot + z))  . \]
Similarly to \eqref{e:stfsupp}--\eqref{e:stfdecay},
 \[ \textrm{supp}(\mathcal{F}[\tilde{\phi}_{s}]) \subseteq B(s/2) \cap  \Z^d  \qquad \text{and} \qquad  |\partial^k \tilde{\phi}_{s}(x) |  \le   c e^{-(|x|s)^v} . \]

\subsection{Energy and capacity}
\label{ss:cap}
In this section we develop basic properties of energy and capacity. Let us start by setting up relevant definitions. 

Recall that $\mathcal{P}(D)$ denotes the set of probability measures on a compact domain $D$, and $\mathcal{M}$ is the set of signed measures. For $\nu,\eta \in \mathcal{M}$ define the energy functional
\[  E[ \nu, \eta ] = \iint K(x-y) d\nu(x) d\eta(y) \ , \quad E[\nu] = E[\nu,\nu] . \]
Denote by $\mathcal{M}_{<\infty} \subset \mathcal{M}$ the signed measures of finite energy $E[\nu] < \infty$; the functional $E$ defines an inner product on the space $\mathcal{M}_{<\infty}$. The energy can also be written as
\begin{equation}
\label{e:fenergy}
  E[ \nu, \eta ] = \int \mathcal{F}[\nu] \overline{\mathcal{F}[\eta]} d\mu \ , \quad E[\nu] = \int |\mathcal{F}[\nu]|^2 d\mu  . 
\end{equation}
Occasionally we write $E_\mu$ to emphasise the dependence of the energy on the spectral measure.

 Let $\mathcal{M}_{>0} \subset \mathcal{M}$ be the signed measures satisfying $E[\nu] > 0$. For each $\nu \in \mathcal{M}_{>0}$ define the \textit{potential}
\[ h_\nu(x) = \frac{ (K * \nu)(x)}{E[\nu]} = \frac{ \int K(x-y) d\nu(y)}{E[\nu]} . \]
Elsewhere in the literature it is common to omit the normalisation $E[\nu]^{-1}$, but it will be convenient for our purposes. Note that for $\nu \in \mathcal{M}_{>0} \cap \mathcal{M}_{<\infty}$ and $\eta \in \mathcal{M}_{<\infty}$ we have \begin{equation}
    \label{e:heta}
 \langle h_\nu, \eta \rangle =  \frac{ \int (K * \nu) d\eta }{ E[\nu] } =  \frac{ E[\nu,\eta] }{ E[\nu] }.   
 \end{equation} 
The reproducing kernel Hilbert space (RKHS) associated to $(f,K,\mu)$ is the set
\[ H = \left\{ h(\cdot) = \int K(\cdot-y) d\nu(y) :\: \nu \in \mathcal{M}_{<\infty}  \right\} \]
equipped with the inner product inherited from $E$, i.e.\
\[ \Big\langle \int K(\cdot-y) d\nu(y) ,  \int K(\cdot-y) d\eta(y) \Big\rangle_H  = E[\nu,\eta] . \] 
This coincides with the previous definition given immediately above \eqref{e:cap1}.

Denote by $\nu_D$ an arbitrarily chosen \textit{equilibrium measure} as defined following \eqref{e:cap1}. We remark that the  equilibrium measure is unique if $K$ is strictly positive definite in the sense that $\mathcal{M} \setminus \mathcal{M}_{>0}$ only contains the zero measure, although this fact will not be of use to us. Assuming $E[\nu_D] > 0$,  we refer to the corresponding potential $h_D = h_{\nu_D}$ as the \textit{equilibrium potential}, that is, the optimiser in \eqref{e:cap2}, which is unique and does not depend on the choice of $\nu_D$ (see Proposition~\ref{p:basiccap}). By definition, the equilibrium potential  satisfies $h_D \ge 1$ on $D$ and
\[ \|h_D\|_H^2 = 1/E[\nu_D] = \capa_f(D) .\]

\smallskip
Next, we record some facts about the energy, capacity, equilibrium measure, and equilibrium potential. These are mostly classical, except for the introduction of an `isotropy' property (Claim \ref{c:iso}) and an extension of the classical subadditivity property (see \eqref{e:saext} and \eqref{e:saext2}).

\begin{claim}[Convexity]
\label{c:convex}
The mapping $\nu \mapsto E[\nu]$ is convex on $\mathcal{M}_{<\infty}$ and the mapping \linebreak$h \mapsto \|h\|_H^2$ is strictly convex on $H$.
\end{claim}
\begin{proof}
For every $\nu,\eta \in \mathcal{M}_{<\infty}$ and $t \in [0,1]$ one may check that 
\begin{equation}
    \label{e:convex}
 E(t  \nu + (1-t) \eta) = t E(\nu) + (1-t) E(\eta) - t(1-t) E(\nu - \eta),
 \end{equation}
from which convexity of $E$ follows. The proof that $\|\cdot\|_H$ is strictly convex is similar, using that $H$ is a Hilbert space (so that $\|h\|_H=0$ implies $h=0$).
\end{proof}

\begin{claim}[Isotropy]
\label{c:iso}
Suppose that $f$ is isotropic and $D$ is a radial domain (e.g.\ $B(T) \setminus B(S)$ for some $0 \le S < T$). Then there exists an equilibrium measure which is radial.
\end{claim}
\begin{proof} 
Let $\mathcal{R}$ be the set of rotations about the origin equipped with the Haar measure. Let~$\nu$ be an equilibrium measure. For any rotation $R\in \mathcal R$, we have 
\begin{equation*}
E[R \nu] = \int |\mathcal{F}[R\nu]|^2 d\mu 
 = \int |\mathcal{F}[R\nu]|^2 d(R\mu) 
 =\int |\mathcal{F}[\nu]|^2 d(\mu)  = E[\nu],
\end{equation*}
where we used that $R\mu=\mu$.
Let $\bar \nu$ be the average of $\nu$ over $\mathcal{R}$. Clearly $\bar{\nu} \in \mathcal{P}(D)$ is radial. As we have shown that all rotations of $\nu$ have the same energy, convexity (Claim~\ref{c:convex}) yields
$E[\bar{\nu}] \le E[\nu]$; by minimality, this must be an equality. Thus
$\bar{\nu}$ is a radial equilibrium measure.
\end{proof}

\begin{claim}[Monotonicity]
\label{c:mon}
For every $D \subseteq D'$, $\nu \in \mathcal{M}_{<\infty}$, and spectral measure $\mu' \le \mu$,
\[ \capa_{\mu'}(D') \ge \capa_\mu(D) \qquad \text{and} \qquad E_{\mu'}[\nu] \le E_\mu[\nu] .\]
\end{claim}
\begin{proof}
These are clear from the definitions \eqref{e:cap2}, \eqref{e:cap1}, and \eqref{e:fenergy}.
\end{proof}

In Section \ref{s:ca} we make use of the subadditivity of the capacity:

\begin{claim}[Subadditivity]
\label{c:sa}
If $K \ge 0$, then for  all compact domains $D_1,D_2$,
\[ \capa_f(D_1 \cup D_2) \le \capa_f(D_1) + \capa_f(D_2)  .\]
More generally, for every $K$ and compact domains $D_1,D_2$
\begin{equation}
    \label{e:saext}
 \capa_f(D_1 \cup D_2) \le \frac{1}{ \frac{1}{\capa_f(D_1) + \capa_f(D_2)} - 2 \int_{D_1} \int_{D_2} (-K(x-y))_+ \, d\nu(x) d\nu(y)} 
 \end{equation}
where $\nu = \nu_{D_1 \cup D_2}$ is an equilibrium measure for $D_1 \cup D_2$. In particular, if $f$ is isotropic and $K$ is eventually non-negative, then there exists a $c > 0$ such that, for all $0 < S < T$,
\begin{equation}
    \label{e:saext2}
 \capa_f(B(T)) \le \frac{1}{ \frac{1}{\capa_f(B(S)) + \capa_f(B(T)\setminus B(S))} -  c S^{-(d-1)} } .
 \end{equation}
\end{claim}
\begin{proof}
Define $\nu_1 = \nu|_{D_1}$ and $\nu_2 = \nu|_{D_2 \setminus D_1}$. Then
\begin{align*}
  1/\capa_f(D_1\cup D_2) =  E[\nu] & = E[\nu_1] + E[\nu_2] + 2 \int_{D_1} \int_{D_2\setminus D_1} K(x-y) \, d\nu_1(x) d\nu_2(y) \\
    & \ge E[\nu_1] + E[\nu_2] - 2 \int_{D_1} \int_{D_2} (-K(x-y))_+ \,  d\nu(x) d\nu(y)  .
     \end{align*}
    Moreover,
  \[    E[\nu_1] + E[\nu_2]  \ge \frac{ (|\nu_1| + |\nu_2|)^2}{ |\nu_1|^2 / E[\nu_1] + |\nu_2|^2 / E[\nu_2] } \ge \frac{1}{ \capa_f(D_1) + \capa_f(D_2) }   , \]
    where the first inequality is by Cauchy-Schwarz, and the second inequality used $|\nu_1| + |\nu_2| = 1$ and the definition \eqref{e:cap2} of the capacity. Combining gives \eqref{e:saext}.
    
    For \eqref{e:saext2}, let $D_1 = B(S)$, $D_2 = B(T) \setminus B(S)$, and let $r > 0$ be sufficiently large so that $K(x) \ge 0$ for $|x| \ge r$. Recall that we may assume that $\nu$ is radial by Claim~\ref{c:iso}. Then observe that, for all $x \in B(S)$, 
    \[ \int_{B(T) \setminus B(S)} (-K(x-y)) d\nu(y) \le \|K\|_\infty  \int_{\substack{y \in B(T) \setminus B(S),\\ |y-x| \le r}} d\nu(y) \le \|K\|_\infty  c_d r^d S^{-(d-1)}, \]
    where the final inequality used that $\nu$ is radial.
\end{proof}

The following lemmas are tailored to our approach, and involve the spectral truncation function $\phi_s$ from \eqref{eq:stf}. 

\smallskip
First we state an important smoothing property of the capacity. Recall that the RKHS $H$ consists of functions of the form $h = K * \mathcal{F}[g]$, or equivalently $h = \mathcal{F}[g \mu]$, for complex Hermitian $g \in L^2(\mathbb{R}^d)$ satisfying $\|h\|_H^2 = \int |g|^2  d\mu < \infty$. 

\begin{lemma}[Smoothing for the capacity]
\label{l:captrunc}
For every $0 < s < T$,
\[ \capa_\mu(B(T)) \ge \capa_{\phi_s^2 \mu}(B(T-s)) \]
and
\[ \capa_\mu(B(T) \cap \Z^d) \ge \capa_{\tilde{\phi}_s^2 \mu}(B(T-s) \cap \Z^d) .\]
\end{lemma}
\begin{proof}
Let $h$ be an equilibrium potential as defined following \eqref{e:cap2} and consider the function $h' = h * \mathcal{F}[\phi_s]$. By construction, the norm of $h'$ in the RKHS $H_{\phi_s^2 \mu}$ associated to $\phi_s^2 \mu$ is equal to the norm of $h$ in the RKHS $H_\mu$ associated to $\mu$. Moreover, since $h \ge 1$ on $B(T)$, and since $\phi_s$ satisfies $\mathcal{F}[\phi_s] \ge 0$, $\| \mathcal{F}[\phi_s]\|_{L^1(\R^d)} = 1$, and $\mathcal{F}[\phi_s]$ is supported on $B(s/2)$, we have $h' \ge 1$ on $B(T-s)$. Taking $h'$ as a candidate in the definition \eqref{e:cap2} for $\capa_{\phi_s^2 \mu}(T-s)$, we have
\[ \capa_{\phi_s^2 \mu}(T-s) \le \|h'\|^2_{H_{\phi_s^2 \mu}} =  \|h\|^2_{H_{\mu}} = \capa_\mu(B(T)) . \]
The proof of the second statement is identical after restricting to the lattice $ \Z^d$.
\end{proof}

Next, we state \emph{a priori} bounds on the capacity in terms of the spectral mass near the origin:  

\begin{lemma}[Bounds on the capacity]
\label{l:capbounds}
For every $\gamma >  \gamma' > 0 $ there exist $c_d > 0$ and $T_0= T_0(d,\gamma, \gamma') > 0$ such that, for all $T \ge T_0$,
\[ \frac{1}{\mu[ B( (\log T)^{1+\gamma} / T ) ] + \|\mu \| e^{- (\log T)^{1+\gamma'}}}  \le \capa(B(T))  \le  \frac{4}{  \mu[B(c_d / T) ] } .\]
Moreover, suppose that the doubling condition \eqref{e:doub} holds, and let $c_0 > 0$ be the constant in this condition. Then there exists a constant $c = c(d,c_0) > 0$ such that, for all $T \ge 1$,
\begin{equation}
    \label{e:capcomp}
 1/c \le \capa(B(T)) \mu[B(1/ T) ]  \le c .
 \end{equation}

\end{lemma}
\begin{proof}
Fix $v = (1+\gamma')/(1+\gamma) \in (0,1)$, and recall $\phi_{s,v}=\phi_s$ from  \eqref{eq:stf}. Let $c_d > 0$ be the constant in Claim~\ref{c:2}. We first claim that, for every $T > 0$, 
\begin{equation}\label{eq:cap bnd1} 
    \frac{1}{\int  \phi_T^2(\lambda) d\mu(\lambda) }  \le \capa(B(T))  \le  \frac{4}{  \mu[B(c_d / T)] } .
\end{equation}
Indeed, the right-hand side of \eqref{eq:cap bnd1} is obtained as $\|h\|^2_H$ for \[ h = \mathcal{F}[ g d\mu] \, , \quad g(x) = 2 \cdot\id_{B(c_d/T)}(x) / \mu[B(c_d / T)], \]
which, by Proposition~\ref{c:2}, satisfies $h \ge 1 $ on $B(T)$. From the definition of capacity in \eqref{e:cap2} we deduce that 
\[  \capa(B(T))  \le \int g^2 d\mu  = 4 / \mu[B(c_d / T)]. \]
On the other hand, the left-hand side of \eqref{eq:cap bnd1} is obtained as
$\left(\int_{\R^d} \int_{\R^d} K(x-y) d\nu(x) d\nu(y)\right)^{-1}$
taking $\nu \in \mathcal{P}(B(T))$ with  density $T^{-d}\mathcal{F}[\phi_T]$, which gives,  by \eqref{e:cap1},
\[  \capa(B(T))  \ge \frac{ T^{2d} }{\int_{\R^d \times \R^d} K(x-y) \, \mathcal{F}[\phi_T](x) \mathcal{F}[\phi_T](y) \, dx dy  } = \frac{1}{\int  \phi^2_T(\lambda) d\mu(\lambda) },\]
concluding the proof of 
\eqref{eq:cap bnd1}.

Next, by \eqref{e:stfdecay}
we have for $T \ge T_0$,
\begin{equation}
\label{e:phibound}
\phi_T(T (\log T)^{1+\gamma} ) \le e^{- (\log T)^{(1+\gamma ) v}} = e^{- (\log T)^{1+\gamma'}}.
\end{equation}
Since $\phi_T \in [0,1]$ we deduce that 
for $T \ge T_0(d,\gamma,\gamma')$ we have
\[   0 \le  \phi_T^2(x) \le  \id_{ B( (\log T)^{1+\gamma} / T )   } (x)  + e^{- (\log T)^{1+\gamma'}}.\]
The first statement of the proposition follows.

For the second statement we use the fact that, by \eqref{e:stfdecay} and \eqref{e:doub}, we have
\begin{align}
    \label{e:dc}
 \int  \phi^2_T(x) d\mu(x)  &  \le  \mu[B(1/T)] + c \sum_{k=0}^\infty \mu[B(2^{k+1}/T)]   e^{-(2^k)^v } \\
\nonumber &  \le c  \mu[B(1/T)]  \Big( 1 +  \sum_{k=0}^\infty c_0^{k+1}   e^{-(2^k)^v } \Big)   \\
\nonumber & \le c' \mu[B(1/T)] 
 \end{align}
 and
 \[ \mu[B(c_d/T)] \ge c' \mu[B(1/T)] ,\]
  where $c' > 0$ depends only on $c_0$, $c_d$, and the choice of $v$.
\end{proof}

\begin{corollary}
\label{c:capbounds}
Suppose $\alpha^+ \in [0,\infty)$. Then as $T \to \infty$,
 \[ T^{\alpha^- + o(1) } \le  \capa(B(T)) \le T^{\alpha^+ + o(1)} . \]
\end{corollary}

Finally we establish \emph{a priori} bounds on the energy of signed measures. This is the only place in the paper where the origin dominance condition \eqref{e:ds} is used.

\begin{lemma}[Energy bounds]
\label{l:enbound}
 There exists an absolute constant $c > 0$ such that, for every $T \ge 1$ and every $\nu \in \mathcal{M}$ supported on $B(T)$,
   \[  E[\nu]  \le c  T^d \|\nu\|_{L^2}^2 \|\mu\|   . \] 
  If additionally \eqref{e:doub} and \eqref{e:ds} hold, then there exists a constant $c > 0$ such that, for every $T \ge 1$ and every measure $\nu \in \mathcal{M}$ supported on $B(T)$,
   \[  E[\nu]  \le  \frac{c  T^d \|\nu\|_{L^2}^2   }{\capa_f(B(T))}  . \] 
   \end{lemma}
   
  \begin{proof}
Let $v \in (0,1)$ be arbitrary, and fix a constant $c_1 > 0$ and a function $\psi$ such that $|\psi(x)| \le c_1 e^{-|x|^v}$ for all $x$, and $\mathcal{F}[\psi]\big|_{B(2)}\equiv 1$. (For instance, one can take $\psi = \mathcal{F}[\id_{B(4)}]  \zeta$, where $\zeta$ is the function from Proposition~\ref{c:1}.) Since $\nu$ is supported on $B(T)$, we have
\begin{align*}
E[\nu] & = \int \int K(x-y) \mathcal{F}[\psi]( (x-y) / T ) \, d\nu(x) d\nu(y)  \\ 
& = \int |\mathcal{F}[\nu]|^2(\lambda) \, d(\mu  * \psi(\cdot/T) )(\lambda) \\
& \le T^d \|\nu\|_{L^2}^2  \sup_s   \int | \psi((x-s)/T) |
d\mu(x)  .
\end{align*}
In general the above is bounded by $T^d \|\nu\|_{L^2}^2 \| \psi\|_{\infty} \|\mu\| $ which concludes the first part of the lemma. Assuming \eqref{e:doub} and \eqref{e:ds},
we can further improve this via
\begin{align*}
   \int  |\psi((x-s)/T) | d\mu(x)  &   \le  c_2 \mu[s + B(1/T)] + c_2 \sum_{k=0}^\infty \mu[s + B(2^{k+1}/T)] e^{-(2^k)^v}  \\
   &  \le  c_2 \mu[B(1/T)] + c_2 \sum_{k=0}^\infty \mu[B(2^{k+1}/T)] e^{-(2^k)^v} \\
   &  \le c_3 \mu[ B(1/T)] 
   \end{align*}
   where \eqref{e:ds} was used in the second step, and \eqref{e:doub} was used in the third step,  much like it was used to obtain~\eqref{e:dc}. The second part of the lemma now follows from \eqref{e:capcomp}.
\end{proof}

We deduce from Lemma \ref{l:enbound} (and its proof) energy bounds on the rescaled measures $\eta_T := T^{-d} \eta(\cdot/T)$ that appear in our entropic repulsion result:

\begin{corollary}
\label{c:enbound}
Suppose \eqref{e:doub} and \eqref{e:ds} hold, and let $\eta \in \mathcal{M}$ be compactly supported such that $\|\eta\|_{L^2} < \infty$.  
\begin{enumerate}[{\bf (1)}]
    \item As $T \to \infty$,
\[ E[ \eta_T ] = O\Big( \frac{1}{\capa_f(B(T))} \Big) . \]
\item Fix $\delta > 0$, and suppose that $\alpha^+ < \infty$. Then as $T \to \infty$,
\[  \widetilde{E}[ \eta_T ] = o\Big( \frac{1}{\capa_f(B(T))} \Big)  ,\]
where $\widetilde{E}[\cdot]$ denotes energy with respect to the spectral measure $\tilde{\mu} := \mu_{B(\delta)^c}$.
\end{enumerate}
\end{corollary}
\begin{proof}
\textbf{(1)} Let $r \ge 1$ be such that $\eta$ is supported on $B(r)$. By Lemma \ref{l:enbound}, and using that $\|\eta_T\|_{L^2}^2 = T^{-d} \|\eta\|_{L^2}^2$, we have
\[ E[ \eta_T ] \le \frac{ c (rT)^d T^{-d} \| \eta \|^2_{L^2} }{ \capa_f(B(rT)) } \le \frac{ c r^d \| \eta \|^2_{L^2} }{ \capa_f(B(T)) } ,\]
where the second inequality is by the monotonicity of the capacity (Claim \ref{c:mon}).

\textbf{(2)} Let $\psi$ be as in the proof of Lemma \ref{l:enbound}. As in that proof 
\[ \widetilde{E}[\eta_T]  = \int |\mathcal{F}[\eta_T]|^2(\lambda) \, d(\tilde{\mu}  * \psi(\cdot/T) )(\lambda) .\]
We split the domain of this integral into $\lambda \in B(\delta/2)^c$ and $\lambda \in B(\delta/2)$. For the first we have 
\begin{align*}
    \int_{\lambda \in B(\delta/2)^c} |\mathcal{F}[\eta_T]|^2(\lambda) \, d(\tilde{\mu}  * \psi(\cdot/T) )(\lambda)  &  \le T^d \||\mathcal{F}[\eta_T]| \|_{L^2(B(\delta/2)^c)}^2  \sup_s   \int | \psi((x-s)/T) | d\mu(x) .
\end{align*}
Since $ \||\mathcal{F}[\eta]| \|_{L^2} = \| \eta\|_{L^2} < \infty$, we have 
\[\||\mathcal{F}[\eta_T]| \|_{L^2(B(\delta/2)^c)}^2 = T^{-d}  \||\mathcal{F}[\eta]| \|_{L^2(B(c_\delta T)^c)}^2 = o(T^{-d}). \]
Since also, as in the proof of Lemma \ref{l:enbound},  
\[   \int  |\psi((x-s)/T) | d\mu(x)     \le \frac{c}{\capa_f(B(T)}, \] 
we see that
\[  \int_{\lambda \in B(\delta/2)^c} |\mathcal{F}[\eta_T]|^2(\lambda) \, d(\tilde{\mu}  * \psi(\cdot/T) )(\lambda)   = o\Big( \frac{1}{\capa_f(B(T))} \Big) . \]
For the second domain we have instead
\begin{align*}
    \int_{\lambda \in B(\delta/2)} |\mathcal{F}[\eta_T]|^2(\lambda) \, d(\tilde{\mu}  * \psi(\cdot/T) )(\lambda)  &  \le T^d \||\mathcal{F}[\eta_T]| \|_{L^2}^2  \sup_{s \in B(\delta/2)}   \int | \psi((x-s)/T) | d\tilde{\mu}(x) \\
    & = \|\eta\|_{L^2}  \sup_{s \in B(\delta/2)}   \int | \psi((x-s)/T) | d\tilde{\mu}(x).
\end{align*}
Since $|\psi(x)| \le c_1 e^{-|x|^{\nu}}$, and $\tilde{\mu}$ is supported outside $B(\delta)^c$,
\[  \sup_{s \in B(\delta/2)}   \int | \psi((x-s)/T) | d\tilde{\mu}(x) \le \|\tilde{\mu}\| c_1 e^{-|\delta T/2|^{\nu}}= o\Big( \frac{1}{\capa_f(B(T))} \Big)\]
where the final step is by Corollary \ref{c:capbounds}, using that $\alpha^+ < \infty$.
\end{proof}

\subsection{The infimum of SGFs}
In this section we study properties of the infimum of general SGFs, discrete or continuous, and related results. We first state standard Gaussian tail bounds.

\begin{claim}[Gaussian tail bounds]
\label{c:gtb}
Let $Z$ be a standard Gaussian random variable. Then
  \[ \frac{1}{\sqrt{2 \pi}} \left(\frac 1 x - \frac 1 {x^3} \right) e^{-x^2/2} \le  \P[ Z \ge x ] \le \frac12 e^{-x^2/2} \, , \quad x > 0 . \]
\end{claim}

The next three results are about the order of the supremum (or infimum) of a SGF on growing domains.

\begin{lemma}[Spectral monotonicity]
\label{lem: mon of Esup}
For spectral measures satisfying $\mu'\le \mu$, and every domain~$D$,
\[
\E \sup_D f_{\mu'} \le \E \sup_D f_{\mu}.
\]
\end{lemma}

\begin{proof}
By the spectral decomposition, $f_\mu \overset{d}= f_{\mu'} \oplus g$, where $f_{\mu'}$ and $g=f_{\mu-\mu'}$ are independent. Let $X = \arg\max_D f_{\mu'}$ be the index of the supremum of $f_{\mu'}$ in~$D$ (breaking ties arbitrarily). Then
\[
\sup_D f_{\mu} \ge f_{\mu}(X) = f_{\mu'}(X) +g(X)
 = \sup_D f_{\mu'} + g(X).
\]
Since $X$ is independent of $g$, and $\E g(x) = 0$ for every $x\in D$, we see that $\E g(X) = 0$. The result follows by taking expectations.
\end{proof}

 \begin{lemma}[Order of the infimum]
 \label{l:infconv}
  Let $\mu$ be a spectral measure in $\R^d$ and let $\delta > 0$.  
 \begin{enumerate}[\rm (1)]
 
 \item There exists $T_0=T_0(\mu,\delta)>0$ such that, for all $T>T_0$,
 \[ \inf_{\mu' \le \mu} \E \inf_{B(T)} f_{\mu'} \ge -\sqrt{(2d\|\mu_{a.c.} + \mu_{s.c.}\|-\delta)\log T} . \]

\item As $T \to \infty$
\begin{equation*}
     \inf_{\mu' \le \mu} \P \Big[  \frac{\inf_{x \in B(T)} f_{\mu'}(x) }{ \sqrt{\log T}} \ge - \sqrt{2d\|\mu_{a.c.} + \mu_{s.c.}\|- \delta } \Big] \to 1. 
     \end{equation*}
\end{enumerate}
 \end{lemma}

\begin{proof}
For $d=1$, and if $f_{\mu'}$ is replaced by $f_{\mu}$, these statements are proven in \cite{mar72,mar74}. The adaptation to higher dimensions is straightforward. The addition of the infimum follows from Lemma~\ref{lem: mon of Esup} and Markov's inequality.
\end{proof}

\begin{lemma}\label{l:DudBTIS}
 Let $\mu$ be a spectral measure in $\R^d$ such that $\int |\lm|^2 d\mu(\lm)\le q<\infty$. Then there exist $T_0(q,d)>0$ and $c(q,d) > 0$ such that, for every $T>T_0$,
\[\P\Big(\sup_{B(T)} f_\mu > 2c\sqrt{ \|\mu\| \log T}\Big) \le T^{-\frac{c^2}{2}}.\]
\end{lemma}

\begin{proof}
By \cite[Lemma 3.12]{ffn21} there are $T_0=T_0(q,d) > 0$ and $c=c(q,d) > 0$ such that, for all $T>T_0$,
\[ \E \sup_{B(T)} f_\mu \le c\sqrt{\|\mu\|\log T},\]
and the statement follows from the Borell-TIS inequality.
\end{proof}

Our next result gives the persistence analogue of the smoothing property of the capacity given in Lemma \ref{l:captrunc}. To obtain it we exploit the convexity of the persistence event. With a view towards entropic repulsion, we also smooth exceedence events for averages of the process. The result is inspired by \cite[Lemma 1.2]{ffm25}, which proved this statement in dimension $d=1$ (without the exceedence event). Recall the spectral truncation function $\phi_s=\phi_{s,v}$ from \eqref{eq:stf}. 

\begin{lemma}[Smoothing for the persistence]
\label{l:trunc} 
Let $\mu = \mu_1 + \mu_2$, let $f'$ be an SGF with spectral measure $\mu_1 + \phi_s^2 \mu_2$. Then for every $0 < s <T$ and $\ell \in \R$,
\[  \P \Big[ \inf_{x \in B(T)} f(x) \ge \ell \Big]  \le  \P \Big[ \inf_{x \in B(T-s)} f'(x) \ge \ell  \Big] .\]
 Moreover, if $\eta \in \mathcal{M}$, then for every $0 < s <T$, $\ell,\ell' \in \R$, and $\delta > 0$,
\begin{align*}
& \P \Big[\inf_{x \in B(T)} f(x) \ge \ell , \langle 
 f,\eta \rangle \ge \ell' \Big] \\
& \qquad \le  \P \Big[ \inf_{x \in B(T-s)} f'(x) \ge \ell , \langle f',\eta \rangle \ge \ell' -\delta \big]  + \P \big[ \sup_{u \in B(s/2)} | \langle f, \eta - \eta_u \rangle | \ge \delta  \Big],
\end{align*}
where $\eta_u(\cdot) = \eta(\cdot - u)$ is the translation of $\eta$. The same result holds in the discrete case after replacing $\phi_s$ with $\tilde{\phi}_s$.
\end{lemma}

\begin{proof}
It suffices to prove the second statement, since the first is obtained by taking $\delta \to \infty$.

We follow the proof of \cite[Lemma 1.2]{ffm25}. Let $f_{\mu_1}$ and $f_{\mu_2}$ be independent SGFs with spectral measures $\mu_1$ and $\mu_2$ respectively. Abbreviate $H = \mathcal{F}[\phi_s]$, which satisfies $H \ge 0$, $\int H = 1$, and $\textrm{supp}(H) \subseteq B(s/2)$. Notice that $\phi_s^2 \mu_2$ is the spectral measure of the SGF $H \star f_{\mu_2}$. Let $v(x)$ be a continuous function. Since $\{ \inf_{x \in B(T-s)} f_{\mu_1}(x)  \ge \ell , \langle f_{\mu_1} , \eta \rangle \ge \ell' -\delta \} $ is a convex event, the log-concavity of the Gaussian measure implies that 
 \begin{align}
 \label{e:lc1}
 &  \log \P \Big[ \inf_{x \in B(T-s)} f_{\mu_1}(x) +  (H \star v)(x)   \ge \ell , \langle f_{\mu_1} + H \star v , \eta \rangle  \ge \ell' -\delta \Big]  \\
\nonumber & \qquad  \qquad  \ge  \int H(u) \log \P \Big[ \inf_{x \in B(T-s)} f_{\mu_1}(x) +  v(x-u)  \ge \ell , \langle f_{\mu_1} + v(\cdot-u) , \eta \rangle  \ge \ell' -\delta \Big] du   .
 \end{align}
By stationarity of $f_{\mu_1}$, and since $\textrm{supp}(H) \subseteq B(s/2)$, the right-hand side of \eqref{e:lc1} equals
\begin{align*}
& \int H(u) \log \P \Big[ \inf_{x \in u + B(T-s)} f_{\mu_1}(x) +  v  \ge \ell , \langle f_{\mu_1} + v(x) , \eta_u \rangle  \ge \ell' -\delta \Big] du  \\
& \qquad \ge \int H(u) \log  \P \Big[  \inf_{x \in B(T)} f_{\mu_1}(x) +  v(x)  \ge \ell , \inf_{u \in B(s/2)} \langle f_{\mu_1} + v , \eta_u \rangle   \ge \ell' -\delta \Big] du   \\
& \qquad = \log  \P \Big[  \inf_{x \in B(T)} f_{\mu_1}(x) +  v(x)  \ge \ell , \inf_{u \in B(s/2)} \langle f_{\mu_1} + v , \eta_u \rangle   \ge \ell' -\delta \Big] 
\end{align*}
where the final equality used that $\int H = 1$. Averaging over $f_{\mu_2} \equiv v$, and recalling that $f_{\mu_2}$ is independent of $f_{\mu_1}$ and that $f \stackrel{d}{=} f_{\mu_1} + f_{\mu_2}$ and $f' \stackrel{d}{=} f_{\mu_1} + H \star f_{\mu_2}$, this yields that
\[ \P \Big[ \inf_{x \in B(T-s)} f'(x) \ge \ell , \langle f',\eta \rangle \ge \ell' -\delta  \Big] \ge \P \Big[ \inf_{x \in B(T)} f(x) \ge \ell , \inf_{u \in B(s/2)} \langle f,\eta_u \rangle \ge \ell' -\delta \Big]    \]
which is bounded below by 
\[ \P \Big[ \inf_{x \in B(T)} f(x) \ge \ell , \langle f,\eta \rangle \ge \ell'  \Big] -  \P \Big[ \sup_{u \in B(s/2)} | \langle f, \eta - \eta_u \rangle | \ge \delta \Big] \]
completing the proof.
\end{proof}

We proceed with a study of moderate deviations for the infimum of an SGP. Our first result shows that, even after extracting the spectral measure in a neighbourhood of the origin, the lower tail of the moderate deviation regime for the infimum is upper bounded by the corresponding tail for an i.i.d.\ Gaussian field with variance $m = \|\mu_{a.c.}\|$.

\begin{proposition}[Upper tail for moderate deviations]
\label{p:m1}
Suppose $m = \|\mu_{a.c.}\| > 0$. Then for every $\eps \in (0,m)$ there exists $\delta_0 > 0$ such that, for all $\delta \in (0, \delta_0)$ and $\beta \in [0,d)$, as $T \to \infty$,
\[ -\log \P \Big[ \inf_{x \in B(T)} f_{\mu|_{B(\delta)^c}}(x) \ge - \sqrt{2 (m-\eps) ( d - \beta)   \log T} \Big] \ge T^{\beta + \frac{\eps}{2m}(d-\beta) + o(1)}  . \]
\end{proposition}

Proposition~\ref{p:m1} will be deduced from the following quantitative lemma for smooth approximations, which will also be of use in Section~\ref{s:sc}.
\begin{lemma}
\label{l:m1}
Let $\mu$ be a spectral measure with density in $C^2(\R^d)$ such that $\|\mathcal{F}[\mu]\|_{L_1(\Z^d)} < \infty$. Let $\rho$ be a signed measure with Jordan decomposition $\rho = \rho_+ - \rho_-$, and assume that $\mu+\rho \ge 0$. Then there exist positive constants $T_0=T_0(\rho)$, $c_1=c_1(d)$ and $c_2=c_2(d)$, such that, for all $\beta \in [0,d)$, $\ep<\|\mu\|$, and $T>T_0$,
\begin{align}
    \label{e:m15}
&  \P \Big[ \inf_{x \in B(T)} f_{\mu+\rho}(x) \ge - \sqrt{2 (\|\mu\|-\ep) ( d - \beta)   \log T} \Big] \\
& \nonumber  \qquad \qquad \qquad \le 2 \exp \Big( \tfrac{-  c_1 \|\mu\|}{ \|\mathcal{F}[\mu]\|_{L^1(\Z^d)}}T^{\beta + \frac{\eps}{\|\mu\|}(d-\beta) -c_2 \sqrt{\frac{\|\rho_{-}\|}{\|\mu\|}}}  \Big)  + e^{-e^{T^{c_1}}} .
\end{align}
Moreover, if $\int_{\R^d} |\lm|^2 d\rho_-(\lm) \le q < \infty$, then \eqref{e:m15} holds for  constants $T_0$ and $c_2$ chosen to depend only on $q$ and $d$.
\end{lemma}

\begin{proof}[Proof of Proposition~\ref{p:m1} using Lemma~\ref{l:m1}]
Let $\eps \in (0,m)$ and choose $\delta_0 > 0$ such that $\| \mu_{a.c.}|_{B(\delta_0)^c} \|  \ge m - \eps/4$. Let $\beta \in [0,d)$ and $\delta\in (0,\delta_0)$, and fix $\eps'\in (0,\eps/4)$ to be chosen later. 
We approximate $\mu_{a.c.}|_{B(\delta)^c}$ by a measure $\mu'=\mu'(\ep')$ with compactly supported density, in the sense that 
\[\mu_{a.c.}|_{B(\delta)^c} = \mu' + \mu^+ - \mu^-,\] 
where $\| \mu^- \|,  \| \mu^+  \|< \eps'$. 
We can also choose $\mu'$ (for instance by mollifying with a smooth kernel) so that, in addition, the norm $\|\mathcal{F}[\mu']\|_{L_1(\Z^d)}$ is finite. Choosing $\eps'$ sufficiently small we obtain, in particular, that $\|\mu'\| \ge \| \mu_{a.c.}|_{B(\delta)^c} \| - \|\mu^+\| \ge  m - \eps/2$.

 We apply Lemma~\ref{l:m1} with $\mu$ replaced by $\mu'$ and $\rho = \mu_{a.c.}|_{B(\delta)^c} - \mu' = \mu^+-\mu^-$. This yields the existence of $c_1(d),c_2(d)> 0$, such that
\[
- \log\P\Big[  \inf_{x \in B(T)} f_{\mu|_{B(\delta)^c}}(x) \ge - \ell   \Big]  \ge  \frac{ c_1 \|\mu'\|}{ \|\mathcal{F}[\mu']\|_{L^1(\Z^d)}}T^{\zeta + o(1)} , 
\]
where
\[ \zeta =\beta + \frac{\eps}{\|\mu'\|}(d-\beta) -c_2 \sqrt{\frac{\| \mu^{-} \|}{\|\mu'\|}} > \beta + \frac {\eps}{m}(d-\beta) - c_2\sqrt{\frac{\ep'}{m-\frac{\eps}{2}}} . \]
This gives the desired asymptotic bound by choosing $\eps'$ sufficiently small.
\end{proof}

\begin{proof}[Proof of Lemma~\ref{l:m1}]
In the proof we denote by 
$(c_i)_{i\in\N}$ positive constants which depend only on $d$, and may change from line to line. By the second item of Lemma \ref{l:infconv}, there are $T_0 = T_0(\rho)$ and $\kappa=\kappa(d)$ such that
\begin{equation}
\label{eq:tmp}
\P \Big[  \inf_{x \in B(T)} f_{\rho_-}(x) \ge - u   \Big]\ge \frac12
\end{equation}
where  $u = \kappa \sqrt{\| \rho\| \log T}$. Let $\ell=\sqrt{2 (\|\mu\|-\eps) ( d - \beta)   \log T}$.
Using \eqref{eq:tmp} together with the spectral decomposition $\mu+\rho +\rho_-=\mu+\rho_+$, we obtain
\begin{align*}
& \P \Big[  \inf_{x \in B(T)} f_{\mu+\rho}(x) \ge - \ell   \Big] 
 = \frac{ \P \Big[  \inf_{x \in B(T)} f_{\mu+\rho}(x) \ge - \ell\Big] \P \Big[ \inf_{x \in B(T)} f_{\rho_-}(x) \ge - u \Big] }{ \P \Big[  \inf_{x \in B(T)} f_{\rho_-}(x) \ge - u   \Big] } \\
& \qquad \qquad \le  \frac{ \P \Big[  \inf_{x \in B(T)} f_{\mu +\rho_+}(x) \ge - \ell - u  \Big]
}{ \P \Big[  \inf_{x \in B(T)} f_{\rho_-}(x) \ge - u   \Big] } \le 2\P \Big[  \inf_{x \in B(T)} f_{\mu +\rho_+}(x) \ge - \ell - u  \Big].
\end{align*} 
By applying Lemma \ref{l:trunc} with $s = T/2$, and taking a union bound, we obtain
 \begin{align} 
 \nonumber& \P \Big[  \inf_{x \in B(T)} f_{\mu+\rho_+}(x) \ge - \ell - u  \Big]  \le  \P \Big[  \inf_{x \in B(T/2)} f_{\mu + \phi_{T/2}^2\rho_+}(x) \ge - \ell-u \Big] \\
\nonumber & \qquad \le \P \Big[ \bigcap_{x \in B(T/2) \cap \Z^d } \{ f_{\mu + \phi_{T/2}^2\rho_+ }(x) \ge - \ell -u \} \Big] \\
\label{e:11} &  \qquad  \le  \P \Big[ \bigcap_{x \in B(T/2) \cap \Z^d } \{     f_{\mu}(x) \ge - \ell - 2u \} \Big] \! \! +   \sum_{x \in B(T/2) \cap \Z^d }  \P \Big[  f_{ \phi_{T/2}^2\rho_+}(x) \ge u   \Big].
 \end{align}

Denote by $Z$ a standard Gaussian random variable. To bound the first term in \eqref{e:11}, we apply the decoupling inequality appearing in \cite[Theorem 1]{kls82}, stationarity, and the fact that $\log (1-x) \le - x$ for $x \in (0,1)$, in order to obtain 
\begin{align*}
\P \hspace{25pt}&\hspace{-25pt}\Big[ \bigcap_{x \in B(T/2) \cap \Z^d } \{  f_{\mu}(x) \ge - \ell  - 2u \} \Big]  \le  \prod_{x \in B(T/2) \cap \Z^d} \P[ f_{\mu}(x)  \ge -\ell-2u]^{\|\mu\| / \|\mathcal{F}[\mu]\|_{L^1(\Z^d)}  }  \\
&  \le  e^{ c_1T^d \frac{\|\mu\|}{ \|\mathcal{F}[\mu]\|_{L^1(\Z^d)}}  \log  \big( 1 - \P \big[ Z \ge ( \ell + 2u) / \sqrt{ \|\mu\| }   \big] \big) } \\
&  \le  \exp \bigg( -\frac{ c_1 \|\mu\|}{\|\mathcal{F}[\mu]\|_{L^1(\Z^d)} } T^d \,  \P \bigg[ Z \ge \sqrt{2(1-\ep/\|\mu\|) ( d - \beta)   \log T} + 2 \kappa \sqrt{\frac{\|\rho_-\|}{ \|\mu\|}\log T}\bigg] \bigg)  .
\end{align*}

For the second term in \eqref{e:11} we observe that, by the decay estimate \eqref{e:stfdecay},
\[ \|\phi_{T/2}^2\rho_+ \|  \le c_2 \|\rho\|e^{- \left(\frac{ T}{2}\right)^\nu} .\]
Hence, by stationarity, we have
\begin{align*}\sum_{x \in B(T/2) \cap \Z^d } \P \Big[  f_{\phi_{T/2}^2\rho_+ }(x) \ge u   \Big] &\le c_3 T^d \P \bigg[ Z \ge  \frac{ue^{\frac12 \left(\frac{ T}{2}\right)^\nu}}{\sqrt{c_2 \|\rho\|}}  \bigg]
\le 
c_3 T^d \P \bigg[ Z \ge c_4  (\log T)e^{\frac12 \left(\frac{ T}{2}\right)^\nu}  \bigg] .
\end{align*}

Now we apply the tail bounds in Claim~\ref{c:gtb}.  We obtain that the second term in \eqref{e:11} is at most $e^{-e^{T^{c_5}}}$ for large $T$, and that the first term in \eqref{e:11} is at most, for $T>T_0$,
\begin{align*}  
&\exp \bigg(- \tfrac{ c_6 \|\mu\|}{ \|\mathcal{F}[\mu]\|_{L^1(\Z^d)}}T^{\beta + \frac{\eps}{\|\mu\|}(d-\beta) -c_2 \kappa \Big(\frac{\|\rho_-\|}{\|\mu\|} + \sqrt{\frac{\|\rho_-\|}{\|\mu\|} (1-\frac {\ep}{\|\mu\|} )(d-\beta)} \Big)} \bigg)
\\ & \qquad \qquad \le \exp \bigg( - \tfrac{ c_6 \|\mu\|}{ \|\mathcal{F}[\mu]\|_{L^1(\Z^d)}}T^{\beta + \frac{\eps}{\|\mu\|}(d-\beta) -c_7  \kappa \sqrt{\frac{\|\rho_-\|}{\|\mu\|} }} \bigg)
\end{align*} 
where we used that $\norm{\rho_-}\le \norm{\mu}$. This completes the proof of the first statement. 

For the second statement the proof is identical, except we replace the constants $T_0 = T_0(\rho)$ and $\kappa=\kappa(d)$ in \eqref{eq:tmp} with constants $T_0(q,d)$ and $\kappa = \kappa(q,d)$, for which the conclusion of \eqref{eq:tmp} is guaranteed by Lemma~\ref{l:DudBTIS}. 
\end{proof}

Next, we show that the probability of a ball event of the equivalent moderate deviation regime is lower bounded by the corresponding tail for an i.i.d.\ Gaussian field with variance $\|\mu_{a.c.} + \mu_{s.c.}\|$.

\begin{proposition}[Moderate deviations for ball events]
\label{p:m2}
Denote $m' =  \|\mu_{a.c.} + \mu_{s.c.}\|$. Then for every $\eps > 0$ and $\beta \in [0,d)$, as $T \to \infty$, 
\[   -\log \inf_{\mu' \le \mu}  \P \Big[ \sup_{x \in B(T)} 
 |f_{\mu'}(x)| \le   \sqrt{2 (m' +\eps) ( d - \beta)   \log T} \Big]  \le   T^{ ( \beta - \frac{\eps}{2m'} (d-\beta))_+ + o(1)}  .\]
\end{proposition}

\begin{proof}[Proof of Proposition \ref{p:m2}]
Let $L \in [1,T]$ and cover $B(T)$ by a collection $\mathcal{B} = \{B_i\}$ of $c_d (T/L)^d$ translated copies of $B(L)$. Then for every $\ell \ge 0$ we have
\begin{align*}
\P \Big[  \sup_{x \in B(T)} |f(x)| \le   \ell   \Big] &\ge \P \Big[ \bigcap_{B_i} \{ | f(x) | \le  \ell  \text{ for } x \in B_i  \}  \Big] \\
& \ge   \Big( \P \Big[ \sup_{x \in B_L} | f(x) | \le  \ell \Big] \Big)^{c_d (T/L)^d },
\end{align*}
where the second inequality used the Gaussian correlation inequality \cite{roy14} and stationarity. We first consider the case $m'  > 0$.  Let $\eps > 0$ and $\beta = [0,d)$ and define
\[ \beta' = \max \big\{ \beta -  \frac{ \eps(d-\beta) }{2m'} , 0 \Big\} .  \]
Define a mesoscopic scale $L = T^{1-\beta'/d}$.  Then 
\begin{align*}
     2 (m'+\eps) ( d - \beta)  ( \log T / \log L) & = 2 (m'+\eps) ( d - \beta) \Big(\frac{d}{d - \beta'} \Big)  \\
     & = \begin{cases} \frac{2d(m' + \eps)}{m' + \eps/2}  &   \text{if } \beta'> 0 \\  2 (m'+\eps) ( d - \beta) & \text{if } \beta'=0  \end{cases}  \\
     & > 2 m' d   ,
     \end{align*}
and so, by Lemma \ref{l:infconv}, 
\begin{equation*}
\inf_{\mu' \le \mu} \P \Big[   \sup_{x \in B(L)} |f_{\mu'}(x)| \le  \sqrt{ 2(m'+\eps)(d-\beta) (\log T)   }  \Big]   \to 1 . 
\end{equation*}  
Setting $\ell =  \sqrt{2 (m' +\eps) ( d - \beta)   \log T}$, we deduce that
\[  -\log \inf_{\mu' \le \mu} \P \Big[  \inf_{x \in B(T)} f_{\mu'}(x) \ge - \ell   \Big] \le c_d T^{\beta'} \]
which proves the result. In case that $m' = 0$ we still have $2 (m'+\eps) ( d - \beta)  > 2 m' d$, and so defining $\beta'=0$ the argument is still valid and the proposition follows. 
\end{proof}

Putting the previous results together we deduce bounds on the constants $m^-$ and $m^+$ defined in \eqref{e:m-} and \eqref{e:m+}:

\begin{corollary}
\label{c:m}
For every stationary Gaussian process on $\R^d$ or $\Z^d$,
\[   m \le m^- \le m^+ \le m + \|\mu_{s.c.}\| .\]
\end{corollary}
\begin{proof}
The bounds $m \le m^-$ and $m^+ \le m + \|\mu_{s.c.}\|$ follow directly from Proposition \ref{p:m1}, and from a combination of Lemma \ref{l:infconv} and Proposition \ref{p:m2} respectively. The bound $m^- \le m^+$ is immediate from definitions \eqref{e:m-} and \eqref{e:m+}.
\end{proof}

\subsection{Bounds on the smoothed field}
Recall the spectral truncation function $\phi_s$ from~\eqref{eq:stf}. We next provide \emph{a priori} bounds on the fluctuations of the `smoothed' field with spectral measure $\phi^2_s \mu$. We only state this result for continuous SGFs:

\begin{lemma}[Bounds on the smoothed field]
\label{l:derbound}
Fix $\gamma > \gamma'  > 0$ satisfying $v \ge  (1+\gamma')/(1+\gamma)$, let $g_s$ denote an SGF with spectral measure $\phi_s^2 \mu$, and let $k$ be a multi-index. Then $g_s$ is almost surely smooth, and there exist $c_d > 0$ and $s_0 = s_0(d,\gamma,\gamma',v,k)$ such that, for all $s \ge s_0$, 
\begin{align*}  \textrm{Var} \big[ \partial^k g_s(0) \big]  &\le s^{-2|k|}  (\log s)^{2|k|(1+\gamma)}   \mu[ B( (\log s)^{1+\gamma} / s )] + \|\mu\| e^{-2(\log s)^{1+\gamma'} },  \\
 \E \Big[ \sup_{x \in B(s)} \partial^k g_s(x) \Big] &\le c_d s^{-|k|} (\log s)^{(|k|+1)(1+\gamma)}  \sqrt{ \mu[ B( (\log s)^{1+\gamma} / s )] } + \sqrt{ \|\mu\| } e^{-(\log s)^{1+\gamma'} } . 
 \end{align*}
The same result holds if $\phi_s$ is replaced by $\tilde{\phi}_s$.
 \end{lemma}
 
 \begin{proof}
Observe that, for $s \ge s_0 = s_0(d, \gamma, \gamma',k)$,
\[   0 \le  |\lambda|^{2|k|} \phi_s^2(\lambda) \le |\lambda|^{2|k|} \id_{ B( (\log s)^{1+\gamma} / s )   } (\lambda)  + e^{- 2(\log s)^{1+\gamma'}} , \]
which follows from the fact that, for $s \ge s_0$ and $|\lambda| \ge (\log s)^{1+\gamma} / s$,
 \[  |\lambda|^{2|k|} \phi^2_s(\lambda ) \le  c_v^2 |\lambda|^{2|k|}  e^{- 2 (2 s |\lambda|)^v} \le e^{- ( \log s)^{(1+\gamma )v}} \le e^{- 2 (\log s)^{1+\gamma'}}. \]
In particular, for $s \ge s_0$,
\begin{align}
\label{e:varbound}
  \textrm{Var} \big[ \partial^k g_s(0) \big]  & \le \int |\lambda|^{2|k|} \phi^2_s(\lambda) d\mu(\lambda) \\
\nonumber   & \le   s^{-2|k|} (\log s)^{2|k|(1+\gamma)}  \mu[ B( (\log s)^{1+\gamma} / s )] + \|\mu\| e^{-2  (\log s)^{1+\gamma'} }  .
  \end{align}
  
For the second statement, define the field $\bar{g}_s(\cdot) = g_s(s \cdot)$, which satisfies
\[     \E \Big[ \sup_{x \in B(s)} \partial^k g_s(x) \Big]  = s^{-|k| }  \E \Big[ \sup_{x \in B(1)} \partial^k \bar{g}_s(x) \Big]  .  \]
Since $\bar{g}_s$ is stationary, by Kolmogorov's theorem there exists a $c_d>0$ such that
\[  \E \Big[ \sup_{x \in B(1)} \partial^k \bar{g}_s(x) \Big]   \le c_d \max_{|k'| \le k+1} \sqrt{ \textrm{Var} \big[ \partial^{k'} \bar{g}_s(0) \big] }  = c_d \max_{|k'| \le k+1} \sqrt{ s^{2|k'|}  \textrm{Var} \big[ \partial^{k'} g_s(0) \big] }   . \]
Combining with \eqref{e:varbound}, and using that $\sqrt{a+b} \le \sqrt{a} + \sqrt{b}$, gives that
\begin{align*}
& \E \big[ \sup_{x \in B(s)} \partial^k g_s(x) \big] \\
& \qquad \le  c_d s^{-|k|} (\log s)^{(|k|+1)(1+\gamma)}  \sqrt{ \mu[ B( (\log s)^{1+\gamma} / s )] } + \sqrt{ \|\mu\| }  s^{(|k|+1)} e^{-(\log s)^{1+\gamma'} } 
\end{align*}
which gives the second statement (after adjusting $s_0$). 

The proof with $\tilde{\phi}_s$ replacing $\phi_s$ is identical.
 \end{proof}

 \begin{remark}
\label{r:dc}
Although the result is stated for continuous fields, we make use of the version with the periodised spectral truncation $\tilde{\phi}_s$ in our study of discrete fields by applying it to the \textit{canonical extension} of the field to $\R^d$, see Section \ref{s:dc}.
\end{remark}

\subsection{A sufficient condition for  origin dominance}
\label{ss:ds}

\begin{proof}[Proof of Proposition \ref{p:pos}]
We prove this in the continuous case (the discrete case is identical after replacing $\phi_s$ with $\tilde{\phi}_s$). Recall $\phi_s$ defined in \eqref{eq:stf}. Since we assume \eqref{e:doub}, by the same argument as in \eqref{e:dc} there exists a $c_1 > 0$ such that, for all $T > 0$,
\[ \int  \phi_T(\lambda) d\mu(\lambda)   \le c_1 \mu[ B(1/T) ]. \]
On the other hand, for every $u \in \R^d$ and $T > 0$,
\[  c_2 \mu[u+B(1/T)]   \le \int \phi_T(u-\lambda) d\mu(\lambda)  \]
where $c_2 = \min_{\lambda \in B(1)} \phi_1(\lambda) > 0$. Let $r_0 > 0$ be such that $K(x) \ge 0$ for $|x| \ge r_0$. Then by Parseval's theorem, for every $u \in \R^d$ and $T > 0$,
\begin{align*}
\mu[ u + B(1/T) ] & \le c_3 \int \phi_T(\lambda - u) d\mu(\lambda)   = c_3 \int \mathcal{F}[\phi_T] e^{2\pi i \langle u, x \rangle } K(x) \, dx  \\
& \le c_3 \int \mathcal{F}[\phi_T] K(x) + 2c_3 \| \mathcal{F}[\phi_T] \|_\infty  \int_{B(r_0)} |K(x)| \, dx \\
& \le c_4   ( \mu[ B(1/T) ]  + T^{-d} ),
\end{align*}
where the second inequality used that $|e^{2\pi i \langle u, x \rangle }| \le 1$, and the final step used that $\| \mathcal{F}[\phi_T] \|_\infty =  T^{-d}$. Since we assume that $T^{-d} \le c_5 \mu[B(1/T)]$ for sufficiently large $T$, the result follows.
\end{proof}

 %%%%%%%%%%%
 %%%%%%%%%%%
 
\medskip
 \section{Bounds on the persistence probability}
 \label{s:proof}

 In this section we prove the bounds on the persistence probability given in Theorem \ref{t:bounds}. We follow the general strategy outlined in Section \ref{ss:heu} above. In the proof we fix $v \in (0,1)$ arbitrarily, and recall the truncation function $\phi_s=\phi_{s,v}$ given in \eqref{eq:stf}. The proofs focus on the case of continuous SGFs, with differences in the discrete case discussed in the final subsection.

\subsection{Relating persistence to capacity.}
We begin by formalising a preliminary connection between the capacity and the probability of persistence above high levels. 
Let $f=f_\mu$ be an arbitrary SGF associated with spectral measure $\mu$. For a compact domain $D$ and $\ell \in \R$, define $\per^\ell_\mu(D)=\{f_\mu\ge \ell\text{ for all $x\in D$}\}$ and $\lp_\mu^\ell(D)= -\log \per^\ell_\mu(D)$.

\begin{lemma}
\label{l:infcap}

Let $D$ be a compact domain. Then for every $\ell,\ell' \in \mathbb{R}$ such that $\ell' \le \ell$,
\[   \frac{\ell^2 \capa_f(D)}{2}   \le \theta^\ell_\mu(D)   \le  \lp^{\ell'}_\mu(D)+ \frac{(\ell-\ell')^2\capa_f(D)+2e^{-1}  }{2\P[\per^{\ell'}_\mu(D) ] }.   \]
\end{lemma}
\begin{proof}
Recall the classical `entropic bound' for Gaussian processes (see \cite[p.421]{bdz95})
\[  - \log \P[ f  \in A ] \le  - \log \P[f + h  \in A]  +   \frac{1}{\P[f + h  \in A]} \Big( \frac{ \|h\|_H^2  +2e^{-1}}{2}     \Big)   \]
valid for every $h \in H$ and event $A$. To obtain the upper bound, we apply this with $A=\per^\ell_\mu(D)$ and $h = (\ell-\ell')h_f(D) \in H$, where $h_f(D)$ is the equilibrium potential, which satisfies
\[  \P[  f + h \in A   ] = \P \big[  \inf_{x \in D} f(x) + h(x) \ge \ell \big] \ge \P \big[ \inf_{x \in D} f(x) \ge \ell' \big] = \P \big[\per^{\ell'}_\mu(D) \big]  .\]

For the lower bound, let $\nu_D \in \mathcal{P}(D)$ be an equilibrium measure, and note that if $f(x) \ge \ell$ for all $x \in D$, then also $\int f d\nu_D \ge \ell$. Using also that $\text{Var}[  \smallint f d\nu_D  ]  = E[\nu_D] = 1/\capa_f(D)$,  
\begin{align*}
    \P \big[ \per^{\ell}_\mu(D)\big] &\le  \P \big[ \smallint \! f d\nu_D  \ge  \ell   \big] = \P \Big[ Z \ge   \ell  \sqrt{ \capa_f(D)} \Big]   \le \exp \Big(   \frac{-\ell^2 \capa_f(D)}{2} \Big),
    \end{align*} 
where $Z$ is a standard normal random variable, and the final step is by Claim \ref{c:gtb}.
\end{proof}

\begin{remark}\label{rmk:cap as ldp}
    By taking $\ell \to \infty$ in Lemma \ref{l:infcap}, we obtain the following representation of the capacity as a large deviation exponent
for persistence above a high level:
\[ \lim_{\ell \to \infty} \frac{1}{\ell^2} \lp_\mu^\ell(D)   = \frac{\capa_f(D) }{2} . \]
\end{remark} 

\subsection{Proof of Theorem \ref{t:bounds} -- Part I : Lower bound on the persistence probability}\label{ss:lowerbd}

We now prove the first statement of Theorem \ref{t:bounds}, which we recall states that, for every $\delta > 0$, as $T \to \infty$ eventually 
\begin{equation}
    \label{e:bounds1}   \lp_\mu(T) \le (m^+  (d-\alpha^-)_+ + \delta) \log T\, \capa_f(B(T + T^\delta ) )  .
\end{equation} 

Let us first suppose that $\alpha^- > 0$; the case $\alpha^- = 0$ will be treated at the end. Let $\delta > 0$ be given, and let $T\ge 1$. 
Let $\eps>0$ be a small real number to be specified later, and decompose the spectral measure 
 \[ \mu = \mu_1 + \mu_2 + \mu_3, \] 
 with \[
 \mu_1 = \mu \phi_{T^\delta}^2,  \quad \mu_2 = \mu(1-\phi_{T^\delta}^2)|_{B(\eps)^c}, \text{ and }\mu_3 = \mu(1-\phi_{T^\delta}^2)|_{B(\eps)},
 \] 
 where $\phi_{T^\delta}$ is the spectral truncation function defined in~\eqref{eq:stf}.
 Let $f_{\mu_1}$, $f_{\mu_2}$ and $f_{\mu_3}$ be independent SGFs with respective spectral measures $\mu_1$, $\mu_2$ and $\mu_3$. 

Denote $\ell =  \sqrt{ 2m^+(d-\alpha^-)_+} +  3\delta  > 0$, and consider the following independent events:
 \begin{itemize}
      \item $A_1= \Big\{ \inf_{B(T)} f_{\mu_1} \ge \ell \sqrt{\log T}  \Big\} $;
     \item $A_2 =  \Big\{ \inf_{B(T)}    f_{\mu_2}  \ge  (-\ell+\delta) \sqrt{\log T} \Big\} $;
     \item $A_3 =  \Big\{ \sup_{B(T)}    |f_{\mu_3}|  \le  \delta \sqrt{\log T} \Big\}   $.
 \end{itemize}
Clearly $A_1 \cap A_2 \cap A_3  \implies \per(B(T))$ and so, by independence,
 \begin{equation}
     \label{e:lbsum}
-\log \P[\per(B(T))] \le -\log \P[A_1] - \log \P[A_2]  - \log \P[A_3]. 
 \end{equation} 
 We bound the probability of the events $A_1$, $A_2$ and $A_3$ separately. 

 \medskip
 \textbf{Bounding $\P(A_1)$.} We claim that, as $T \to \infty$,
 \begin{equation}
 \label{e:lb1}
   -\log \P[A_1]  \le    \frac{1}{2} (  \ell + \delta   )^2 (\log T) \, \capa_f(B(T + T^{\delta}))    ( 1 + o(1) ).
   \end{equation}
To see this, denote  \[ E_T = \Big\{ \inf_{x \in B(T)} f_{\mu_1}(x) \ge   - \delta  \sqrt{\log T}  \Big\}.\]
 Observe that as $T \to \infty$, eventually $\mu_1 =  \mu \phi_{T^\delta}^2  \le \mu'$ for some $\mu'$ satisfying $2d \|\mu'\| < \delta $, so that, by Lemma \ref{l:infconv}, we have $\P[E_T] \to 1$ as $T \to \infty$. Plugging this into  Lemma \ref{l:infcap}, applied to $f_{\mu_1}$ with $\ell'$ replaced by $-\delta \sqrt{\log T}$ and $\ell$ replaced by $\ell \sqrt{\log T}$, we obtain
 \[  - \log \P[A_1] \le  \Big(   \frac{ (\ell + \delta   )^2 (\log T)  \capa_{\mu_1}(B(T))  }{2}\Big)(1+o(1)).   \]
Since, by Lemma~\ref{l:captrunc}, $\capa_{\mu_1}(B(T)) \le \capa_f(B(T + T^\delta))$ we deduce \eqref{e:lb1}.

  \medskip
 \textbf{Bounding $\P(A_2)$.} We restrict our choice of $\eps$ to satisfy both $\|\mu|_{B(\eps)}\| <  \delta /(2d)$  and 
\begin{equation}
    \label{e:mu2}
 - \log \P\Big[ \inf_{x \in B(T)} f_{\mu|_{B(\eps)^c}}(x) \ge (- \ell +2\delta) \sqrt{\log T}   \Big] \le T^{ \alpha^- - c_\delta + o(1) }.
 \end{equation}
This is possible by the definitions of $\ell$ and $m^+$ as in~\eqref{e:m+}. 

Observe that, since $\mu|_{B(\eps)^c} = \mu_2 +  \mu \phi_{T^\delta}^2|_{B(\eps)^c}$,
\[ \P[A_2] \ge  \P\Big[ \inf_{x \in B(T)} f_{\mu|_{B(\eps)^c}}(x) \ge (- \ell +2\delta) \sqrt{\log T}   \Big] -  \P \big[   \sup_{x \in B(T)}    |f_{ \mu \phi_{T^\delta}^2|_{B(\eps)^c}}(x) |  \ge  \delta \sqrt{\log T} \big] . \]
By \eqref{e:mu2} the first term in this expression is at least $\exp(-T^{\alpha^- - c_\delta + o(1) })$. By Lemma \ref{l:derbound} and the Borell-TIS inequality, the second term is at most $\exp(-e^{(\log T)^{1 + c}})$ for some constant $c > 0$ and sufficiently large $T$. Combining these we conclude that
\begin{equation}
    \label{e:lb2}
- \log \P[A_2] \le T^{\alpha^- - c_\delta + o(1) } .
\end{equation}

\textbf{Bounding $\P(A_3)$.} To see that 
\begin{equation}
    \label{e:lb3}
 \P[A_3] \to 1 
\end{equation}
we simply recall that $\|\mu|_{B(\eps)}\| <  \delta /(2d)$, and apply Lemma \ref{l:infconv}. 

\medskip
\textbf{Concluding the proof.}  
Combining \eqref{e:lbsum}, \eqref{e:lb1}, \eqref{e:lb2} and \eqref{e:lb3}, we have 
\[ -\log \P[\per(B(T))] \le    \frac 1 2 (  \ell + \delta   )^2 (\log T)  \capa_f(B(T + T^\delta)) ( 1 + o(1) )    + T^{\alpha^- - c_\delta + o(1) } .  \]
Adjusting the constant $\delta > 0$, and since $\capa_f(B(T))  \ge T^{\alpha^- + o(1)}$ by Corollary \ref{c:capbounds},  this completes the proof of \eqref{e:bounds1} in the case $\alpha^- > 0$.

In the case $\alpha^- = 0$, we instead restrict $\eps$ to satisfy 
\begin{equation*}
\P\Big[ \inf_{x \in B(T)} f_{\mu|_{B(\eps)^c}}(x) \ge (- \ell +2\delta) \sqrt{\log T}   \Big] \to 1 ,
 \end{equation*}
 rather than \eqref{e:mu2}, which is again possible by the definitions of $\ell$ and $m^+$. Using this we verify $\P[A_2] \to 1$ in this case, and the remainder of the proof proceeds as before.
\qed

\subsection{Proof of Theorem \ref{t:bounds} -- Part II: Upper bound on persistence probability}
\label{s:ub}

We prove the second statement of Theorem \ref{t:bounds}, which we recall states that, if $\mu$ is regular in the sense of \eqref{e:rss}, then for every $\gamma > 0$ there exists an $\eps > 0$ such that, as $T \to \infty$, eventually 
\begin{equation}
    \label{e:bound2}
 \theta_\mu(T)  \ge (m^- (d-\alpha^+ ) - \gamma) \log T\,\capa_f(B(T-T^{1-\eps}))  .
 \end{equation}

Suppose that $m^- > 0$ and $\alpha^+ < d$, as otherwise the bound trivially holds. It is enough to prove the result for $\gamma > 0$ sufficiently small. Given $T \ge 1$ we decompose the  measure \[\mu = \mu_1 + \mu_2 + \mu_3,\] where
\[\mu_1 = \mu|_{B(\delta)} \phi^2_{T^{1-\eps} }, \quad
\mu_2 = \mu|_{B(\delta)^c},  \text{   and }\mu_3 = \mu|_{B(\delta)}  (1 - \phi^2_{ T^{1-\eps} } ),\] 
for small constants $\delta,\eps>0$ that will be chosen later. Write $f_{\mu_1}$ and $f_{\mu_2}$ for two independent SGFs with respective spectral measures $\mu_1$ and $\mu_2$. By Lemma \ref{l:trunc},
 \[ \P[ f \in \per(B(T)) ] \le \P[ f_{\mu_1} + f_{\mu_2} \in \per(B(T-T^{1-\eps}) ) ] ,\]
 so that it suffices to bound the latter probability. 

Set $\ell=\sqrt{2m^{-}(d-\alpha^+)}-3\gamma > 0$.
 Abbreviating $\widetilde{B}(T)=B(T-T^{1-\eps})$, we make the following claim:

\begin{claim}
\label{cl:events containment}
For every $L >0$ such that $L \le T - T^{1-\eps}$ we have
\[\{f_{\mu_1} + f_{\mu_2} \in \per(\widetilde{B}(T)) \}\subset A_1\cup A_2 \cup A_3,\] where, abbreviating $\widehat{B}(T) = B(T - T^{1-\eps} - L)$ and $B_x(r) := x + B(r)$:
 \begin{itemize}[itemsep=4pt]
   \item $A_1 = \{ \inf_{ \widehat{B}(T) } f_{\mu_1} \ge  \ell \sqrt{\log T}  \} $;
   \item $A_2 = \{ \exists x \in \widetilde{B}(T)  \, \text{s.t.} \, \inf_{B_x(L)} f_{\mu_2} \ge -(\ell+\gamma) \sqrt{  \log T} \}   $;
   \item $A_1' = \{ \sup_{\widetilde{B}(T)} \| \nabla f_{\mu_1} \|_2  \ge  \gamma  L^{-1}  \sqrt{\log T}  \} $.
 \end{itemize}     
\end{claim}
\begin{proof}
On the event $(A_1 \cup A_1')^c$, taking $x_0=\arg\min_{\widehat{B}(T)} f_{\mu_1}$, we have $f_{\mu_1}(y)\le \ell+\gamma$ for all $y\in B_{x_0}(L) \subset \widetilde{B}(T)$. Hence, if $\per(\widetilde B(T))$ occurs, $A_2$ holds with respect to $x=x_0$.
 \end{proof}

We are left with showing that for all sufficiently small $\gamma$, and with a suitable choice of the parameters $\delta$, $\eps$, and $L$, the events $A_1$, $A_2$, and $A_1'$, occur with sufficiently small probability. We write $c_d>0$ for a dimensional constant that may change from line to line.

\medskip
\textbf{Fixing the parameters.} 
By the definition of $m^-$~\eqref{e:m-} with $\beta=\alpha^+$, and since $\gamma > 0$ is sufficiently small, we may fix some $\delta> 0$ sufficiently small such that (recall $\mu_2 = \mu|_{B(\delta)^c}$)
\begin{equation}
    \label{e:mixing1}
 -\log \P \Big[ \inf_{B(T)} f_{\mu_2} \ge - \sqrt{2 m^-( d - \alpha^+)  \log T} \Big]  \ge  T^{\alpha^+ + \gamma  + o(1)}   .
\end{equation}
Then we choose $\zeta>0$ sufficiently small so that 
\begin{align}
\frac{\ell + 2\gamma}{\sqrt{1 - 3\zeta}} &<  \sqrt{ 2m^-(d-\alpha^+)}, \quad \text{and}   \label{e:eps1}\\
    \label{e:eps2}
 \alpha^+ + 3\zeta &<   (1 -3\zeta) ( \alpha^+ + \gamma).
\end{align}
Hence, we obtain from \eqref{e:mixing1} and \eqref{e:eps1} that for all sufficiently small $\delta$, we have
\begin{equation}
    \label{e:mixing}
 -\log \P \Big[ \inf_{x \in B(T)} f_{\mu_2}(x) \ge  \frac{ - (\ell+2\gamma)}{\sqrt{1 -3\zeta}}  \sqrt{ \log T} \Big]  \ge  T^{\alpha^+ + \gamma  + o(1)}   .
\end{equation}
Next we set $\eps\in (0,\zeta/2) $ sufficiently small to satisfy 
\begin{equation}
    \label{e:reg2}
   \frac{ \mu[ B( 1 / T^{1-\eps/2} )] }{  \mu[B(1/T^{1+\eps})]  }  \le T^{\zeta}, 
   \end{equation} 
   eventually, as $T \to \infty$. This is possible since $\mu$ is regular in the sense of \eqref{e:rss}. Finally we set
$L=T^{1-2\eps-2\zeta}$, so that $L \in ( T^{1-3\zeta} , T^{1-\eps})$.

\medskip
\textbf{Bounding $\P(A_1)$.} By Lemma \ref{l:infcap} and the monotonicity in Claim \ref{c:mon} we have
\begin{equation}
\label{e:a1}
  -\log \P[A_1]   \ge  \frac{1}{2} \ell^2 (\log T)   \capa_{f_{\mu_1} }(\widehat{B}(T) ) \ge \frac{1}{2} \ell^2 (\log T)  \capa_f(B(T-2T^{1-\eps})) .
\end{equation}

\medskip
\textbf{Bounding $\P(A_2)$.} 
Observe that we can find a collection $\mathcal{B}$ of $c_d (T/L)^d$ translated copies of the ball $B(L/c_d)$ such that every $B_x(L)$, $x \in \widetilde B(T) \subset B(T)$, contains at least one ball in $\mathcal{B}$. Using a union bound and stationarity,  we obtain
\begin{align*}
 -\log \P[A_2] & \ge - \log (c_d (T/L)^d)  -\log \P \left[  \inf_{ B(L/c_d)} f_{\mu_2} \ge -(\ell+\gamma) \sqrt{ \log T}  \right] \\
 & =   -\log c_d - d(2\eps+2\zeta)\log (T) - \log \P \left[  \inf_{ B(L/c_d)} f_{\mu_2} \ge \frac{ - (\ell+\gamma) }{ \sqrt{1 - 2\eps-2\zeta} }\sqrt{ \log L }  \right].
 \end{align*}

From \eqref{e:mixing}  we thus obtain for all sufficiently small $\delta>0$, 
 \begin{align}\label{e:a3}
 -\log \P[A_2]  & \ge  -  d(2\eps+2\zeta)\log (T)+  L^{\alpha^+ + \gamma + o(1)}  \notag \\ 
  & \ge   -  d(3\zeta)\log (T)   +  T^{(\alpha^+ + \gamma)(1-3\zeta) + o(1) } \notag
  \\ & \ge T^{\alpha^++3\zeta + o(1) }  \ge \capa_f(B(T)) \,T^{3\zeta+o(1)},  \end{align}  
where the third inequality uses \eqref{e:eps2} and the fourth uses Corollary \ref{c:capbounds}.

\medskip
\textbf{Bounding $\P(A_1')$.}
Observe that 
\[A_1' \subset \Big\{ \sup_{k\in[d]}\sup_{\widetilde B(T)}c_d|\partial^{e_k}f_{\mu_1}|\ge \gamma L^{-1}\sqrt{\log T}\Big\},\]
where $e_k$ are cardinal direction unit vectors. Covering $\widetilde{B}(T)$ by a set of $c_d T^{\eps d}$ balls of radius $T^{1-\eps d}$, taking union bound over balls of radius $T^{1-\eps}$, and using stationarity, we obtain
\begin{equation}\label{eq: PA2a}
 - \log \P[A_1']    \ge - c_d - \eps d\log T - \log \P \Big[ \sup_{k\in [d]} \sup_{ B(T^{1-\eps})} c_d\left|\partial^{e_k} f_{\mu_1}  \right|  \ge  \gamma L^{-1}  \sqrt{\log T}  \Big].
 \end{equation}
To control this probability, let $1 \le k\le d$. By Lemma \ref{l:derbound} we have for all sufficiently large $T$,
\begin{align}
 \sqrt{\textrm{Var}[ \partial^{e_k} f_{\mu_1}(0) ]} &\le  T^{-(1-\eps) }  (\log T)^{c}  \sqrt{ \mu[ B( 1/ T^{1-\eps/2} )]},   \notag\\
\E \Big[ \sup_{x \in B(T^{1-\eps})} \partial^{e_k}  f_{\mu_1}(x) \Big] &\le T^{-(1-\eps)}(\log T)^c \sqrt{ \mu[ B( 1/ T^{1-\eps/2} )] },\notag
\end{align}
where $c > 0$ is a constant. 
Recalling our choice $L=T^{1-2\eps-2\zeta}$, we see that both these bounds are asymptotically smaller than $L^{-1} \sqrt{\log T}$ as $T \to \infty$. By the Borell-TIS inequality, we thus have
\begin{align*}  
  - \log \P &\Big[ \sup_{ B(T^{1-\eps})} c_d| \partial^{e_k}  \nabla f_{\mu_1} | \ge  \gamma  L^{-1}  \sqrt{\log T}  \Big]  \notag\\& \ge \frac{\left(\gamma  L^{-1}  \sqrt{\log T} - c_d\E \Big[ \sup_{ B(T^{1-\eps})} \partial^{e_k}  f_{\mu_1} \Big]\right)^2}{2c_d^2\textrm{Var}[ \partial^{e_k} f_{\mu_1}(0) ]} \ge \frac{ T^{ 2\zeta +  o(1)}    }{ \mu[ B( 1/T^{1-\eps/2} )] }  .
\end{align*}
Using \eqref{e:reg2}, 
taking a union bound over $k$, and plugging this into the dominant term of \eqref{eq: PA2a} we obtain
\begin{align}
     - \log \P[A_1']  &\ge   \frac{ T^{\zeta + o(1)}  }{  \mu[B(1/T^{1+\eps})]  }  \ge  T^{\zeta + o(1)}  \,\capa_f(B(T)  ),\label{e:a2} 
\end{align}
where the second inequality relates $\mu[B(1/T^{1+\eps})]$ and $\capa_f(B(T))$ via Lemma~\ref{l:capbounds}.

\medskip
\textbf{Concluding the proof.}  
Gathering the estimates \eqref{e:a1}, \eqref{e:a3} and \eqref{e:a2}, we conclude that as $T \to \infty$, eventually
\begin{align*} 
& - \log \P \big[ \per(B(T))   \big]  \\
& \qquad \ge  \min\big\{\ell^2(\log T) \capa_f(B(T-2T^{1-\eps})),
T^{3\zeta+o(1)} \capa_f(B(T))  \big\}   
\\& \qquad =
(m^-(d-\alpha^+) -  3\gamma)(\log T) \capa_f(B(T-2T^{1-\eps})) .
\end{align*}
Recalling that $\ell=\sqrt{2m^{-}(d-\alpha^+)}-3\gamma$, \eqref{e:bound2} follows by adjusting the constants $\gamma, \eps > 0$.
\qed

\subsection{Adapting the argument to discrete fields}
\label{s:dc}
In the case of discrete fields we replace the spectral truncation functions $\phi_{T^\delta}$ and $\phi_{T^{1-\eps}}$ with their respective periodisations $\tilde{\phi}_{T^\delta}$ and $\tilde{\phi}_{T^{1-\eps}}$. Then the proof proceeds as in the continuous case, with the exception that the event $A_1'$ in Claim \ref{cl:events containment} must be redefined. Specifically, recall that an SGF on $\Z^d$ has a \textit{canonical extension} to a smooth SGF on $\R^d$, induced via the inclusion of its spectral measure from the torus to $\R^d$. Let $\tilde{f}_{\mu_1}$ denote the canonical extension of $f_{\mu_1}$. Then we redefine the event $A_1'$ with respect to $\tilde{f}_{\mu_1}$, and the remainder of the proof adapts immediately. In particular the conclusion of  Claim \ref{cl:events containment} is still valid, and we can apply Lemma \ref{l:derbound} to $\tilde{f}_{\mu_1}$ (see Remark~\ref{r:dc}).

 %%%%%%%%%%%
 %%%%%%%%%%%

%%%%%%
\medskip
\section{Entropic Repulsion}
\label{s:er}
In this section we prove Theorem \ref{t:er} regarding entropic repulsion. 

\subsection{Reduction to a bound on persistence with atypical shape}

Our proof relies on the following enhancement of the upper bound in Theorem~\ref{t:main}, which controls the probability of persistence with an atypical shape. Recall that $h_T = h_f(B(T))$ is the equilibrium potential of $B(T)$.

\begin{proposition}[Persistence with atypical shape]
\label{p:er2}
 Assume the conditions of Theorem \ref{t:er} hold. Then for every $\eta \in \mathcal{T}$ and $\gamma,\Delta > 0$, as $T \to \infty$, eventually
 \begin{align}
    \nonumber  &  -\log \P \Big(\per(B(T)) , \Big| \Big\langle \tfrac{f}{\sqrt{2m(d-\alpha) \log T}}, \eta_T \Big\rangle -  \langle h_T, \eta_T \rangle \Big| \ge   \Delta  \Big)  \\
 \nonumber   & \qquad \qquad \ge  (m(d-\alpha) - \gamma) \Big( \capa_f  (B(T)) + \min\Big\{ \frac{ \Delta^2-\gamma}{ E[\eta_T]}, \capa_f(B(T)) \Big\} \Big) \log T .
       \end{align} 
\end{proposition}

\begin{proof}[Proof of Theorem \ref{t:er} given Proposition \ref{p:er2}]
By the first statement of Corollary \ref{c:enbound} there exists a $c > 0$ such that $E[\eta_T] \le c / \capa_f(B(T))$ for all $T \ge 1$. Now, given  $\Delta > 0$, choose $\gamma>0$ sufficiently small so that 
\[ \big( m(d-\alpha) - \gamma \big) \big( 1 + \min\{ (\Delta^2-\gamma)/c , 1 \} \big) > m(d-\alpha) + \gamma .\]
Combining Proposition \ref{p:er2} with the lower bound in Theorem \ref{t:main} gives that
\[ \P \Big[ \Big|\Big\langle \tfrac{f}{\sqrt{2m(d-\alpha) \log T}}, \eta_T \Big\rangle  -   \langle h_T, \eta_T \rangle \Big| \ge \Delta    \: \Big| \: \per(B(T))  \Big]  \to 0 ,\]
from which we deduce Theorem \ref{t:er} by taking $\Delta \to 0$.
\end{proof}

\subsection{Auxiliary results}

The proof of Proposition~\ref{p:er2} is based on two auxiliary results, whose proofs are given in Section~\ref{s:stab}. The first is an extension of the bound in Lemma~\ref{l:infcap} by additionally demanding that macroscopic averages of $f$ deviate from the equilibrium potential. This is inspired by \cite[Lemma 4.2]{nit18} and \cite[Proposition 3.3]{cn20}, which proved similar results for the GFF. Recall that $h_\nu$ is the potential associated with a signed measure $\nu \in \mathcal{M}_{>0}$.

\begin{lemma}[Persistence and deviation from a given potential]\label{l:shape}
Let $D$ be a compact domain. Then for every $\ell, \Delta > 0$, $\nu \in \mathcal{P}(D) \cap \mathcal{M}_{>0} \cap \mathcal{M}_{<\infty}$, and $\eta \in \mathcal{M}_{<\infty}$,
\begin{equation}
    \label{e:eraux1}
- \log \P \Big[  \inf_{x \in D} f(x) \ge \ell , | \langle f,\eta \rangle - \ell \langle h_\nu, \eta \rangle| \ge \ell \Delta  \Big]   \ge  \frac{\ell^2}{2} \Big( \frac{1}{E[\nu]}   + \min \Big\{ \frac{\Delta^2}{E[\eta]}, \frac{1}{E[\nu] }  \Big\} \Big)  .   
\end{equation}
\end{lemma}
Applying Lemma~\ref{l:shape} with $\nu$ an equilibrium measure for $D$, and recalling that $h_\nu = h_{f}(D)$ and $E[\nu] = 1/\capa_f(D)$, we arrive at the following:

\begin{corollary}[Persistence and deviation from the equilibrium potential]\label{l:infcap2}
Let $D$ be a compact domain such that $\capa_f(D)  \in (0,\infty)$. Then for every $\ell, \Delta > 0$ and $\eta \in \mathcal{M}_{<\infty}$,
\[ - \log \P \Big[  \inf_{x \in D} f(x) \ge \ell , | \langle f,\eta \rangle - \ell \langle h_f(D), \eta \rangle| \ge \ell \Delta  \Big]   \ge  \frac{\ell^2}{2} \Big( \capa_f(D)   + \min \Big\{ \frac{\Delta^2}{E[\eta]}, \capa_f(D) \Big\} \Big)  .   \]
\end{corollary}

The second auxiliary result is a stability property of the equilibrium measure. This states roughly that if $\nu \in \mathcal{P}(D)$ approximates an equilibrium measure $\nu_D$ for $D$ in the sense that $E[\nu] \approx E[\nu_D] = 1/\capa_f(D)$, then $h_\nu$ approximates the equilibrium potential $h_f(D)$ in the sense that $\langle h_\nu, \eta \rangle \approx \langle h_f(D),\eta \rangle$ for every $\eta \in \mathcal{M}$. More precisely, we prove the following:

\begin{lemma}[Stability of the equilibrium potential]
\label{l:stab}
Let $D$ be a compact domain. Then for all $\nu \in \mathcal{P}(D)$ such that $E[\nu] \capa_f(D) \le 2$ and all $\eta \in \mathcal{M}$,
\[ |  \langle h_f(D) , \eta \rangle - \langle h_\nu , \eta \rangle   | \le 2 \sqrt{ E[\eta] \capa_f(D)(E[\nu] \capa_f(D)-1 ) } . \]
\end{lemma}
Our proof is probabilistic (see Section~\ref{s:stab}); presumably there exists a non-probabilistic proof of this fact, but we were not able to locate it in the literature. 
As a consequence of this stability, we show that the regularity of the capacity implies an analogous property for the equilibrium potential averaged against rescaled test functions $\eta_T := T^{-d} \eta(\cdot/T)$:

\begin{corollary}[Asymptotic stability to change of scale]\label{c:regh}
Suppose that \eqref{e:capcon},~\eqref{e:doub} and ~\eqref{e:ds} all hold. Then for every $\eps > 0$ and $\eta \in \mathcal{M}$ that is compactly supported with bounded density, 
\[  \lim_{T\to\infty} \sup_{s \in [0, T^{1-\eps}] } \big| \langle  h_T, \eta_T \rangle - \langle h_{T - s}, \eta_T \rangle    \big| =0 .  \]
\end{corollary}

\begin{proof}
Let $T > 0$ and $s \in [0, T^{1-\eps}]$, and let $\nu$ be an equilibrium measure for $B(T-s)$. By~\eqref{e:capcon} we have $\capa_f(B(T)) \le 2 \capa_f(B(T - s))$ for sufficiently large $T$. Applying Lemma \ref{l:stab} (with $D = B(T)$ and $\eta = \eta_T$) we see that
\begin{align}
\nonumber
& \big| \langle  h_T, \eta_T \rangle - \langle h_{T - s}, \eta_T \rangle    \big| \\
\label{e:regh1} & \qquad \le  2 \sqrt{ E[\eta_T] \capa_f(B(T)) \big(  \capa_f(B(T)) / \capa_f(B(T - s)) - 1 \big) } . 
 \end{align} 
Since we assume \eqref{e:doub}, \eqref{e:ds}, and $\eta$ is compactly supported with bounded density, by the first statement of Corollary \ref{c:enbound} $\E[\eta_T] \capa_f(B(T))$ is bounded over $T \ge 1$. Thus using \eqref{e:capcon} again, the right-hand side of \eqref{e:regh1} tends to zero as $T \to \infty$, uniformly over  $s \in [0, T^{1-\eps}]$.
\end{proof}

\subsection{Proof of Proposition ~\ref{p:er2}: Persistence with atypical shape}
Let $\eta \in \mathcal{T}$ and observe that, by \eqref{e:heta}, the Cauchy-Schwarz inequality, and the first statement of Corollary \ref{c:enbound},  for all $T\ge1 $ we have
\[ \langle h_T, \eta_T \rangle  = \frac{E[\mu_{B(T)},\eta_T]}{E[\mu_{B(T)}]} \le \sqrt{E[\eta_T]\capa_f(B(T))  }\le c_\eta, \]
 for some $c_\eta > 0$. Hence we assume without loss of generality that 
\begin{equation}
\label{e:hbound}
    \text{$|\langle h_T, \eta_T \rangle| \le 1$  for all $T \ge 1$,}
\end{equation} 
as otherwise we may replace $\eta$ by $\eta/c_\eta$ and $\Delta$ by $\Delta/c_\eta$, and adjust $\gamma > 0$ accordingly.

We decompose the spectral measure as in Section~\ref{s:ub}:
\[\mu = \mu_1 + \mu_2 + \mu_3,\] where
\[\mu_1 = \mu|_{B(\delta)} \phi^2_{T^{1-\eps} }, \quad
\mu_2 = \mu|_{B(\delta)^c},  \text{   and }\mu_3 = \mu|_{B(\delta)}  (1 - \phi^2_{ T^{1-\eps} } ),\] 
where  $\delta, \eps>0$ small constants defined as in Section~\ref{s:ub}. Write $f_{\mu_1}$ and $f_{\mu_2}$ for two independent SGFs with respective spectral measures $\mu_1$ and $\mu_2$.

As in Section~\ref{s:ub} we set $\ell = \sqrt{2 m (d-\alpha)} - 3 \gamma > 0$. Writing $\ell_T =  \sqrt{2 m (d-\alpha) \log T} = (\ell+3\gamma)\sqrt{\log T}$ we have, using \eqref{e:hbound},
\[   \big|\langle f, \eta_T \rangle -  \ell_T  \langle h_T, \eta_T \rangle \big| \ge \ell_T \Delta   \implies   \big|\langle f, \eta_T \rangle -  \ell \sqrt{\log T}  \langle h_T, \eta_T \rangle \big| \ge (\ell \Delta-3\gamma)\sqrt{\log T} .  \]
Applying Lemma \ref{l:trunc}, and abbreviating $\widetilde{B}(T)=B(T-T^{1-\eps})$, we then have
 \begin{align} 
\nonumber & \P\left[ f \in \per(B(T)) ,   \big|\langle f, \eta_T \rangle -  \ell_T  \langle h_T, \eta_T \rangle \big| \ge \ell_T \Delta    \right]  \\
\nonumber &  \qquad  \le  \P\left[ f \in \per(B(T)) ,   \langle f, \eta_T \rangle \ge  \ell \sqrt{\log T} \langle h_T, \eta_T \rangle  + (\ell \Delta - 3\gamma) \sqrt{\log T}   \right]  \\
\nonumber & \qquad \qquad + \P\Big[ f \in \per(B(T)) ,  \: \langle f, -\eta_T \rangle \ge  \ell \sqrt{\log T}  \langle h_T,- \eta_T \rangle  + (\ell \Delta - 3\gamma) \sqrt{\log T} \Big]  \\
  & \qquad \le  \P \big[ f_{\mu_1} + f_{\mu_2}   \in \per(\widetilde{B}(T)), \: E \big]    + 2 \P[F]  ,\label{e:erbound1}
    \end{align}
where
\[ E = \Big\{ \big| \langle f_{\mu_1} + f_{\mu_2}  , \eta_T \rangle -    \ell \sqrt{\log T}  \langle h_T, \eta_T \rangle \big| \ge (\ell \Delta - 4\gamma ) \sqrt{\log T}  \Big\}, \]
and, abbreviating $\eta_{T,u} = \eta_T(\cdot-u)$, 
\[ F =  \Big\{ \sup_{u \in B(T^{1-\eps})} \big| \langle f, \eta_T - \eta_{T,u} \rangle \big| \ge \gamma \sqrt{\log T} \Big\} . \]

It suffices to bound \eqref{e:erbound1}. We make the following claim, analogous to Claim \ref{cl:events containment} in Section~\ref{s:ub}:

\begin{claim}
    \label{cl:events containment2}
For every $L$ such that $L \le T - T^{1-\eps}$ we have
\[\{f_{\mu_1} + f_{\mu_2} \in \per(\widetilde{B}(T)), E \}\subset (A_1 \cap A_1'')  \cup A_2 \cup A_1' \cup A_2',\] where, abbreviating $\widehat{B}(T) = B(T - T^{1-\eps} - L)$:
 \begin{itemize}[itemsep=4pt]
    \item $A_1 =  \Big\{ \inf_{\widehat{B}(T) } f_{\mu_1} \ge  \ell \sqrt{\log T}    \Big\}$;
    \item $A_2 = \{ \exists x \in \widetilde{B}(T) \, \text{ s.t.} \, \ \inf_{B_x(L)} f_{\mu_2} \ge -(\ell+\gamma) \sqrt{  \log T} \}$;
   \item $A_1' = \{ \sup_{\widetilde{B}(T)} \| \nabla f_{\mu_1} \|_2  \ge  \gamma  L^{-1}  \sqrt{\log T}  \} $.
     \item $A_1'' =  \Big\{ \big|  \langle f_{\mu_1}  , \eta_T \rangle -   \ell \sqrt{\log T}  \langle h_T, \eta_T \rangle \big| \ge (\ell \Delta - 5\gamma) \sqrt{ \log T}   \Big\}$
   \item $A_2' =  \{  | \langle f_{\mu_2}  , \eta_T \rangle | \ge \gamma  \sqrt{\log T}  \}.$
 \end{itemize} 
\end{claim}

\begin{proof}
The events $A_1$, $A_2$, $A_1'$ are the same as in Claim \ref{cl:events containment}, and so $\{f_{\mu_1} + f_{\mu_2} \in \per(\widetilde{B}(T)) \}\subset A_1 \cup A_2 \cup A_1'$. On the other hand it is clear that $E \subset A_1'' \cup A_2'$, and the claim follows.
\end{proof}

Fixing $L$ as in Section \ref{s:ub}, we are left to show that each of the events $A_1 \cap A''_1$, $A_2$, $A_1'$, $A_2'$, and $F$, have sufficiently small probability.

\medskip
\textbf{Bounding $\P(A_1 \cap A''_1)$.} By Corollary \ref{c:regh}, as $T \to \infty$,
\[ | \langle h_T, \eta_T \rangle - \langle h_{f, \widehat B(T)}, \eta_T \rangle | \to 0 .\]
Hence for all $T$ sufficiently large, $A_1'$ implies
that
\[ \Big\{ \inf_{x \in \widehat B(T) } f_{\mu_1} \ge  \ell \sqrt{\log T}, \:\:  \big|  \langle f_{\mu_1} , \eta_T \rangle -   \ell \sqrt{\log T }  \langle h_{\widehat B(T)}, \eta_T \rangle \big| \ge (\ell \Delta - 5\gamma) \sqrt{ \log T}   \Big\}.   \]
Applying Corollary \ref{l:infcap2}, and using the monotonicity of Claim~\ref{c:mon} and the regularity of the capacity \eqref{e:capcon}, we have
\begin{align}
 \nonumber   &  -\log \P[A_1 \cap A_1'']  \ge  \frac{1}{2} \ell^2 (\log T)  \Big(  \capa_{ \mu_1 }( \widehat B(T) )  +  \min \Big\{ \frac{(\Delta-5\gamma/\ell)^2}{E_{\mu_1}[\eta_T]}, \capa_{ \mu_1 }(\widehat B(T) )   \Big\} \Big)  \\
    \label{e:erbound6}    & \qquad \quad  \ge  \frac{1+o(1)}{2} \ell^2 (\log T)  \Big(  \capa_f(B(T) )  +  \min \Big\{ \frac{(\Delta-5\gamma/\ell)^2}{E[\eta_T]}, \capa_f(B(T) )   \Big\} \Big)  .
\end{align}

\medskip
\textbf{Bounding $\P(A_2)$ and $\P(A_1')$.}  As in Section~\ref{s:ub}, we obtain
\begin{equation}
     \label{e:erbound7} 
 -\log \P[A_2] \ge \capa_f(B(T)) T^{ \zeta + o(1)} \textrm{and} \quad - \log \P[A_1']   \ge  T^{\zeta + o(1)}  \,\capa_f(B(T)  ) T^{\zeta + o(1)}, 
\end{equation}
for some constant $\zeta > 0$.

\medskip
\textbf{Bounding $\P(A_2')$.} 
Notice that $ \langle f_{\mu_2} , \eta_T \rangle $ is a centred Gaussian variable with variance 
\[  \textrm{Var}[ \langle f_{\mu_2} , \eta_T \rangle ] = E_{\mu_2}[ \eta_T ] = o(1/\capa_f(B(T)) )  , \]
where we used the second statement of Corollary \ref{c:enbound}. Hence, by the standard Gaussian tail bound in Claim \ref{c:gtb}, as $T \to \infty$
\begin{equation}
     \label{e:erbound8} 
\frac{ - \log \P[A_2'] }{ (\log T) \capa_f(B(T)) } \to \infty . 
\end{equation}

\medskip
\textbf{Bounding $\P(F)$.} We will prove that, as $T \to \infty$
\begin{equation}
    \label{e:erbound2}
 - \log \P [F] \ge T^{ \min\{ \alpha + \eps, d \} + o(1) } .
 \end{equation}
For this we use a fine mesh approximation as follows. Let $c > 0$ be a constant that depends only on $f$, $\eta$, and $d$ and may change from line to line.  Fix $\zeta \ge d/2 - 1$, and observe that 
\begin{align*}
 \sup_{u \in B(T^{1-\eps})} | \langle f, \eta_T- \eta_{T,u} \rangle|  
&\le  \sup_{u \in B(T^{1-\eps}) \cap T^{-\zeta} \Z^d} | \langle f, \eta_T- \eta_{T,u} \rangle| +  \sup_{\substack{u,v \in B(T^{1-\eps}),\\ |u-v| \le \sqrt{d} T^{-\zeta}}}  | \langle f, \eta_{T,u}- \eta_{T,v} \rangle|.
 \end{align*}
By the union bound,
\begin{align}
    \label{e:erbound3}
\P[F] & \le c T^{\zeta d} \sup_{u \in B(T^{1-\eps}) \cap T^{-\zeta} \Z^d}   \P \big[  | \langle f, \eta_T- \eta_{T,u} \rangle| \ge (\gamma/2)  \sqrt{\log T} \big]  \\
\nonumber & \qquad + \P \Big[ \sup_{\substack{u,v \in B(T^{1-\eps}),\\ |u-v| \le \sqrt{d} T^{-\zeta}}}  | \langle f, \eta_{T,u}- \eta_{T,v} \rangle| \ge (\gamma/2) \sqrt{\log T} \Big]  .  \end{align}
To control the terms in \eqref{e:erbound3} we use the following  modulus of continuity estimates:
\begin{align}
\sup_{u \in B(T^{1-\eps})}  E [   \eta_T- \eta_{T,u}   ]  
\label{e:enbound3}  & \le c T^{-\alpha -  \eps + o(1) },  \\ 
\|\eta_{T,u} - \eta_{T,v}\|_{L^1}  &\le c|u-v|/T.\label{e:decompeta2}
\end{align}

To establish these, recall that $\eta \in \mathcal{T}$ is supported on a domain $D$ with piece-wise smooth boundary, and is Lipschitz on its interior. Decompose 
\begin{equation*}
\eta_{T,u}- \eta_{T} =\eta^1_{T,u}  + \eta^2_{T,u},
\end{equation*}
where $\eta^1_{u,T}$ is supported on $(u+TD) \cap TD$ with $\|\eta^1_{T,u}\|_\infty \le c (|u|/T) T^{-d}  $, and $\eta^2_{T,u}$ is supported on the symmetric difference $(u+TD)\!\ominus\!TD$ with $\|\eta^2_{T,u}\|_\infty \le c T^{-d}$. Observing that, as $D$ is piece-wise smooth, $\vol((u+TD)\!\ominus\!TD)\le c |u| T^{d-1}$, we obtain for $T\ge |u|$
\begin{align*}
    \max_{i=1,2} \| \eta^i_{T,u} \|_{L^1} \le c |u|/T  \quad \text{and} \quad   \max_{i=1,2} \| \eta^i_{T,u} \|^2_{L^2} \le c (|u|/T) T^{-d},
\end{align*}
from which \eqref{e:decompeta2} follows by $\|\eta_{T,u} - \eta_{T,v}\|_{L^1} = \|\eta_{T,u-v}\|_{L^1} \le c|u-v|/T$. Next, applying convexity of the energy (Claim \ref{c:convex}), and Lemma \ref{l:enbound}, as $T \to \infty$
\begin{align}
 \nonumber \sup_{u \in B(T^{1-\eps})}  E [   \eta_T- \eta_{T,u}   ]  & \le   2  \sup_{u \in B(T^{1-\eps})}  E [\eta^1_{T,u}] + E[\eta^2_{T,u}]   \\
 \nonumber &  \le   \sup_{u \in B(T^{1-\eps})} c(|u|/T) / \capa_f(B(T))
\end{align}
from which \eqref{e:enbound3} readily follows by Corollary~\ref{c:capbounds}. 

We may now bound the terms of \eqref{e:erbound3}. Note that $\langle f, \eta_T- \eta_{T,u} \rangle$ is a centred Gaussian variable such that, by \eqref{e:enbound3},
\[   \sup_{u \in B(T^{1-\eps})}  \textrm{Var} [  
 \langle f, \eta_T- \eta_{T,u}  \rangle ]   = \sup_{u \in B(T^{1-\eps})}  E [   \eta_T- \eta_{T,u}   ] \le  T^{-\alpha -\eps + o(1)} . \]
Then the first term in \eqref{e:erbound3} is at most $e^{-  T^{\alpha +  \eps + o(1) } }$ by a standard tail estimate (Claim~\ref{c:gtb}).

Turning to the second term,  applying \eqref{e:decompeta2} we have
\begin{align*}
 \sup_{\substack{u,v \in B(T^{1-\eps}),\\ |u-v| \le \sqrt{d} T^{-\zeta}}}  | \langle f, \eta_{T,u}- \eta_{T,v} \rangle| \le  c  T^{-\zeta-1}
 \sup_{u \in B(T^{1-\eps})} |f| , 
 \end{align*}
 and so the second term  of \eqref{e:erbound3} is at most
 \[ \P \Big[c  T^{-\zeta-1}
 \sup_{u \in B(T^{1-\eps})} |f| \ge (\gamma/2) \sqrt{\log T} \Big]\]
Taking a union bound over translated copies of $B(1)$, and by stationarity, this is bounded from above by 
 \[ c T^d \P \Big[ \sup_{u \in B(1)} |f| \ge (\gamma/(2c)) T^{\zeta+1} \sqrt{\log T}\Big].\]
 Since $f$ is continuous and stationary, by the Borell-TIS inequality, the second term in \eqref{e:erbound3} is therefore at most $ e^{- T^{2(\zeta+1) + o(1) }} \le e^{-T^{d + o(1)}}$ completing the proof of \eqref{e:erbound2}.

 \medskip
\textbf{Concluding the proof.}
Gathering the estimates \eqref{e:erbound6}-- \eqref{e:erbound2}, and using Corollary \ref{c:capbounds} to obtain $\capa_f(B(T)) = T^{\alpha + o(1)}$, we deduce that
\begin{align*}
    & - \log \Big( \P \big[ f_{\mu_1} + f_{\mu_2}    \in  \per(\widetilde{T})  , E    \big] + 2 \P[F] \Big)   \\
    & \qquad \ge   \frac{1+o(1)}{2} \ell^2 (\log T)  \Big(  \capa_f(B(T) )  +   \min\Big\{ \frac{(\Delta-5\gamma/\ell)^2}{E[\eta_T]} ,  \capa_f(B(T) )\Big\} \Big) .
    \end{align*}
    Recalling \eqref{e:erbound1}, and that  $\ell = \sqrt{2 m (d-\alpha)} - 3 \gamma$, this concludes the proof of Proposition \ref{p:er2} by adjusting the constants $\gamma$ and $\Delta$.
\qed

\subsection{Adapting the argument to discrete fields}
\label{s:dc2}
For discrete fields we replace the spectral truncation function $\phi_{T^{1-\eps}}$ with its periodisation $\tilde{\phi}_{T^{1-\eps}}$, and the proof proceeds as in the continuous case, with the exception of: (i) as described in Section \ref{s:dc}, the event $A_1'$ must be redefined; and (ii) the argument to bound the second term of \eqref{e:erbound3} needs to be adapted, since it relied on the continuity. In regards to (ii), we simply replace the `fine mesh' $T^{-\zeta} \Z^d$ used to control \eqref{e:erbound2} with the lattice $\Z^d$, and then apply the same point-wise bound we used to control the first term of \eqref{e:erbound3}.

\subsection{Proof of auxiliary results}\label{s:stab}

\begin{proof}[Proof of Lemma~\ref{l:shape}]
 Without loss of generality, assume that $E[\eta]>0$ (as otherwise we have $\textrm{Var}[\langle f, \eta \rangle] = E[\eta] = 0$, implying $\P\left(| \langle f,\eta \rangle - \ell \langle h_\nu, \eta \rangle| \ge \ell \Delta   \right)=0$).
 We may also assume that 
 \begin{equation}
     \label{e:assumptdelta}
 \Delta \le \sqrt{\frac{E(\eta)}{E(\nu)}},
 \end{equation}
since the left-hand-side of \eqref{e:eraux1} is monotone increasing in $\Delta$ while the right-hand-side remains constant for $\Delta\ge \sqrt{\frac{E(\eta)}{E(\nu)}}$.
 %To ease notation we abbreviate $h = h_\nu$. 

Recall that the energy functional $E[\nu, \eta] = \int K(x-y)  d\nu(x) d\eta(y)$ defines an inner product, and also that $\langle h_\nu, \eta \rangle = E[\nu,\eta] / E[\nu]$ by \eqref{e:heta}. 
Fix $\Delta\in \R$ and define the event
\begin{align*}
A_{\Delta} &= \left\{  \inf_{x \in D} f(x) \ge \ell , \:\:  \left\langle  \ell h_\nu-f, \frac{\eta}{\Delta} \right\rangle \ge \ell   \right\}
\\ & 
= \left\{  \inf_{x \in D} f(x) \ge \ell , \:\: 
\ell\cdot\frac{E[\nu,\tfrac{\eta}{\Delta}]}{E[\nu]} -\left\langle f,\frac{\eta}{\Delta}\right\rangle \ge \ell   \right\}
\end{align*}

Note that the probability we seek to bound is $\P[A_\Delta] + \P[A_{-\Delta}]$. 
We introduce an auxiliary parameter $a\in \mathbb{R}$, such that
\begin{equation}\label{eq: cond on a}
\frac{E[\nu,\tfrac {\eta}{\Delta}]
}{E[\nu]} + a > 0.
\end{equation}
Since $\nu \in \mathcal{P}(D)$, we have $\langle f,\nu\rangle \ge \ell$ on the event $A_\Delta$, and therefore
\begin{align}
\P[A_\Delta] & 
=\P \left[ 
\inf_{ D} f \ge \ell , \:\: 
\ell \left( \tfrac{
E[\nu,\tfrac {\eta}{\Delta}]
}{E[\nu]} + a\right) - \langle f,\tfrac {\eta}{\Delta}\rangle \ge \ell (1+a)
\right] \notag
\\ &
\le \P \left[   
\langle f,\nu\rangle \left( \tfrac{
E[\nu,\tfrac {\eta}{\Delta}]
}{E[\nu]} + a\right) - \langle f,\tfrac {\eta}{\Delta}\rangle \ge \ell (1+a)
\right] \notag
\\ & = 
\P \left[  
\left\langle f, a\nu - \left( 
\frac{\eta}{\Delta}- \tfrac{
E[\nu,\frac {\eta}{\Delta}]
}{E[\nu]} \nu 
\right) \right\rangle \ge \ell (1+a)
\right]. \label{eq: AD first}
\end{align}
Because of the
orthogonality of $\nu$ and 
$\tfrac{\eta}{\Delta}- \tfrac{
E[\nu,\frac {\eta}{\Delta}]
}{E[\nu]} \nu $, we have
\begin{align*}
E \left[
 a\nu - \left( 
\frac{\eta}{\Delta}- \tfrac{
E[\nu,\frac {\eta}{\Delta}]
}{E[\nu]} \nu 
\right) 
\right]   &=a^2 E[\nu]
 + E\left[ 
 \frac {\eta}{\Delta} - 
\frac{E[\nu,\tfrac {\eta}{\Delta}]
}{E[\nu]}  \nu \right]\\&
= a^2 E[\nu] + E\left[\frac \eta \Delta\right]- 
\frac{(E[\nu,\tfrac \eta \Delta])^2}{E[\nu]}  
=E[\nu] (a^2 + \gamma),
\end{align*} 
where
\[
\gamma = \frac{E[\tfrac \eta \Delta]}{E[\nu]} -\frac{(E[\nu,\tfrac \eta \Delta])^2}{E[\nu]^2}
\le \frac{E[\tfrac \eta \Delta]}{E[\nu]}.
\]

Applying the Gaussian tail bound from Claim~\ref{c:gtb}, we deduce from \eqref{eq: AD first} that
\begin{equation}\label{eq: AD bd}
\P(A_\Delta) \le \frac 1 2 
\exp\left( - \frac {\ell ^2
(1+a)^2
}{ 2  E[\nu] (a^2 + \gamma) }\right)
\le \frac 1 2 
\exp\left( - \frac {\ell ^2}{2 E[\nu]}\cdot \frac{
(1+a)^2
}{ a^2 + \frac{E\big[\tfrac \eta \Delta\big]}{E[\nu]}  }\right).
\end{equation}
This bound is brought to minimum by the choice
$a = E[\tfrac{\eta}{\Delta}]/E[\nu]$, which indeed obeys condition~\eqref{eq: cond on a} 
since, by assumption \eqref{e:assumptdelta} on $\Delta$, we have 
\[
\big| E[\nu,\tfrac{\eta}{\Delta}]\big|
 \le \sqrt{E[\nu]E[\tfrac{\eta}{\Delta}]} \le E[\tfrac{\eta}{\Delta}].
\]
Plugging this choice of $a$ into \eqref{eq: AD bd}, we conclude that
\[ \P[A_\Delta] \le \frac 1 2 \exp\left(
-\frac{\ell^2}{2} \left(\frac 1 {E[\nu]} + \frac{\Delta^2}{E[\eta]} \right)
\right) .\]
As this bound holds regardless of the sign of $\Delta$, we conclude the proof.
\end{proof}

\begin{proof}[Proof of Lemma~\ref{l:stab}]
Let  $\nu \in \mathcal{P}(D)$ be such that $E[\nu] \capa_f(D) \le 2$, and let $\eta \in \mathcal{M}$. We shall prove that for all $\Delta > 0$ such that
\begin{equation}
    \label{e:stab1}
  \frac{1}{E[\nu]}   + \frac{\Delta^2}{E[\eta]} > \capa_f(D) , 
  \end{equation} 
we have 
\begin{equation}
    \label{e:stab2}
 \big| \langle h_\nu, \eta \rangle - \langle h_D ,\eta \rangle \big| < 2 \Delta  .
 \end{equation}
From this we may obtain the lemma by setting
\[ \Delta = \sqrt{ \frac{ ( E[\nu] \capa_f(D) -1)  E[\eta]}{E[\nu]} } + \eps,
\]
and observing that 
\[
\big| \langle h_\nu, \eta \rangle - \langle h_D ,\eta \rangle \big| <  2\Delta\le \sqrt{ ( E[\nu] \capa_f(D) -1) E[\eta] \capa_f(D)  } + \eps,\]
where the first inequality uses \eqref{e:stab2}, and the second uses the fact that 
$1/E[\nu] \le \capa_f(D)$, which follows from  definition \eqref{e:cap1}.
Taking $\eps \downarrow 0$, the lemma follows.

To show \eqref{e:stab2}, it would suffice to  find $x\in \R$
satisfying 
\begin{equation}\label{eq:trinagle eq}
|x - \langle h_\nu, \eta\rangle|+|x-\langle h_D,\eta\rangle |<2\Delta 
\end{equation}
as \eqref{e:stab2} would then follow using the triangle inequality. This we establish using a probabilistic method, showing that \eqref{eq:trinagle eq} holds with positive probability for the random variable $\frac{1}{\ell}\langle f,\eta \rangle$ conditioned on $\{\inf_{x \in D} f(x)>\ell\}$ for sufficiently large $\ell$.

Indeed, combining Lemmas \ref{l:infcap} and \ref{l:shape}, we obtain that for all  $\ell \ge \ell'$ we have
\begin{align*}
& - \log \P \left[ | \langle f,\eta \rangle - \ell \langle h_\nu, \eta \rangle| \ge \ell \Delta 
 \: \big| \:   \inf_{x \in D} f(x) \ge \ell  \right]   \\
 & \qquad \qquad \ge  \frac{\ell^2}{2} \left( \frac{1}{E[\nu]}   + \frac{\Delta^2}{E[\eta]}  \right)  - \frac{1}{p(\ell')} \left( \frac{(\ell-\ell')^2 \capa_f(D) }{2} + e^{-1} \right)  
 +\log p(\ell'),
 \end{align*}
 where  $p(\ell') =  \P[ \inf_{x \in D} f(x) \ge \ell'] $. By assumption~ \eqref{e:stab1}, and since $p(\ell') \to 1$ as $\ell' \to \infty$, we can choose $\ell$ and $\ell'$ sufficiently large so that 
 \[ \P \Big[ | \langle f,\eta \rangle - \ell \langle h_\nu, \eta \rangle| \ge \ell \Delta 
 \: \big| \:   \inf_{x \in D} f(x) \ge \ell  \Big]  <  1/2  .\]
 Since \eqref{e:stab1} also holds for the equilibrium measure, we may apply this to get
  \[ \P \Big[ | \langle f,\eta \rangle - \ell \langle h_D, \eta \rangle| \ge \ell \Delta 
 \: \big| \:   \inf_{x \in D} f(x) \ge \ell  \Big]  <  1/2  .\]
 Combining these, we see that indeed
 \[ \P \Big[ \left| \tfrac 1 \ell \langle f,\eta \rangle - \langle  h_\nu,  \eta \rangle\right|+ \left| \tfrac 1 \ell \langle f,\eta \rangle -  \langle h_D, \eta \rangle\right| <  2\Delta  
 \:\ \Big|\ \inf_{x \in D} f(x) \ge \ell  \Big]  >  0,\]
 from which \eqref{eq:trinagle eq} follows.
\end{proof}

 %%%%%%%%%%%
 %%%%%%%%%%%

%%%%%
\medskip
\section{Universal scaling of the capacity and equilibrium potential}
\label{s:rv}

In this section we study universality properties of the capacity and equilibrium potential. Recall from Section \ref{ss:uni} the $\alpha$-Riesz kernel $k_\alpha(x) = B_{\alpha,d}  |x|^{-\alpha}$, $\alpha \in [0,d)$, and its associated spectral measure $\mu_\alpha$, density $\rho_\alpha(s)  = A_{\alpha,d}|s|^{\alpha-d}$ (only if $\alpha > 0$), capacity $c_{\alpha,d} > 0$, and equilibrium potential $h_{\alpha,d}(\cdot)$. Let $\mathcal{M}_c^\infty \subset \mathcal{M}$ be the set of signed measures on $\R^d$ with compact support and bounded~density. Our main result is the following:

\begin{theorem}
\label{t:capuniv}
Let $\alpha \in [0,d)$, and let $f$ be an SGF on $\R^d$ or $\Z^d$ satisfying one of the conditions~\ref{item:1}, \ref{item:2}, \ref{item:3} or \ref{item:4} in Theorem~\ref{t: univ pro}.  Then
\begin{equation}
\label{e:univcap1}
\capa_f(B(T)) \sim \frac{c_{\alpha,d}}{\mu[B(\tfrac 1 T)]}, \quad \text{ as }T\to\infty .
\end{equation}
If, in addition, $\mu$ is origin dominant~\eqref{e:ds}, then for every $\eta \in \mathcal{M}_c^\infty$ (required to be radial in case~\ref{item:3}),
 \begin{equation}
     \label{e:univcap2}
  \lim_{T \to \infty}\langle h_T, \eta_T \rangle =  \langle h_{\alpha,d}, \eta \rangle  \in (-\infty,\infty) .  
 \end{equation}
\end{theorem}

\begin{remark}
\label{r:k2}
As in Remark \ref{r:k}, one can use Proposition \ref{p:taub} to rewrite~\eqref{e:univcap1} as follows:
\[
 \capa_f(B(T)) \sim  
 \begin{cases}
 c_{\alpha,d} \, B_{\alpha,d} K(T)^{-1} , & \text{case~\ref{item:1}}\\
c_{\alpha,d} \, A_{\alpha,d} T^{d} \rho(\tfrac 1 T)^{-1}, & \text{case~\ref{item:2}}.
  \end{cases}
\]
\end{remark}

Using this, we can now deduce Theorems \ref{t: univ pro} and \ref{t: univ ent} (and Proposition \ref{p:cond}), from Theorems \ref{t:main} and \ref{t:er} respectively:

\begin{proof}[Proof of Proposition \ref{p:cond} and Theorems \ref{t: univ pro} and \ref{t: univ ent}]
The first item of Proposition \ref{p:cond} is implied by Proposition \ref{p:taub}, and the second item follows from the first item and \eqref{e:univcap1}. 

Now assume the conditions of Theorem \ref{t: univ pro} hold. Then by Proposition \ref{p:cond}, the conditions of Theorem \ref{t:main} are satisfied. Theorem \ref{t: univ pro} follows by combining the conclusions of Theorems \ref{t:main} and \eqref{e:univcap1}. The proof of Theorem \ref{t: univ ent} is similar, combining Theorems \ref{t:er} and \eqref{e:univcap2}.
\end{proof}

In the remainder of the section we prove Theorem \ref{t:capuniv} following the approach to similar results in \cite{ams14,ms22}. To this end, we first present approximations of the Riesz kernel in Section~\ref{s:Rapprox}. The proof of Theorem~\ref{t:capuniv} 
 appears in Section~\ref{s:t7proof}; it is given for continuous fields. The differences to the discrete case are discussed in Section~\ref{s:fordiscrete}.

\subsection{The Riesz kernel and its smooth approximations}\label{s:Rapprox} 
We begin by recalling some properties of the Riesz kernel and its various smooth approximations. These will be used to justify the dominated convergence arguments that underpin the proof. 

Define the energy forms associated to the Riesz kernel
\begin{align}\label{eq: riesz}
E_\alpha[\nu, \eta] &=   B_{\alpha,d} \int \int |x-y|^{-\alpha} d\nu(x) d\eta(y) \notag \\
& = \begin{cases} A_{\alpha,d} \int \mathcal{F}[\nu](s)  \overline{\mathcal{F}[\eta](s)}  |s|^{\alpha-d} \, ds & \alpha \in (0,d), \\ \mathcal{F}[\nu](0)  \mathcal{F}[\eta](0) 
 & \alpha = 0 , \end{cases}
 \end{align}
and $ E_\alpha[\nu] = E_\alpha[\nu,\nu]$. By the classical theory of the Riesz kernel~\cite{lan72}, $E$ defines an inner product.
We define the equilibrium measure of the unit ball
$\nu_\alpha \in \mathcal{P}(B(1))$ satisfying
\[c_\alpha^{-1}  = \inf_{\nu \in \mathcal{P}(B(1))} E_\alpha[\nu] = E_\alpha[\nu_\alpha]  . \]  
For $\alpha>0$ it is uniquely defined, while for $\alpha=0$ every $\nu \in \mathcal{P}(B(1))$ is an equilibrium measure, and for concreteness we define $\nu_0$ to be the uniform measure on $B(1)$. The equilibrium potential $h_\alpha$ can be written as 
\[ h_\alpha = c_\alpha (k_\alpha * \nu_\alpha) , \]
so that $\langle h_\alpha,\eta \rangle = c_\alpha E[ \nu_\alpha,\eta]$. 
Explicit expressions for
$c_\alpha$, $h_\alpha$  and
$\nu_\alpha$ are known but unneeded here
(see Appendix~\ref{a:explicit} for $c_\alpha$, $h_\alpha$, and \cite[p.163]{lan72} for $\nu_\alpha$).

\smallskip
Next, we introduce two approximations of the Riesz capacity $c_\alpha$, defined by either smoothing the equilibrium measure $\nu_\alpha$ or the Riesz kernel $k_\alpha$.

\subsubsection{Smoothing the equilibrium measure}
Recall the spectral truncation function $\phi_s$ given in~\eqref{eq:stf}. For $\theta > 0$, define 
\[ \nu_\alpha^\theta(\cdot) = (\nu_\alpha * \mathcal{F}[\phi_\theta])( (1+\theta) \cdot)  \in \mathcal{P}(B(1)) .  \]
By the stretched-exponential decay of $\phi_1$ and its derivatives, $\mathcal{F}[\nu_\alpha^\theta](\lambda)$ is smooth, satisfying for every multi-index $k$,
\begin{equation}
\label{e:nudecay}
 | \partial^k \mathcal{F}[\nu_\alpha^\theta](\lambda) | \le  c_1 e^{-c_2|\lambda|^v}, 
\end{equation}
for some $c_1,c_2>0$ (that may depend on $k, \alpha$ and $\theta$).

The next lemma shows that, as $\theta \to 0$, this smoothing has negligible effect on the energy.
\begin{lemma}\label{lem:thetatozero}
\begin{equation} \label{e:thetatozero}
\lim_{\theta \to 0} E_\alpha[\nu_\alpha^\theta] = E_\alpha[ \nu_\alpha] = c_\alpha^{-1} \in (0,\infty).
\end{equation}
Moreover, for every $\eta \in \mathcal{M}_c^\infty$,
\begin{equation}
    \label{e:thetatozero1}
\lim_{\theta \to 0} E_\alpha[\nu_\alpha^\theta,\eta] = E_\alpha[ \nu_\alpha,\eta] \in (-\infty,\infty) .
\end{equation}
\end{lemma}
\begin{proof}
We first observe  that $E_\alpha[\eta] < \infty$ since $\eta$ is compactly supported with bounded density, which implies the finiteness of $E_\alpha[ \nu_\alpha,\eta]$ by the Cauchy-Schwarz inequality. Next, for $x \in B(1)$ define
\[ \Phi^1_\theta(x) = (k_\alpha * \mathcal{F}[\phi_\theta] )(x) = B_{\alpha,d} \int |x-s|^{-\alpha} \mathcal{F}[\phi_\theta](s) \, ds  \] 
and
\[ \Phi^2_\theta(x) = B_{\alpha,d} (k_\alpha * \mathcal{F}[\phi_\theta]  * \mathcal{F}[\phi_\theta] )(x) = B_{\alpha,d} \int \int |x-s-t|^{-\alpha} \mathcal{F}[\phi_\theta](s)  \mathcal{F}[\phi_\theta](t) \, ds dt . \]
Since $\mathcal{F}[\phi_\theta](s)\to \delta_0(s)$ as $\theta \to 0$, and using the diminishing support of $\mathcal{F}[\phi_\theta](s)$, for every $x \neq 0$,  we obtain, by dominated convergence,
\[ \Phi^1_\theta(x) \to  k_\alpha(x)  \quad \text{and} \quad \Phi^2_\theta(x) \to  k_\alpha(x). \]  
We can also check that, since $\mathcal{F}[\phi_1]$ is a probability measure with bounded density, there exists $c > 0$ depending only on $\alpha$, $d$, and $\phi_1$ for which 
\[  \Phi^1_\theta(x), \Phi^2_\theta(x)  \le  c |x|^{-\alpha} = c k_\alpha(x), \]
uniformly in $\theta$ and $x$. Hence, applying dominated convergence once again, we obtain
\[ \lim_{\theta \to 0} E_\alpha[\nu_\alpha^\theta] = \lim_{\theta \to 0} \int \int \Phi_\theta^2(x-y) \, d\nu_\alpha(x)  d\nu_\alpha(y) =  E_\alpha[\nu_\alpha] , \]
and
\[ \lim_{\theta \to 0} E_\alpha[\nu_\alpha^\theta,\eta] = \lim_{\theta \to 0} \int \int \Phi_\theta^1(x-y) \, d\nu_\alpha(x)  d\eta(y) =  E_\alpha[\nu_\alpha,\eta] \]
as required.
\end{proof}

\subsubsection{Smoothing the kernel} 
The second approximation involves smoothing the Riesz kernel. For $\theta > 0$, recall the truncated measure \[\mu_\alpha^\theta = \phi_\theta^2 \mu_\alpha.\] In the case $\alpha \in (0,d)$ this has density $A_{\alpha,d} \phi_\theta^2(\lambda) |\lambda|^{\alpha-d}ds$, whereas if $\alpha = 0$ then $\mu_\alpha^\theta = \mu_\alpha = \delta_0$. We then define the \emph{smoothed Riesz kernel} 
\[k_\alpha^\theta := \mathcal{F}[\mu_\alpha^\theta] = k_\alpha * \mathcal{F}[\phi_\theta] *  \mathcal{F}[\phi_\theta].\] Similarly to \eqref{e:thetatozero}, 
this smoothing has negligible effect on the capacity:

\begin{lemma}
 \label{l:captruncriesz}
 For every $s>0$,
\[ (1+\theta)^\alpha  c_\alpha \ge \capa_{\mu_\alpha^\theta}(B(1)) \ge   c_\alpha .\]
Consequently,
\begin{equation*}
%\label{e:thetatozero2}
\lim_{\theta \to 0} \capa_{\mu_\alpha^\theta}(B(1))  = c_\alpha.
\end{equation*}
\end{lemma}

\begin{proof}
Since $\mu_\alpha^\theta \le \mu_\alpha$, we have 
$\capa_{\mu_\alpha^\theta}(B(1)) \ge   c_\alpha$ by monotonicity of the capacity (Claim~\ref{c:mon}). For the other bound, the same proof as in Lemma \ref{l:captrunc} shows that, for every $0 <\theta  < T$,
\[ \capa_{\mu_\alpha}(B(T)) \ge \capa_{\mu_\alpha^\theta}(B(T-\theta)) . \]
Then the bound follows by setting $T = \theta+1$ and using that $\capa_{\mu_\alpha}(B(T)) = T^\alpha c_\alpha$ by the scale invariance of the Riesz kernel.
  \end{proof}

\subsection{Proof of Theorem \ref{t:capuniv}}\label{s:t7proof}
First, let us rewrite the conditions of the theorem more explicitly.
Conditions~\ref{item:1}, \ref{item:2} and \ref{item:3} guarantee 
the existence of a function 
$w:(0,\infty)\to(0,\infty)$ which is slowly varying at $\infty$, obeying, respectively, one of the following conditions:
\begin{enumerate}[{\rm (i')}]

\item \label{item:1w}
$
K(x)  \sim B_{\alpha,d} |x|^{-\alpha} w(|x|), \text{  as  } x\to\infty$.

\item \label{item:2w}
$\rho(\lambda) \sim A_{\alpha,d} |\lambda|^{\alpha-d} w(1/|\lambda|)$, as $\lambda \to 0$, where $\alpha>0$ and $\rho$ is the density of the absolutely continuous part of $\mu$.
\item\label{item:3w} $\mu[B(1/T)] \sim T^{-\alpha} w(T)$, as $T\to\infty$, and $\mu$ is radial.
\end{enumerate}
Without loss of generality we are free to modify $w$ in a compact set around the origin. Under condition $(iv)$ we set $w \equiv 1$.

By Proposition \ref{p:taub} of the appendix, we observe that, in all cases,
\begin{equation}\label{e:common asym} \mu[B(1/T)] \sim T^{-\alpha} w(T), \quad\text{ as  } T\to\infty.
\end{equation}

We shall reduce the theorem to the following statements about convergence of energies, first for the smoothed equilibrium measures $\nu_\alpha^\theta$ and then for the equilibrium measure $\nu_{B(T)}$, both normalised using $w$.

\begin{proposition}\label{prop: for univ}
Under each of the conditions~\ref{item:1}, \ref{item:2}, \ref{item:3}, \ref{item:4}, 
we have:

\begin{enumerate}[{\rm (a)}]
\item\label{item:a}
For every $\theta>0$ and $\eta \in \mathcal{T}\subseteq \mathcal{M}_c^\infty$ (required to be radial in case~\ref{item:3}), we have
\begin{equation}
\label{e:rvsuff2}
\lim_{T\to\infty} \frac{E[ \nu_\alpha^\theta(\cdot/T), \eta_T ] }{T^{-\alpha} w(T)}
=E_\alpha[\nu_\alpha^\theta,\eta]. 
\end{equation}
In particular,
\begin{equation}
\label{e:rvsuff1}
 \lim_{T\to\infty}\frac{ E[ \nu_\alpha^\theta(\cdot/T) ] }{T^{-\alpha} w(T)} = E_\alpha[\nu_\alpha^\theta].
\end{equation}
\item\label{item:b}
\begin{equation*}%\label{eq:enu}
\lim_{T\to\infty}\frac{ E[\nu_{B(T)}] }{T^{-\alpha} w(T)}  = E_\alpha[\nu_\alpha].
\end{equation*}
\end{enumerate}
\end{proposition}

\begin{proof}[Proof of Theorem \ref{t:capuniv} from Proposition~\ref{prop: for univ}]

By definition of the capacity \eqref{e:cap1} and the asymptotics~\eqref{e:common asym}, we have 
\begin{equation*}  
\frac{ E[\nu_{B(T)}] }{T^{-\alpha} w(T)} 
\sim
\frac{\capa_f(B(T))^{-1}}{\mu(B(\frac 1 T))}.
\end{equation*}
By Proposition~\ref{prop: for univ}--\ref{item:b}, the left-hand-side converges to $c_{\alpha,d}^{-1} = E_{\alpha}(\nu_\alpha)$, which yields ~\eqref{e:univcap1}.

\medskip
For \eqref{e:univcap2} we assume further that 
\eqref{e:ds} holds. By \eqref{e:common asym}, also \eqref{e:doub} holds. We apply Lemma \ref{l:stab} with $D = B(T)$, $\nu =  \nu_\alpha^\theta(\cdot/T)$, and $\eta \mapsto \eta_T$, which gives, assuming $E[\nu_\alpha^\theta(\cdot/T)] \capa_f(B(T))  \le 2$,
\[ \big| \langle h_T, \eta_T \rangle  -  \langle h_{ \nu_\alpha^\theta(\cdot/T) }, \eta_T \rangle   \big|  \le  2 \sqrt{ \big(E[\nu_\alpha^\theta(\cdot/T)] \capa_f(B(T))  - 1  \big) E[\eta_T] \capa_f(B(T) ) } . \]
Since we assume \eqref{e:doub}--\eqref{e:ds}, by the first statement of Corollary \ref{c:enbound}, $E[\eta_T] \capa_f(B(T))$ is bounded over $T \ge 1$.  Moreover, by 
both parts of Proposition~\ref{prop: for univ}, 
\[ 
E[\nu_\alpha^\theta(\cdot/T)] \capa_f(B(T)) =  \frac{ E[\nu_\alpha^\theta(\cdot/T)]}{T^{-\alpha} w(T)} \cdot \frac{T^{-\alpha} w(T)}{E[\nu_{B(T)}]} \overset{T\to\infty}{\longrightarrow}\frac{E_\alpha[\nu_\alpha^\theta]}{E_\alpha[\nu_\alpha]}. 
\]
Hence we deduce that 
\[ \limsup_{T \to \infty} \big| \langle h_T, \eta_T \rangle  -  \langle h_{\nu_\alpha^\theta(\cdot/T) }, \eta_T \rangle   \big|  \le c \sqrt{ \frac{E_\alpha[\nu_\alpha^\theta]}{E_\alpha[\nu_\alpha]}  - 1}  \]
for some $c > 0$ independent of $\theta$. To complete the proof we observe that
 \[ \langle h_{\nu_\alpha^\theta(\cdot/T) }, \eta_T \rangle  = \frac{E[\nu_\alpha^\theta(\cdot/T), \eta_T ] }{E[\nu_\alpha^\theta(\cdot/T)]}  \overset{T\to\infty}{\longrightarrow}  \frac{E_\alpha[\nu^\theta_\alpha,\eta] }{E_\alpha[\nu^\theta_\alpha] }   ,  \]
 where we used Proposition~\ref{prop: for univ}--\ref{item:a}. Letting $\theta \to 0$ and combining with Lemma~\ref{lem:thetatozero} yields
 \[ \lim_{T\to\infty}\langle h_T, \eta_T \rangle  =\frac{E_\alpha[\nu_\alpha,\eta]}{E_\alpha[\nu_\alpha]} = \langle h_{\alpha,d},\eta\rangle     , \]
 which is precisely \eqref{e:univcap2}.
\end{proof}

The rest of the section is devoted to the proof of Proposition~\ref{prop: for univ} under its various conditions. We shall use the following technical lemma.

\begin{lemma}\label{lem: use Potter}
Let $w:(0,\infty)\to(0,\infty)$ be slowly varying at $\infty$. Let $g:\R^d\to\R$ be a function such that there exists $\ep>0$ for which 
\[\int_{\R^d}\max(|x|^\ep,|x|^{-\ep}) |g(x)| \: dx < \infty.\]
Let $k_T:\R^d\to\R$ be a family of uniformly bounded functions such that for all $z\in \R^d$, $k(z):=\lim\limits_{T\to\infty} k_T(z)$. 
Then 
\[
\lim_{T\to\infty }
\int g(z) k_T(z) \frac{w(T|z|)}{w(T)} \: dz =
\int g(z) k(z) \: dz.
\]
\end{lemma}

\begin{proof}
Since $w$ is slowly varying, we have $\lim_{T\to\infty} \frac{w(T|x|)}{w(T)}=1$. Moreover, by Potter's bounds (see \cite[Theorem 1.5.6-(ii)]{bgt87}), for any fixed $\ep>0$ we have
\[  
\exists C>0 \: \forall z\in\R^d: \quad \frac{w(T|z|)}{w(T)}   \le C\max(|z|^{-\ep},|z|^{\ep}).
\]
The lemma follows by a simple application of the dominated convergence theorem.
\end{proof}

\subsubsection{Proof of Proposition \ref{prop: for univ} under condition \ref{item:1}}\label{sec: prop for univ item 1}

By the definition of energy,
\begin{equation*}%\label{eq:e_nu_theta}
\frac{E[\nu^\theta_\alpha(\cdot/T)]}{T^{-\alpha}w(T)} = 
\iint \frac{K(T(x-y))}{T^{-\alpha} w(T)} 
    d\nu^\theta_\alpha(x) \,d\nu^{\theta}_\alpha(y).
\end{equation*}
Fix $z$ and consider the decomposition
\begin{equation}\label{eq: 3parts}
\frac{ K(Tz)} {T^{-\alpha} w(T)} = 
|z|^{-\alpha} \cdot \frac{K(Tz)}{(T|z|)^{-\alpha}\, w(T|z|)} \cdot \frac{w(T|z|)}{w(T)}.    
\end{equation}
By assumption~\ref{item:1w} we have $
\lim\limits_{T\to\infty} \frac{K(Tz)}{(T|z|)^{-\alpha}\, w(T|z|)}  = B_{\alpha,d}$, which yields the pointwise limit
\begin{equation}\label{e:limit}
\lim_{T\to\infty}\frac{K(Tz)}{T^{-\alpha}w(T)}= B_{\alpha,d}|z|^{-\alpha}.
\end{equation} 
Moreover, by modifying $w$ so that $|x|^{-\alpha} w(|x|) \to 1$ as $|x| \to 0$, we obtain that $\frac{K(x)}{|x|^{-\al}w(|x|)}$ is a bounded function.
Recalling that also $\nu^\theta_\al$ is compactly supported and    has bounded density, we may use Lemma~\ref{lem: use Potter} with any $\eps \in (0,d-\alpha)$ to get
\begin{align}
\label{e:unikerdom}
    \lim_{T\to\infty} \iint 
    \frac{K(T(x-y))}{T^{-\alpha} w(T)} %\id_{\{|x-y| \ge \tfrac{t_0}{T}\}}
    d\nu^\theta_\alpha(x) \,d\nu^{\theta}_\alpha(y)
    =  B_{\alpha,d} \iint |x-y|^{-\alpha} \, d\nu_\alpha^\theta(x) d\nu_\alpha^\theta(y) =E_\alpha[\nu^\theta_\alpha].
 \end{align}
A similar application of Lemma~\ref{lem: use Potter} yields 
 \[
 \lim_{T\to\infty} \frac{E[ \nu_\alpha^\theta(\cdot/T), \eta_T ] }{T^{-\alpha} w(T)}
=E_\alpha[\nu_\alpha^\theta,\eta],
 \]
which proves part~\ref{item:a}.

\medskip
By the definition of energy,
\begin{equation*}
\frac{ E[\nu_{B(T)}] }{T^{-\alpha} w(T)} 
=\inf_{\nu \in \mathcal{P}(B(1))}  \iint \frac{ K(T(x-y))} {T^{-\alpha} w(T)} \, d\nu(x) d\nu(y)  . 
\end{equation*}
Thus item~\ref{item:b} reduces to showing 
\begin{equation}\label{e:univcap4}
     \lim_{T \to \infty}  \inf_{\nu \in \mathcal{P}(B(1))}  \iint \frac{ K(T(x-y))} {T^{-\alpha} w(T)} \,d\nu(x) d\nu(y)  = E_{\alpha}[\nu_\alpha].
     %c_\alpha^{-1} . 
\end{equation}
We use a decomposition appearing in \cite[Proposition C.6]{ms22}\footnote{Although the statement of \cite[Prop. C.6]{ms22} assumed isotropy, one can check that this is not needed in the proof. Also \cite[Prop. C.6]{ms22} claims that $K_1 \ge 0$, but the proof gives $K_1 > 0$.}: there exist non-negative measures $\mu_1$ and $\mu_2$ such that $\mu = \mu_1 + \mu_2$, and $K_1=\mathcal{F}[\mu_1]$ obeys $K_1 > 0$ and $K_1(x) \sim K(x)$ as $x \to \infty$. By monotonicity (Claim~\ref{c:mon}), for all $T$ we have \[  \capa_\mu(B(T))\le  \capa_{\mu_1}(B(T))  , \]
which by the definition of capacity~\eqref{e:cap1} translates into
 \[
 \inf_{\nu \in \mathcal{P}(B(1))}  \iint \frac{ K(T(x-y))} {T^{-\alpha} w(T)} \,d\nu(x) d\nu(y)
 \ge 
 \inf_{\nu \in \mathcal{P}(B(1))}  \iint \frac{ K_1(T(x-y))} {T^{-\alpha} w(T)} \,d\nu(x) d\nu(y).
  \]
Recall the pointwise limit ~\eqref{e:limit}, which also holds for $K_1$. In fact for $K_1$ we can upgrade \eqref{e:limit} to 
\begin{equation*}
\lim_{T\to\infty, z_T \to z}\frac{K_1(Tz_T)}{T^{-\alpha}w(T)}= B_{\alpha,d}|z|^{-\alpha}.
\end{equation*} 
For $z \neq 0$, this follows from the uniform convergence theorem for regularly varying functions \cite[Theorem 1.5.2]{bgt87}. For $z = 0$, this is since $K_1$ is asymptotic to an eventually decreasing function (see Remark \ref{r:taub}) and is strictly positive. Then for any $T_n \to \infty$ and $\nu_n \in \mathcal{P}(B(1))$ which converge weakly to $\nu_\infty \in \mathcal{P}(B(1))$, we have, by Fatou's lemma for weakly convergent probability measures~\cite{fkz14},
\[ \liminf_{n \to \infty}  \iint \frac{ K_1(T_n(x-y))} {T_n^{-\alpha} w(T_n)} \, d\nu_n(x) d\nu_n(y) 
\ge B_{\alpha,d}\iint |x-y|^{-\alpha}
d\nu_\infty(x)d\nu_\infty(y)\ge E_{\alpha}[\nu_\alpha], \]
where the last inequality uses the minimality of $c_{\alpha}^{-1}$.
To conclude the proof of \eqref{e:univcap4}, we are thus left with showing 
\begin{equation}\label{e:capup} \limsup_{T \to \infty}  \inf_{\nu \in \mathcal{P}(B(1))}  \iint \frac{ K(T(x-y))} {T^{-\alpha} w(T)} \,d\nu(x) d\nu(y)   \le  E_{\alpha}[\nu_\alpha] . \end{equation}
To this end, we note that for any $\theta>0$ we have $\nu_\alpha^\theta\in\mathcal{P}(B(1))$, and hence
\begin{equation*}%\label{e:totheta} 
\inf_{\nu\in \mathcal{P}(B(1))} \iint \frac{K(T(x-y))}{T^{-\alpha}  w(T) } \, d\nu(x) d\nu(y)\le 
\iint \frac{K(T(x-y))}{T^{-\alpha}  w(T) } \, d\nu_\alpha^\theta(x) d\nu_\alpha^\theta(y). \end{equation*}
Taking $T\to\infty$ (by part~\ref{item:a}) and then $\theta\to 0$ (using Lemma~\ref{lem:thetatozero}) %in~\eqref{e:totheta} 
establishes our goal~\eqref{e:capup}.  
\qed

\subsubsection{Proof of Proposition~\ref{prop: for univ} under condition \ref{item:2}}
Let $\mu = \mu_{a.c.}\id_{B(a)} + \mu'$, 
where $\mu' = \mu_{a.c.}\id_{B(a)^c} + \mu_s$, and let $\rho$ denote the density of $\mu_{a.c.}\id_{B(a)}$.  
As described in~\ref{item:2w}, we assume that $\alpha>0$ and $\rho(\lambda) \sim A_{\alpha,d} |\lambda|^{\alpha-d} w(1/|\lambda|)$ as $\lambda \to 0$, where $w$ is slowly varying as $\lambda \to 0$.
The number $a$ is chosen such that 
$\frac{\rho(\lm) }{|\lm|^{\al-d} w(1/|\lm|)}$ is bounded on $B(a)$.

After a Fourier transform and change of variables,
\begin{equation}\label{e:rvsuff3}
  \frac{ E[\nu_\alpha^\theta(\cdot/T)] }{T^{-\alpha} w(T)}  
   =   \int_{|\lambda| \le a T} |\mathcal{F}[\nu^\theta_\alpha](\lm)|^2\cdot \frac{  \rho(\lm/T)}{T^{d-\alpha} w(T)}   \, d\lambda +  \frac{ E_{\mu'}[\nu_\alpha^\theta(\cdot/T)]}{T^{-\alpha} w(T)}. 
 \end{equation}
Fix $\lm>0$ and consider the decomposition
\[
 \frac{\rho(\lm/T)}{T^{d-\alpha} w(T)} =
 |\lambda|^{\alpha-d} \cdot \frac{ \rho(\lambda/T)  }{ (|\lambda|/T)^{\alpha-d} w(T/|\lambda|)} \cdot \frac{ w(T/|\lambda|) }{ w(T)}.
\] 
Assumption~\ref{item:2w} yields the pointwise limit
 \begin{equation}\label{e:limit2}
 \lim_{T\to\infty}\frac{ \rho(\lm/T)}{T^{d-\alpha} w(T)} = A_{\alpha,d} |\lm|^{\alpha-d}.
 \end{equation}
Recalling that $|\mathcal{F}[\nu^\theta_\alpha](\lambda)|^2$ is non-negative and integrable, we may apply Lemma~\ref{lem: use Potter} (with some $\eps \in (0,\alpha)$) to get
\begin{align}
    \label{e:unidendom}
\lim_{T\to\infty} \int_{|\lambda| \le aT} |\mathcal{F}[\nu^\theta_\alpha](\lambda)|^2 \frac{\rho(\lm/T)}{T^{d-\alpha} w(T) }\,  d\lambda 
= A_{\alpha,d} \int |\mathcal{F}[\nu^\theta_\alpha](\lambda)|^2  |\lambda|^{\alpha-d} \, d\lambda  = E_\alpha[ \nu_\alpha^\theta] .
\end{align}
We are left with showing that the second term in~\eqref{e:rvsuff3} vanishes as $T\to\infty$. Indeed, recalling that $\mu_s$ is supported in $B(r)^c$ for some $r>0$, we have
\[ \frac{E_{\mu'}[\nu_\alpha^\theta(\cdot/T)] }{T^{-\alpha} w(T)}
= \frac{\int_{|\lm|>\min\{1,r\} }|\mathcal{F}[\nu^\theta_\alpha]|^2(T\lm) d\mu'(\lm) }{T^{-\alpha} w(T) }
\le O(e^{-c T^v}), \quad \text{as  } T\to\infty, \]
where $c,v>0$ are constants coming from the rapid decay of $\mathcal{F}[\nu_\alpha^\theta]$, as stated in~\eqref{e:nudecay}. A similar application of Lemma~\ref{lem: use Potter} establishes that \[\lim_{T\to\infty} \frac{E[ \nu_\alpha^\theta(\cdot/T), \eta_T ] }{T^{-\alpha} w(T)}=\lim_{T\to\infty} \frac{\int \mathcal{F}[\nu^\theta_\alpha](T\lm) \overline{\mathcal{F}[\eta](T\lm)} d\mu(\lm) }{T^{-\alpha} w(T) }
=E_\alpha[\nu_\alpha^\theta,\eta],\] using that $|\mathcal{F}[\eta]|$ is bounded.
This establishes part~\ref{item:a}. 

\medskip
We turn to prove part~\ref{item:b}.
By definition~\eqref{e:cap1}, 
\begin{align*}
 \frac{E[\nu_{B(T)}]}{T^{-\alpha} w(T)} & = \frac{1}{T^{-\alpha}  w(T)}   \min_{\nu \in \mathcal{P}(B(T))}  \int |\mathcal{F}[\nu]|^2  d\mu \\
 & \ge \frac{1}{T^{-\alpha}  w(T)}   \min_{\nu \in \mathcal{P}(B(T))}  \int |\mathcal{F}[\nu]|^2  d\mu_{a.c.} \\
 & = \frac{1}{T^{d-\alpha}  w(T)}   \min_{\nu \in \mathcal{P}(B(1))}  \int |\mathcal{F}[\nu](\lambda)|^2  \rho(\lambda/T) \, d\lm  .
 \end{align*} 
Recall the pointwise limit~\eqref{e:limit2}. For any $T_n \to \infty$ and $\nu_n \in \mathcal{P}(B(1))$ which converge weakly to $\nu_\infty \in \mathcal{P}(B(1))$, we have, by Fatou's lemma
\begin{align*} 
\liminf_{n \to \infty}  &\frac{1}{T_n^{d-\alpha}  w(T_n)}  \int |\mathcal{F}[\nu_n](\lambda)|^2  \rho(\lambda/T_n) \, d\lambda  
 \ge 
A_{\alpha,d}\int |\lm|^{\alpha-d}|\mathcal{F}[\nu_\infty](\lambda)|^2   d\lm
\ge E_{\alpha}[\nu_\alpha].
%c_\alpha^{-1} . 
\end{align*}
We conclude that 
\[
\liminf_{T\to\infty}\frac{E[\nu_{B(T)}]}{T^{-\alpha}w(T)}\ge E_{\alpha}[\nu_{\alpha}].
\]
The reverse inequality is precisely~\eqref{e:capup}, which has been shown to follow from item~\ref{item:a} and Lemma~\ref{lem:thetatozero}. We thus conclude the proof of item~\ref{item:b}.

\subsubsection{Proof of Proposition \ref{prop: for univ} under condition \ref{item:3}}
 As described in~\ref{item:3w}, we assume that $\mu$ is radial and $\mu[B(1/T)] \sim T^{-\alpha} w(T)$, where $w$ is slowly varying as $T \to \infty$. Moreover, arguing as in the previous case, we may assume that $\mu$ is supported on $B(1)$.
Abbreviate $G(\delta) = \mu[B(\delta)]$. We write:
\begin{align*}
E[\nu_\alpha^\theta(\cdot/T)] 
& = 
\int |\F[\nu_\al^\theta]|^2(T\lm) d\mu(\lm) & \gray{}
\\ & = \int_0^\infty |\F[\nu_\al^\theta]|^2(T\lm) dG(\lm)  & 
\text{(radiality)}
\\ & = \int _0^\infty -\tfrac{d}{d\lm}\left(|\F[\nu_\al^\theta]|^2(T\lm)\right) G(\lm) d\lm & \text{(integration by parts)}
\\ & =\int_0^\infty -(|\F[\nu_\al^\theta]|^2)'(\lm) \, G(\tfrac \lm T) d\lm,  &\text{(change of variable)}
\end{align*}
where to justify the integration-by-parts
we use that $\mathcal{F}[\nu_\al^\theta]$ is smooth and rapidly decaying by~\eqref{e:nudecay}, and also that $G(\lambda) \in [0,1]$ with $G(0) = 0$. 
Thus
\begin{align} 
\notag \frac{ E[\nu_\alpha^\theta]}{T^{-\alpha} w(T) }  
& = 
\frac{\int_0^\infty -( |\mathcal{F}[\nu^\theta_\alpha]|^2)'(\lambda) G(\tfrac \lm T) d\lm} {T^{-\al}w(T)} \\ &=   \int_0^\infty -( |\mathcal{F}[\nu^\theta_\alpha]|^2)'(\lambda)  \cdot \lambda^\alpha \cdot \frac{ G(\tfrac {\lambda}{ T})}{ \left(\tfrac{\lm}{T}\right)^{\alpha} \,w(\tfrac{T}{\lambda})  } \cdot \frac{w(\tfrac T \lambda)}{w(T)} \, d\lambda. 
\label{e:decomp3}
\end{align}
By assumption~\ref{item:3w}, we have 
\begin{equation}\label{e:limit3}
\lim\limits_{T\to\infty}\frac{ G\left(\tfrac {\lambda}{ T}\right)}{ \left(\tfrac{\lm}{T}\right)^{\alpha}  \,w\left(\tfrac{T}{\lambda}\right)  } =1 \qquad \text{and} \qquad  \sup_{\delta < 1} \frac{G(\delta)}{\delta^\alpha w(1/\delta) } < \infty . \end{equation}

Recalling again the smoothness and rapid decay of derivatives of $\mathcal{F}[\nu_\al^\theta]$ given in~\eqref{e:nudecay}, we apply Lemma~\ref{lem: use Potter} with $\ep\in (0,1)$ to get
\begin{align*}
\lim_{T\to\infty} \frac{E[\nu_\al^\theta]}{T^{-\al} w(T)} 
& =\int_0^\infty -(|\F[\nu_\al^\theta]|^2)'(\lm) \cdot \lm^\al \, d\lm 
\\ & =  \begin{cases}
     A_{\al,d}\int |\F[\nu_\al^\theta]|^2 (\lm) \cdot |\lm|^{\al-d} \, d\lm  &  \alpha \in (0,d)  \\
 |\F[\nu_\al^\theta]|^2(0) & \alpha = 0  
    \end{cases} \\
    & = E_\al[\nu_\al^\theta],
\end{align*}
where the second equality is another integration by parts. 
This establishes \eqref{e:rvsuff1}. The proof of \eqref{e:rvsuff2} is similar, using that $\mathcal{F}[\eta]$ is radial and bounded. We have thus proved part~\ref{item:a}.

\medskip
We turn to prove part~\ref{item:b}. 
Let $\mathcal{P}_{rad}(B(T))$ be the subset of $\mathcal{P}(B(T))$ which are radial. We have
\begin{align*}
E[\nu_{B(T)}] & = \inf_{\nu \in \mathcal{P}_{rad}(B(T))}  \int |\mathcal{F}[\nu]|^2(\lambda)  \, d\mu(\lambda)  &\text{(Claim \ref{c:iso})}
 \\ & \ge   \inf_{\nu \in \mathcal{P}_{rad}(B(T))}  \int |\mathcal{F}[\nu]|^2(\lambda)  \phi_{\theta T}^2(\lambda) \, d\mu(\lambda)
 &(\phi_s\le 1,   \text{by Prop.~\ref{c:1})}
 \\
 & =  \inf_{\nu \in \mathcal{P}_{rad}(B(T))}  \int_0^\infty -( |\mathcal{F}[\nu]|^2 \phi_{\theta T}^2)'(\lambda) G(\lambda) \, d\lambda 
 & \text{(integration by parts)}
 \\ &= \inf_{\nu \in \mathcal{P}_{rad}(B(1))}  \int_0^\infty -( |\mathcal{F}[\nu]|^2 \phi_\theta^2  )'(\lambda)   G(\tfrac{\lambda}{T})  \, d\lambda & \text{(change of variable)}
\end{align*} 
where to justify the integration-by-parts
we use that $\mathcal{F}[\nu] \in C^\infty$, that $\phi_s^2$ and its derivatives have rapid decay (by Prop.~\ref{c:1}), and that $G(\lambda) \in [0,1]$ with $G(0) = 0$. 

\smallskip
 For any $T_n \to \infty$ and $\nu_n \in \mathcal{P}(B(1))$ which converge weakly to $\nu_\infty \in \mathcal{P}(B(1))$, we have, by another application of Lemma~\ref{lem: use Potter},
\begin{equation}\label{eq: dom limit3} 
\lim_{n \to \infty}   \, \frac{ \int_0^\infty -( |\mathcal{F}[\nu_n]|^2 \phi_\theta^2)'(\lambda)   G(\tfrac{\lambda}{T_n})  \, d\lambda }{T_n^{-\alpha}  w(T_n)} 
 = 
 \int_0^\infty -(|\F[\nu_\infty]|^2 \phi_\theta^2))'(\lm) \cdot \lm^\al \, d\lm.
\end{equation}
(Unlike the previous cases, we cannot apply Fatou's lemma here, since $-(|\mathcal{F}[\nu_n](\lambda)|^2 \phi_\theta^2(\lambda))'$ is not necessarily positive). To justify \eqref{eq: dom limit3}, recall the decomposition in~\eqref{e:decomp3} and the bounds in~\eqref{e:limit3}. Moreover, note that
$(|\mathcal{F}[\nu_n](\lambda)|^2\phi_\theta^2(\lambda)) ' \to (\mathcal{F}[\nu_\infty](\lambda)|^2 \phi_\theta^2(\lambda))'$ since $\nu_n$ is compactly supported and converges weakly to $\nu_\infty$. We also have the uniform decay estimate
\[ (|\mathcal{F}[\nu_n](\lambda)|^2 \phi_\theta^2(\lambda))' \le c_1 e^{ - c_2 |\lambda|^v }\]
since $|\mathcal{F}[\nu_n]|^2$ and its derivatives are uniformly bounded and $\phi_\theta$ and its derivatives have rapid decay. This justifies~\eqref{eq: dom limit3}, from which we conclude that
\begin{align*} \liminf_{T\to\infty}\frac{E[\nu_{B(T)}]}{T^{-\alpha}w(T)}
 &\ge 
 \int_0^\infty -(|\F[\nu_\infty]|^2 \phi_\theta^2))'(\lm) \cdot \lm^\al \, d\lm
 \\ &=
 \begin{cases}
 A_{\alpha,d}\int |\mathcal{F}[\nu_\infty]|^2(\lambda) \phi_\theta^2(\lm) \cdot |\lm|^{\alpha-d} d\lm & \alpha \in (0,d) \\
  |\mathcal{F}[\nu_\infty]|^2(0)  & \alpha = 0 
 \end{cases} 
 \\&  \ge \left(\capa_{\mu^\theta_\al}(B(1))\right)^{-1}
\end{align*}
Taking $\theta\to 0$, by Lemma~\ref{l:captruncriesz} we conclude that 
\[
\liminf_{T\to\infty}\frac{E[\nu_{B(T)}]}{T^{-\alpha}w(T)}\ge c_\al^{-1} = E_{\alpha}[\nu_{\alpha}].
\]
The reverse inequality is precisely~\eqref{e:capup}, which has been shown in Section~\ref{sec: prop for univ item 1} to follow from item~\ref{item:a} and Lemma~\ref{lem:thetatozero}. We thus conclude the proof of item~\ref{item:b}.

\subsubsection{Proof of Proposition~\ref{prop: for univ} under condition~\ref{item:4}}
In this case $\alpha = 0$, $\mu$ has an atom at the origin of mass $a$, and $w\equiv a$. Recall that the energy $E_0$ can be computed by~\eqref{eq: riesz}. Denoting by $\mu'$ the measure without the atom at $0$, we have
\begin{align*}
\frac{E[\nu_0^\theta\left(\cdot/T\right)]}{w(T)} &=
a^{-1}  \int \left|\mathcal{F}[\nu_0^\theta](T\lm) \right|^2 d\mu
\\ & = a^{-1}\cdot a|\mathcal{F}[\nu_0^\theta](0)|^2 + O\left( \int e^{-cT^\nu \lm^\nu} d\mu'(\lm) \right) = E_0[\nu_0^\theta] + o(1),
\end{align*}
as $T\to\infty$. This proves part~\ref{item:a}. 
For part~\ref{item:b}, we have:
\begin{align*}
\frac {E[\nu_{B(T)}]}{w(T)}  &= a^{-1}\cdot\min_{\nu\in \mathcal{P}(B(T))} \left( a|\mathcal{F}[\nu](0)|^2 + 
 \int |\mathcal{F}[\nu](\lm)|^2 d\mu'(\lm)\right) \\ &= 1 + o(1) = E_0[\nu_0] + o(1).
\end{align*}

\subsection{Adapting the argument to discrete fields}\label{s:fordiscrete}
For discrete fields, Proposition \ref{prop: for univ}-\ref{item:a} does not make sense as stated, since $\nu_\alpha^\theta(\cdot/T)$ is not in $\mathcal{P}(B(T) \cap \Z^d)$, and the proof of Proposition \ref{prop: for univ}-\ref{item:b} is not valid for similar reasons. To remedy \ref{item:b}, one can simply replace $\nu_\alpha^\theta(\cdot/T)$ by any probability measure supported on $B(T) \cap \Z^d$ that is obtained by moving mass by a bounded distance $\sqrt{d}$, and the arguments go through with minimal change. Similarly, for \ref{item:b} one replaces $\nu_n \in \mathcal{P}(B(1))$ with $\nu_n \in \mathcal{P}(B(1) \cap ( \Z^d/T_n) )$, and also replaces $\nu_\alpha^\theta$ in a similar way to as described above.

%%%%%
%%%%%%

\medskip
\section{Regularity of the capacity}
\label{s:ca}

In this section we study the asymptotic regularity of the capacity, proving Propositions \ref{p:suffcapcon} and \ref{p:counterexample2}.

\subsection{Sufficient conditions for capacity regularity}

In this section we prove Proposition~\ref{p:suffcapcon}, handling the two alternate conditions separately. In both cases we use that, since \eqref{e:blowup} holds, by Corollary \ref{c:capbounds} we have the \emph{a priori} estimate
\begin{equation}
    \label{e:apriori}
\capa_f(B(T)) = T^{\alpha + o(1)} .
\end{equation}

\begin{proof}[Proof of Proposition \ref{p:suffcapcon} under the first condition]
 Let $\eps \in (0,1)$ be given, let $T \ge 1$, and abbreviate $S = T^{1-\eps}$. Consider covering $ B(T + S) \setminus B(T)$ with $c_d T^{\eps(d-1)}$ translated copies of $B(S)$. By subadditivity of the capacity for non-negative $K$ (Claim \ref{c:sa}), we have
\[  \capa_{f}(B(T + S))  -  \capa_{f}(B(T)) \le c_d T^{\eps(d-1)} \capa_f(B(S)) , \]
so that by \eqref{e:apriori}
 \[  \frac{ \capa_f(B(T + S))  -  \capa_f(B(T )) }{\capa_f(B(T)) }  =  T^{\eps(d-1) + \alpha(1-\eps) - \alpha + o(1) }  = T^{\eps(d-1-\alpha) + o(1)}.\]
Since we assume $\alpha > d-1$, the latter tends to zero, as required.
\end{proof}

\begin{proof}[Proof of Proposition \ref{p:suffcapcon} under the second condition]
Recall the spectral truncation function $\phi_s$ from \eqref{eq:stf}, defined with respect to some fixed parameter $v \in (0,1)$. Let $\eps \in (0,1)$ be given, let $T \ge 1$, abbreviate $S = T^{1-\eps}$, and let $g_S$ denote an SGF with spectral measure $\mu \phi_S^2$. 

To exploit isotropy and enact smoothing, we first claim that removing a small ball around the origin and replacing $f$ by $g_S$ does not significantly reduce capacity. In particular, we claim that to show capacity regularity, it would suffice to find an $\eps' = \eps'(\eps) \in (0,1)$ that would satisfy $\alpha < (d-1)(1-\eps')$ and
\begin{equation}
    \label{e:suffe1}
\Big| \Big(\capa_{g_S}(B(T+S) \setminus B(T^{1-\eps'}) ) \Big)^{-1}  -  \Big(  \capa_{g_S}(B(T-S))  \Big)^{-1} \Big| \le T^{-\alpha-\eps' + o(1)}   . 
 \end{equation}
 
To prove this is sufficient, note that since $f$ is isotropic and $K$ is eventually non-negative, by  \eqref{e:saext2} of Claim \ref{c:sa}
\begin{align*} 
\capa_f(B(T+S)) &\le \left(\frac{1}{   \capa_f(B(T+S) \setminus B(T^{1-\eps'}) ) +  \capa_f(B(T^{1-\eps'}))  } - c T^{-(1-\eps')(d-1)}\right)^{-1} \\
& \le (1+o(1))  \Big( \capa_f(B(T+S) \setminus B(T^{1-\eps'}) ) +  \capa_f(B(T^{1-\eps'}))  \Big) \\ 
& \le (1+o(1))  \capa_f(B(T+S) \setminus B(T^{1-\eps'}) ) +  o( \capa_f(B(T)) ) ,
\end{align*}
where the second inequality used monotonicity, \eqref{e:apriori} (applied to $T+S \le 2T$), and the fact that $\alpha < (d-1)(1-\eps')$, and the third used \eqref{e:apriori} and that $\alpha > 0$. We further observe that, by monotonicity (Claim~\ref{c:mon}),
\begin{equation*}    \capa_f(B(T+S) \setminus B(T^{1-\eps'}) )\le \capa_{g_S}(B(T+S) \setminus B(T^{1-\eps'}) )   \end{equation*}
and by the smoothing lemma (Lemma \ref{l:captrunc}),
\begin{equation*}   \capa_f(B(T))\ge \capa_{g_S}(B(T-S)) . \end{equation*}
Combining the previous three displays, 
\begin{align*}
&\capa_f(B(T+S)) - \capa_f(B(T)) \\
 & \qquad\le \Big|  \capa_{g_S}(B(T+S) \setminus B(T^{1-\eps'}) )  -  \capa_{g_S}(B(T-S)) \Big|(1+o(1)) +  o( \capa_f(B(T)) )  \\
 & \qquad \le  T^{-\alpha-\eps' + o(1)}   \capa_{g_S}(B(T+S))  \capa_{g_S}(B(T-S))  +   o(\capa_f(B(T))) , \\
 & \qquad \le  T^{-\eps' + o(1)}      \capa_{f}(B(T))  +   o(\capa_f(B(T))) ,
 \end{align*}
 where the second inequality uses \eqref{e:suffe1} and monotonicity, and the third uses the smoothing lemma and \eqref{e:apriori} (applied to $T+2S \le 3T$). This completes the proof of the lemma, assuming~\eqref{e:suffe1}. 

\medskip
Towards showing \eqref{e:suffe1}, fix some $\eps' \in (0,1)$ to be determined later and denote $S' = T^{1-\eps'}$. Write $\nu$ for a radial equilibrium measure for $B(T+S) \setminus B(S')$ with respect to the field $g_S$, whose existence is guaranteed by Claim~\ref{c:iso}. Next, define $\nu' \in \mathcal{P}(B(T-S))$ by rescaling the component of $\nu$ supported in $B(T+S)\setminus B(T-S)$ onto $B(T-S)$, i.e.,
\[ \nu'(\cdot) = \nu|_{B(T-S)}(\cdot) + \theta \nu|_{B(T+S)\setminus B(T-S)} \big(\theta \cdot \big),\]
  where %$\theta = \frac{T-S}{T+S}$ and $\pi(x)=\theta x$ for $x\in \R^d$. We note that 
  \begin{equation}\label{eq:theta}
  \theta = \frac{T-S}{T+S} = 1 - 2\frac{S}{T} + O\Big( \Big( \frac{S}{T} \Big)^2\Big). 
  \end{equation}
Denote $q_1(\lm) = \mathcal{F}[ \nu|_{B(T+S)\setminus B(T-S)} ](\lm) $ and $q_2(\lm) = \mathcal{F}[ \nu|_{B(T-S)} ](\lm) $, so that
\[  \capa_{g_S}( B(T+S) \setminus B(S')) = \Big( \int \Big|q_1(\lambda) + q_2(\lambda) \Big|^2 \phi_S^2(\lambda) d\mu(\lambda) \Big)^{-1}.\]
By maximality of the capacity,
\[  \capa_{g_S}( B(T-S) ) \ge  \left( E[\nu']\right)^{-1}=\left( \int \left| q_1(\theta \lambda ) + q_2(\lambda) \right|^2 \phi_S^2(\lambda)  d \mu(\lambda) \right)^{-1} . \]
Hence \eqref{e:suffe1} reduces to 
\begin{equation}
    \label{e:suffe1.5}
\left| \int \left| q_1(\theta\lambda ) + q_2(\lambda) \right|^2   - \left|q_1(\lambda) + q_2(\lambda) \right|^2 \phi_S^2(\lambda)  d \mu(\lambda)\right| \le T^{-\alpha-\eps' + o(1)}   . 
 \end{equation}
Using the identity $(a+c)^2 - (b+c)^2 = (a-b)(a+b+2c)$ and the boundedness and fast decay of $\phi_S$ given by Proposition~\ref{c:1},
%~\eqref{e:stfdecay}, 
this further reduces to showing
\begin{equation}
\label{e:suffe2}
\int_{B(S^{-1+\eps''})} \big| q_1(\theta\lambda) - q_1(\lambda) \big| \big| q_1(\theta\lambda) + q_1(\lambda) + 2 q_2(\lambda) \big|   d\mu(\lambda)   \le T^{-\alpha-\eps' + o(1)},
\end{equation}
for sufficiently small $\eps'' > 0$.

\smallskip
Let us now estimate these terms. Since $\nu$ is radial and supported on $B(T+S) \setminus B(S')$, by Proposition~\ref{c:3} we have for $i\in\{1,2\}$,
\[  |q_i(\lambda)| \le c \left(1 + (\lambda S')^{-(d-1)/2}\right)  \quad \text{and} \quad |\partial^k q_i(\lambda) | \le  c (T+S) \left( 1 +  (\lambda S')^{-(d-1)/2}  \right)  \]
for all $\lambda$ and multi-indices $k$ with $|k| = 1$. Plugging in the first order estimate of $\theta$ given in~\eqref{eq:theta}, 
%$$(T-S)/(T+S) = 1 - 2S/T + O((S/T)^2),$$
we obtain the bounds
\[  | q_1(\theta\lambda) - q_1(\lambda) | \le c \begin{cases} \left( |\lambda|\frac{S}{T}\right) T  ( |\lambda| S')^{-(d-1)/2}   & |\lambda| \ge 1/S', \\   |\lambda|S  &  |\lambda| \le 1/S', \end{cases} \]
and 
\[ |q_1(\theta\lambda) + q_1(\lambda) + 2 q_2(\lambda)| \le c  \begin{cases}   ( |\lambda| S')^{-(d-1)/2}  & |\lambda| \ge 1/S', \\   1 &  |\lambda| \le 1/S' . \end{cases}   \]
Dividing the integration region radially at radii $2^k/S'$, $k\in \N_0$, the left-hand side of \eqref{e:suffe2} is thus at most
\begin{align*}
    & c \cdot \frac{S}{S'} \sum_{2^k \le S'/S^{1-\eps''}+1}     ( 2^k)^{-(d-2)}    \mu( B( 2^k/S') ) \\
    & \qquad \le  c \cdot \frac{S}{S'}\cdot  (S')^{-\alpha + o(1)}   \cdot\sum_{2^k \le S'/S^{1-\eps''}+1}    (2^k)^{\alpha - (d-2)  + o(1)}    \\
    & \qquad  \le   c\cdot \frac {S}{S'}\cdot (S')^{-\alpha}  \left(\frac{S'}{S^{1-\eps''}}\right)^{ \max\{\alpha-(d-2),0\} + o(1)} \\
& \qquad \le T^{-\alpha + \eps'-\eps + \eps' \alpha + (\eps + \eps'' + 2 \eps \eps'' - \eps' )\max\{\alpha-(d-2),0\}  + o(1) } .
    \end{align*}
Since we assume $\alpha < d-1$, we can choose $\eps',\eps''>0$ sufficiently small so that 
\[  \eps'-\eps + \eps' \alpha + (\eps + \eps'' + 2 \eps \eps'' - \eps' )\max\{\alpha- (d-2),0\}  < - \eps', \]
which completes the proof of \eqref{e:suffe2}.
\end{proof}

\subsection{Examples of irregular capacity growth}
We now consider whether $\capa_f(B(T+g_T))$ and $\capa_f(B(T))$ are asymptotically equivalent for a more general class of $g_T$ than in \eqref{e:capcon}. 

\smallskip
We first show that $g_T \to \infty$ is insufficient for asymptotic equivalence in general, no matter how slow the growth of $g_T$, and even if we demand that $\mu[B(\delta)] \asymp \delta^\alpha$. In particular this proves Proposition \ref{p:counterexample2}. For simplicity we focus on the $d=1$ case, although similar examples could be constructed in higher dimensions.

\begin{proposition}
\label{p:irreg}
For every $\alpha > 0$, $g_T \to \infty$, and $\rho > 1$, there exists a smooth stationary Gaussian process $f$ satisfying 
\begin{equation}
    \label{e:mureg}
\mu[B(\delta)] \asymp \delta^\alpha  
\end{equation}
as $\delta \to 0$, for which 
 \begin{equation}
    \label{e:capreg}
    \limsup_{T \to \infty} \frac{ \capa_f(B(T + g_T)) }{ \capa_f(B(T)) }  > \rho . 
    \end{equation}
\end{proposition}
\begin{proof}[Proof of Proposition \ref{p:counterexample2} given Proposition \ref{p:irreg}]
Fix $\eps \in (0,1)$ and set $g_T = T^{1-\eps}$ in Proposition \ref{p:irreg}.
\end{proof}
\begin{proof}[Proof of Proposition \ref{p:irreg}]
Let $\alpha >0$ be given, let $\eps \in (0,1)$ be a small constant to be determined later, and fix a positive sequence $(T_i)_{i \ge 0}$ such that $T_0 = 1$, $T_i/T_{i-1} \in \{3,5,\ldots\}$, and
\begin{equation}
    \label{e:gt}
 g_{T_i/4 -  T_{i-1}/\eps} \ge  T_{i-1} /\eps, 
 \end{equation}
which is possible since $g_T \to \infty$. We will show that there exists a smooth Gaussian process $f$ satisfying \eqref{e:mureg}, and constants $c_1,c_2>0$ depending only on $\alpha$, such that, for all $i \ge 1$,
\begin{equation}
    \label{e:capbound1}
\capa_f( B(T_i/4-  T_{i-1}/\eps)) < c_1  T_i^{\alpha} 
\end{equation}
and
\begin{equation}
    \label{e:capbound2} \capa_f(B(T_i/4)) > c_2  \eps^{-\alpha} T_i^\alpha .
    \end{equation}
By \eqref{e:gt}, monotonicity of the capacity (Claim \ref{c:mon}), and taking $\eps>0$ sufficiently small, this demonstrates \eqref{e:capreg}. 

Before constructing the process, let us convey the intuition. Roughly speaking our process has a spectrum, when viewed on scale $1/T_i$, which is approximately supported on the lattice $(1/T_i)(2 \mathbb{Z} + 1)$. Since the RKHS associated to this portion of the spectrum only contains periodic functions with period $T_i$, any function in the RKHS of $f$ that persists over a longer length $T_i + s \gg T_i$ must charge the measure corresponding to the subsequent scale $1/T_{i+1}$. By construction, such functions will have much larger norm, which leads to $\capa_f(B(T_i + s)) \gg \capa_f(B(T_i))$. By defining the scales appropriately, we can ensure \eqref{e:mureg} still holds.
 
Let us now construct the process formally. Let $c > 0$ be a constant that depends only on $\alpha$ but may change from line to line. Recall the spectral truncation function $\phi_s$ given in \eqref{eq:stf}. Define the measure 
\[ \mu'_0(\cdot) = \frac12 \delta_{-1}(\cdot) + \frac12 \delta_{1}(\cdot)    \]
and, for $i \ge 1$, the measures
\[ \mu'_i(\cdot) := s_i(\cdot)  \sum_{k \in 2 \Z + 1}  \delta_{k/T_i}(\cdot) \, , \quad s_i(\lambda) =  T_i^{-1}  |\lambda|^{\alpha-1}  \phi_{\eps/T_{i-1}}(\lambda)\]
where $\delta_{x}$ is a Dirac mass at the point $x$.  Note that $\mu_i$ is supported on 
\[ (1/T_i)(2\Z + 1) \subseteq \cap_{j \ge i}(1/T_j)(2 \Z + 1), \]
and also that, for $i \ge i$,
\[  \|\mu'_i\| = T_i^{-\alpha} \sum_{k \in 2 \Z + 1}  |k|^{\alpha-1} \phi_{\eps/T_{i-1}}(k/T_i) \le c T_i^{-\alpha} ( \eps T_i / T_{i-1})^\alpha = c \eps^\alpha T_{i-1}^{-\alpha} ,  \]
where we used the decay of the truncation. Recalling that $\eps < 1$ and $T_i / T_{i-1} \ge 3$, the latter bound ensures that 
\[ \|\mu'\| := \sum_{i \ge 0} 
\|\mu'_i\| \in (1, 
 c)  < \infty, \]
 so we may define rescaled measures $\mu_i := \mu'_i / \|\mu'\| $ and a probability measure $\mu := \sum_{i \ge 0} \mu_i$. Let $f$ be the stationary Gaussian process with spectral measure~$\mu$. 

We claim that $f$ is smooth and satisfies \eqref{e:mureg}. Indeed smoothness follows by the decay properties of the truncation $\phi$, so let us to check \eqref{e:mureg}. Fix $\delta \le 1$, and let $i \ge 1$ be such that $1/T_{i-1} > \delta \ge 1/T_i$. Then
\[ \mu[B(\delta)] \le \mu_i[B(\delta)] + \sum_{j \ge i+1} \|\mu_j\| \le T_i^{-\alpha} \sum_{k \in 2\Z + 1 : |k| \le \delta/T_i } |k|^{\alpha-1}  +  c T_i^{-\alpha}  \le c \delta^\alpha  \]
and moreover
\[ \mu[B(\delta)] \ge  \mu_i[B(\delta)] \ge c T_i^{-\alpha}  \sum_{k \in 2\Z + 1 : |k| \le \delta/T_i } |k|^{\alpha-1}  \phi_{\eps/T_{i-1}}(k/T_i) \ge c_\eps \delta^\alpha , \]
where we used Claim \ref{c:2} to bound $\phi$ from below. This verifies \eqref{e:mureg}.

It remains to establish \eqref{e:capbound1} and \eqref{e:capbound2}, beginning with the latter. Fix $i$ and consider the probability measure $\eta \in \mathcal{P}(B(T_i/4))$ given by
\[  \eta(\cdot) = \frac12 \delta_{-T_i/4} + \frac12 \delta_{T_i/4} \, , \quad \mathcal{F}[\eta](\lambda) = \cos(  \lambda T_i  \pi / 2 ). \]
We claim that 
\[  \int  | \mathcal{F}[\eta](\lambda)|^2  d \mu(\lambda)  = \sum_j \int | \mathcal{F}[\eta](\lambda)|^2 d\mu_j(\lambda)  < c \eps^{\alpha} T_i^{-\alpha} , \]
which verifies \eqref{e:capbound2} by taking $\eta$ as a candidate in \eqref{e:cap1}. Indeed, for all $j \le i$,
\[  \int | \mathcal{F}[\eta](\lambda)|^2 d\mu_j(\lambda)  = 0 \]
since $\mu_j$ is supported on odd multiples of $1/T_i$, and $\mathcal{F}[\eta]$ vanishes at such points. On the other hand, since $|\mathcal{F}[\eta]|\le 1$, for all $j \ge 1$ we have
\[  \int  | \mathcal{F}[\eta](\lambda)|^2 d\mu_j(\lambda)  \le \|\mu_j\| \le \|\mu'_j\| \le c \eps^{\alpha} T_{j-1}^{-\alpha} .\]
Hence 
\[ \int  | \mathcal{F}[\eta](\lambda)|^2  d \mu(\lambda)  \le c  \eps^{\alpha} \sum_{j \ge i+1}  T_{j-1}^{-\alpha} \le c \eps^{\alpha} T_i^{-\alpha} \]
as required.

We turn to \eqref{e:capbound1}. By monotonicity (Claim \ref{c:mon}) it suffices to bound $\capa_{\mu_i}(B( T_i/4 - T_{i-1}/\eps))$. Consider the function 
\[ g(\lambda) = (2 \|\mu'\|)|\lambda|^{1-\alpha}  \mathcal{F}[ \id_{B(T_i/4)}](\lambda)  \]
and define $h = \mathcal{F}[ g \mu_i ]$. We claim that
\begin{equation}
    \label{e:h1}
h(x) \ge 1 \ \text{for} \ |x| \le T_i/4 - T_{i-1} /  \eps
\end{equation}
and also that
\begin{equation}
\label{e:h2}
\|h\|_H^2 = \int g(\lambda)^2 d\mu_i(\lambda) \le  c T_i^{-\alpha} 
\end{equation}
 Given this, using $h$ as a candidate in definition \eqref{e:cap2} of the capacity shows that
\[\capa_{\mu_i}(B( T_i/4 -  T_{i-1}/\eps)) \le  c T_i^{-\alpha}  , \]
which verifies \eqref{e:capbound1}.

So let us establish \eqref{e:h1} and \eqref{e:h2}. By the identity in Corollary~\ref{cc:4},
\begin{align}
\nonumber h(\cdot) &= (1/\|\mu'\|) \mathcal{F}\Big[ g(\cdot)s_i(\cdot)  \sum_{k \in 2\Z+1}  \delta_{k/T_i}(\cdot)  \Big](\cdot) \\&= 
(1/\|\mu'\|) \mathcal{F}[\text{Dis}_{1/T}(g)](\cdot) - \mathcal{F}[\text{Dis}_{2/T}(g)](\cdot)\nonumber \\
\label{e:h3} & = (1/\|\mu'\|) (T_i/2) \sum_{k \in \Z} (-1)^k  \mathcal{F}[g(\cdot) s(\cdot)]    * \delta_{k T_i/2 }(\cdot) . 
\end{align}
Moreover 
\[ T_i/(2 \|\mu'\|) \mathcal{F}[g(\cdot) s(\cdot)] =   \mathcal{F}[ \mathcal{F}[ \id_{B(T_i/4)}] \phi_{\eps/T_{i-1}} ] = \id_{[-T_i/4,T_i/4]} * \mathcal{F}[ \phi_{\eps/T_{i-1}} ] , \]
and so, recalling that $\mathcal{F}[ \phi_{s} ]$ is supported on $B(s/2)$,
\[ T_i/(2 \|\mu'\|)  \mathcal{F}[g(\cdot) s(\cdot)] \]
is supported on $|x| \le T_i/4 + T_{i-1}/\eps$, and has value exceeding $1$ on $|x| \le T_i/4 - T_{i-1}/\eps$. Combining with \eqref{e:h3}, we verify \eqref{e:h1}.

 On the other hand, recalling that $|\phi_s| \le 1$,
\[ \int g(\lambda)^2 d\mu_i(\lambda) \le  c T_i^{-\alpha} \sum_{k \in 2 \Z + 1}  |k|^{1-\alpha}  (\mathcal{F}[ \id_{B(1/4)}](k))^2 \le c T_i^{-\alpha} , \]
where we used that $|k|^{1-\alpha}  (\mathcal{F}[ \id_{B(1/4)}](k))^2$ is summable by Claim \ref{c:3}. This proves \eqref{e:h2}.
\end{proof}

We conclude with some final remarks:

\begin{enumerate}
\item The process we constructed in Proposition \ref{p:irreg} had a purely atomic spectrum, however by appropriate smoothing one could also create an example with spectral density.

\item It would be interesting to determine whether $g_T \to \infty$ is necessary in Proposition \ref{p:irreg}, i.e.\ whether, for fixed $c > 0$ and assuming \eqref{e:mureg}, $\capa_f(B(T + c)) / \capa_f(B(T))$ always converges (it is easy to see that this may fail if one does not impose~\eqref{e:mureg}).
\end{enumerate}

%%%%%%%
%%%%%%

\medskip
\section{The effect of singular continuous measures}
\label{s:sc}

In this section we prove Proposition~\ref{p:counterlim}, which shows that the limit in Theorem~\ref{t:main} may not exist in the presence of a singular continuous component of the spectral measure.

\smallskip
Let $(s_n)_{n \ge 1}$ be a rapidly increasing sequence of positive integers such that
\begin{equation}
    \label{e:sgrow}
 \lim_{n \to \infty} s_{n+1}/ s_n = \infty . 
\end{equation}
For instance one could take $s_n = 2^{2^n}$. Define the set of positive integers
\[ J = \bigcup_{n=1}^\infty \big[ s_{2n-1}, s_{2n} \big] \subset \N, \]
and the measurable set
\[ E = \cup_{m \in \N}E_m \ , \quad E_m = \Big\{1 + \sum_{j \in J \cap [m]} b_j 2^{-j} \mid b_j\in \{0,1\} \Big\}, \]
where $[m] = \{1,2,\ldots,m\}$. Denote by $\mu_E$ the singular continuous measure supported on $E$, which may be characterised by its action on dyadic intervals: 
\[\mu_E\left([k2^{-m},(k+1)2^{-m}]\right) = \begin{cases} 1/|E_m| = 2^{-|J\cap [m]|} &  k \in E_m , \\ 0 & k \notin E_m.  \end{cases} \]

Fix $\alpha\in (0,1)$ and construct the spectral measure $\mu = \mu_{a.c.}+\mu_{s.c.}$, where:
\begin{itemize}
\item $\mu_{a.c.}$ has the canonical density $\rho_\alpha(\lambda) = A_{\alpha,1}|\lambda|^{\alpha-1}$ restricted to $[-1,1]$ (see Section~\ref{ss:uni}).
\item $\mu_{s.c.}$ is supported on $[-2,-1]\cup [1,2]$, and is the symmetrisation of $\mu_E$.
\end{itemize}
Observe that $\mu$ satisfies \eqref{e:blowup} and that, by Theorem \ref{t:capuniv}, as $T \to \infty$
 \begin{equation}
     \label{e:capexample}
 \capa_\mu(B(T)) \sim c_{\alpha,1} T^\alpha . 
  \end{equation} 
  Hence, applying Theorem \ref{t:bounds},
\begin{equation}
\label{e:thm5}
\liminf_{T \to \infty} 
 \frac{ \theta_\mu(T) / (1-\alpha)}{\capa_f(B(T)) \log T} \ge \|\mu_{a.c.}\|  \quad \text{and} \quad  \limsup_{T \to \infty} 
 \frac{ \theta_\mu(T)/(1-\alpha) }{\capa_f(B(T)) \log T}  \le \|\mu\| .
 \end{equation}
Let $(T_n)_{n \ge 1}$ be an increasing sequence of positive integers such that
\begin{equation}\label{eq:T and s} \lim_{n \to \infty} \frac{\log  T_n  }{ s_n}   =  \lim_{n \to \infty}  \frac{s_{n+1} }{  \log T_n } = \infty ,  \end{equation}
which exists by \eqref{e:sgrow}. We aim to show that, as $n \to \infty$,
\begin{equation}
    \label{e:twolimits}
 \frac{ \theta_\mu(T_{2n})/(1-\alpha)}{\capa_f(B(T_{2n})) \log T_{2n}}  \le \| \mu_{a.c.} \| + o(1)   \ \ \text{and} \ \  \frac{ \theta_\mu(T_{2n+1})/(1-\alpha)}{\capa_f(B(T_{2n+1}))\log T_{2n+1}}  \ge \| \mu\| - o(1) .
\end{equation}
Combining with \eqref{e:thm5}, this will complete the proof of Proposition \ref{p:counterlim}

\subsection{Even indices}
\smallskip

 Let us first examine $\mu_{s.c.}$ when viewed at scales $T_{2n}^{-1}$. Since $J \cap (s_{2n}, s_{2n+1})$ is empty, $\mu_{s.c.}$ is supported on the union of the $2^{|J \cap [s_{2n}]|} \le 2^{s_{2n}}$ disjoint intervals 
 \[  [k 2^{-(s_{2n+1}-1)},(k+1)2^{-(s_{2n+1}-1)}] \ ,  \quad k \in E_{s_{2n}} = E_{s_{2n+1}-1}  \]
of width $2^{-(s_{2n+1}-1)}$. As such $\mu_{s.c.}$ appears to be `zero dimensional' when viewed at scales $T_{2n}^{-1}$ (recall that $2^{s_2n}=T_{2n}^{o(1)}$).
Such zero dimensional measures behave very much like discrete measures.
    
We formalise this intuition by showing that, on scales $T_{2n}$, $f_{\mu_{s.c.}}$ has moderate deviations of order $\ll \sqrt{\log T_{2n}}$, just as if $\mu_{s.c.}$ were discrete (cf.\ Proposition \ref{p:m2}):

\begin{proposition}
\label{p:m3}
For every $\delta > 0$, as $n \to \infty$
\[ - \log \P \Big[ \sup_{x \in B(T_{2n})} |f_{\mu_{s.c.}}(x)|  \le \delta \sqrt{\log T_{2n}} \Big]   \le T_{2n}^{o(1)} .\]
\end{proposition}
This proposition, whose proof is given below, allows us to establish the first statement in \eqref{e:twolimits}, by repeating the arguments in the proof of the lower bound of Theorem~\ref{t:bounds} (see Section~\ref{ss:lowerbd}). Recall that $m = \|\mu_{a.c.}\|$. Let $\delta > 0$ be given and define $\ell =  \sqrt{ 2m(1-\alpha)} +  3\delta  > 0$. Fix $\eps > 0$ sufficiently small so that $\|\mu|_{B(\eps)}\| <  \delta /2$, and also
\begin{equation}
    \label{e:mu2tag}
 - \log \P\Big[ \inf_{x \in B(T)} f_{\mu_{a.c.}|_{B(\eps)^c}}(x) \ge (- \ell +2\delta) \sqrt{\log T}   \Big] \le T^{ \alpha - c_\delta + o(1) } , 
 \end{equation}
which is possible by the definitions of $\ell$, $m^+$ as in~\eqref{e:m+}, and the fact that $m^+(\mu_{a.c.})=m$ (Corollary~\ref{c:m}). For $n \ge 1$, decompose the absolutely continuous part of the spectral measure 
 \[ \mu_{a.c.} = \mu_1 + \mu_2 + \mu_3, \] 
 where $\mu_1 = \mu \phi_{T_{2n}^\delta}^2$, $\mu_2 = \mu(1-\phi_{T_{2n}^\delta}^2)|_{B(\eps)^c}$, and $\mu_3 = \mu(1-\phi_{T_{2n}^\delta}^2)|_{B(\eps)}$. 
 Consider the independent events:
 \begin{itemize}
     \item $A_1= \Big\{ \inf_{x \in B(T_{2n})} f_{\mu_1}(x) \ge \ell \sqrt{\log T_{2n}}  \Big\} $;
     \item $A_2 =  \Big\{ \inf_{x \in B(T_{2n})}    f_{\mu_2}(x)   \ge  (-\ell+\delta) \sqrt{\log T_{2n}} \Big\} $;
     \item $A_3 =  \Big\{ \sup_{x \in B(T_{2n})}    |f_{\mu_3}(x) |  \le  \tfrac{\delta}2 \sqrt{\log T_{2n}} \Big\}   $;
     \item $A_4 = \Big\{\sup_{x\in B(T_{2n})} |f_{\mu_{s.c.}}(x)|\le \tfrac \delta 2 \sqrt{\log T_{2n}} \Big\} $.
 \end{itemize}
 The events $A_1$ and $A_2$ are exactly the same as in Section~\ref{ss:lowerbd}, while in $A_3$ the constant $\delta$ was replaced by~ $\delta/2$. Clearly $\cap_{j=1}^4 A_j  \implies \per(B(T_{2n}))$, so that
\begin{equation*}\label{e:lb_dec}
-\log \P(\per(B(T_{2n}))) \le -\log \P\left[A_1\cap A_2 \cap A_3\right]
-\log \P \left[A_4 \right].
\end{equation*}
By the arguments in Section~\ref{ss:lowerbd} we obtain that $-\log \P[A_1\cap A_2\cap A_3]$ is at most
\begin{equation*}
-\log \P\left[A_1\cap A_2\cap A_3\right] \le 
 (\ell+ \delta)^2 \capa_f(B(T_{2n} + T_{2n}^\delta)) \log T_{2n} \,( 1 + o(1) )    + T_{2n}^{\alpha - c_\delta + o(1) } .
 \end{equation*}
On the other hand, by Proposition \ref{p:m3}, as $n \to \infty$ 
\[  -\log \P[A_4] \le T_{2n}^{o(1)} . \]
Combining the previous three displays, and recalling \eqref{e:capexample}, completes the proof.

\medskip
\noindent\emph{Proof of Proposition~\ref{p:m3}.}
The proof follows the lines of \cite[Theorem 2]{ffm25}. Fix $\enn\in\N$ and denote the intervals
$I^{\enn}_{\pm j} =\pm \big((j-1)\frac{2}{\enn},j\frac{2}{\enn}\big]$ for $j\in [\enn]$. Write
\begin{equation*}\label{eq: Cj and Sj}
C_j(t) =\frac 1{\mu_{s.c.}(I_{j})}  \int_{I_{j}} \cos(\lm t) d\mu_{s.c.}(\lm),
\quad
S_j(t) = \frac 1{\mu_{s.c.}(I_{j})}  \int_{I_{j}} \sin(\lm t) d\mu_{s.c.}(\lm).
\end{equation*}
Note that by Jensen's inequality
\begin{equation}\label{eq: C2+S2}
|C_j(t)|^2+ |S_j(t)|^2  \le 1,
\end{equation}
and, by simple trigonometric inequalities, for every $t,s\in\R $ and $j\in [n]$
\begin{equation}\label{eq: C and S dif}
    \max\left\{|C_j(t)-C_j(s)|, |S_j(t)-S_j(s)|\right\} \le 2|t-s|.
\end{equation}
Then, by \cite[Lemma 4.2]{ffm25} we have
the following decomposition: 
\begin{equation}\label{eq: def R}
f_{\mu_{s.c.}}(t) \overset{d}{=}
 \sum_{j=1}^{ \enn  }    \sqrt{\mu_{s.c.}(I_j\cup I_{-j})} \Big(\zeta_j C_j(t)\oplus \eta_j S_j(t)\Big)  \oplus R_{\enn}(t),
\end{equation}
where $\{\zeta_j\}_{j=1}^{\enn}\cup\{\eta_j\}_{j=1}^{\enn}$ are i.i.d.\ standard Gaussian random variables, and $R_{\enn}(t)$ is an independent Gaussian process for which
\begin{equation}\label{eq: goal var}
\sup_{t\in [0,T]} \var(R_{\enn}(t)  ) \le  C\left(\frac{T}{\enn}\right)^2
\end{equation}
and for any $t\in [0,T],\:h>0$,
\begin{equation}\label{eq: goal dif}
\var(R_{\enn}(t)-R_{\enn}(t+h)  ) \le  C h^2,
\end{equation}
where $C$ is a constant depending only on $\|\mu\|$.

Fix $\ep>0$ and partition the indices in $[\enn]$ into
\begin{equation}\label{eq: AB}
 \cA_{\enn,\ep}  =\left\{j\in [\enn] : \: \mu_{s.c.}(I_j\cup I_{-j})\ge  \frac {\pi \ep}{\enn}\right\}, \quad \cB_{\enn,\ep}  = [\enn]  \setminus \cA_{\enn,\ep} ,
\end{equation}
defining accordingly the functions
\begin{align*}
A_{\enn,\ep} (t) &= \sum_{j\in \cA_{\enn,\ep} }  \sqrt{\mu_{s.c.}(I_j\cup I_{-j})} \Big(\zeta_j C_j(t)\oplus \eta_j S_j(t)\Big),\\
B_{\enn,\ep} (t) &=\sum_{j\in \cB_{\enn,\ep} }   \sqrt{\mu_{s.c.}(I_j\cup I_{-j})} \Big(\zeta_j C_j(t)\oplus \eta_j S_j(t)\Big).
\end{align*}
Thus \eqref{eq: def R} becomes
$f_{\mu_{s.c.}} \overset{d}{=} A_{\enn,\ep}  \oplus B_{\enn,\ep}  \oplus R_\enn$, and for any $\delta>0$
we have
\begin{align*}
& -\log \P\Big( |f_{\mu_{s.c.}}| < \delta \sqrt{\log T}\: \text{ on } [0,T] \Big) \\
& \qquad \qquad  \le  -\log \P\Big(\sup_{[0,T]}\big| A_{\enn,\ep}  \big| <\tfrac  \delta 3\sqrt{\log T}\Big) 
-\log \P\Big(\sup_{[0,T]} \big| B_{\enn,\ep}   \big| <\tfrac  \delta 3 \sqrt{\log T}\Big)\\
& \qquad \qquad \qquad \qquad
-\log \P\Big(\sup_{[0,T]} \big| R_{\enn}   \big| <\tfrac  \delta 3 \sqrt{\log T} \Big).
\end{align*}

We carry out this decomposition with $\enn=mT$, where $m$ will be chosen shortly. Throughout, we consider $\delta>0$ as fixed.
The proposition follows easily by combining the next three claims. 

\begin{claim}\label{clm: A}
For any fixed  $m \in \N$ and $\ep>0$, writing $T=T_{2n}$,   as $n\to\infty$ we have
\[  -\log \P\Big(\sup_{[0,T]} \big| A_{mT,\ep}  \big| <\tfrac \delta 3 \sqrt{\log T} \Big) \le T^{o(1)} .\]
\end{claim}

\begin{claim}\label{clm: B}
There exists $T_0(\delta),\ep_0(\delta)>0$, such that for all $T>T_0$, $m\in\N$ and $\ep<\ep_0$,
\begin{equation*}
- \log \P\Big(\sup_{[0,T]} \big| B_{mT,\ep}   \big| <\tfrac \delta 3 \sqrt{\log T}\Big) \le \log 2.
\end{equation*}
\end{claim}

\begin{claim}\label{clm: R}
There exists $T_0(\delta), m_0(\delta)>0$, such that for all $T>T_0$ and $m>m_0$,
\[
-\log \P\Big(\sup_{[0,T]} \big| R_{mT}  \big| <\tfrac \delta 3 \sqrt{\log T} \Big) \le 
\log 2.
\]
\end{claim}

\begin{proof}[Proof of Claim~\ref{clm: A}]
The proof is very similar to the proof of Claim 4.3 in \cite{ffm25}. We write 
$$u(t) = \Big(\sqrt{\rho(I_j \cup I_{-j}) } C_j(t),\sqrt{\mu_{s.c.}(I_j \cup I_{-j}) } S_j(t)  \Big)_{j\in \cA_{mT,\ep}}$$
Denote $d'=2|\cA_{mT,\ep}|$. As pointed out earlier, $d'=e^{c(m)s_{2n}} = T_{2n}^{o(1)}$.
We recall that $A_{mT,\ep}(t)= \langle  u(t), \zeta \rangle$ where $\zeta\sim \gamma_{d'}$ is a standard $d'$-dimensional Gaussian random vector.
By \eqref{eq: C2+S2},
\begin{equation}\label{eq: ubound}
\norm{ u (t)}^2 = \sum_{j\in\cA_{mT,\ep}} \mu_{s.c.}(I_j\cup I_{-j} ) (C_j^2(t) +S_j^2(t) ) \le \sum_{0\le j< mT} \mu_{s.c.}(I_j \cup I_{-j}) \le \norm{\mu_{s.c.}}.
\end{equation}
Hence,
 \begin{align*}
&\P\Big( \bigcap_{k=1}^{T} \left\{ |A_{mT,\ep}(k)| <  \tfrac {\delta} 3 \sqrt{\log T}\right\}\Big)
 = \gamma_{d'} \Big(  \bigcap_{k=1}^{T}
\left\{ \zeta \in \R^d : \ \left| \left\langle \zeta,  u(k) \right\rangle \right| \le \tfrac {\delta} 3 \sqrt{\log T} \right\} \Big) \notag
\\ & \qquad \qquad \ge  \gamma_{d'} \Big(  \bigcap_{k=1}^{T }
\left\{ \zeta \in \R^{d'} : \ \left| \left\langle  \zeta, \tfrac {  u(k)}{ \norm{u(k)} } \right\rangle \right| \le  \tfrac {\delta} {3\sqrt{2}\norm{\mu_{s.c.}}} \sqrt{ 2\log T}\right\} \Big)
\\ &\qquad \qquad \ge  \left(\tfrac {\delta} {3\sqrt{2}\norm{\mu_{s.c.}}}\right)^{d'}
 \gamma_{d'} \Big(  \bigcap_{k=1}^{T}
\left\{ \zeta \in \R^d : \ \left| \left\langle  \zeta, \tfrac {  u(k)}{ \norm{u(k)} } \right\rangle \right| \le \sqrt{2 \log T} \right\} \Big)
  \\ &  \qquad \qquad  \ge  \left(\tfrac {\delta} {3\sqrt{2}\norm{\mu_{s.c.}}}\right)^{d'} \gamma_1 \big(\{ \zeta\in\R: \: |\zeta|\le \sqrt{2\log T}]\} \big)^T
   \\ & \qquad \qquad \ge 
     \left(\tfrac {\delta} {3\sqrt{2}\norm{\mu_{s.c.}}}\right)^{d'} e^{-1}.
\end{align*}
The first inequality used~\eqref{eq: ubound}, the second is an application of Observation 2.2 in \cite{ffm25} (used with $\alpha = 3\sqrt{2}\norm{\mu_{s.c.}}\delta^{-1}>1$), the third is an application of the Gaussian correlation inequality, while the fourth is due to Gaussian tail bounds (Claim~\ref{c:gtb}).  This yields, by the Gaussian correlation inequality, that
\begin{align*}
-\log \P &\Big(\sup_{[0,T]} \big| A_{mT,\ep}  \big| <\tfrac \delta 3 \sqrt{\log T} \Big)
 \\ &\le -\log 
\P\Big( \bigcap_{k=1}^{T} \left\{ |A_{mT,\ep}(k)| <  \tfrac {\delta} 6 \sqrt{\log T}\right\}\Big)\\ &
\qquad \qquad -\frac{\lceil T\rceil }{T}\log \P\Big(\sup_{t\in[0,1]} |A_{mT,\ep}(t)-A_{mT,\ep}(0)|<\tfrac \delta 6 \sqrt{\log T}\Big)
\\ & \le
d' \log \left(\tfrac {\delta} {3\sqrt{2}\norm{\mu_{s.c.}}}\right)    - 1 -2\log \P\Big(\sup_{[0,1]}|f'_{\mu_{s.c.}}| < \tfrac \delta 6 \sqrt{\log T} \Big)
\\ & 
\le T^{o(1)}.
\end{align*}
In the second inequality we used the previous estimate and Anderson's inequality, and in the third inequality we used the fact that $d'= T^{o(1)}$ and the Borell-TIS inequality.
\end{proof}

\begin{proof}[Proof of Claim~\ref{clm: B}]
Fix $m\in \N$.
Using \eqref{eq: C2+S2} and \eqref{eq: AB}, we have for any $0\le t\le T$ and $\ep>0$,
\begin{align*}
\var (B_{mT,\ep}(t) ) &= \sum_{j\in\cB_{mT,\ep}} \mu_{s.c.}(I_j \cup I_{-j}) \left( C_j^2(t)+S_j^2(t)\right)
\\ & \le |\cB_{mT,\ep}|\cdot 2\max_{j\in\cB_{mT,\ep}} \mu_{s.c.}(I_j) \sup_{t\in [0,T]} \left(C_j^2(t) +S_j^2(t)\right)
\\ & \le  mT \cdot \frac{2\pi \ep}{mT} \cdot 1 = 2\pi \ep.
\end{align*}

Similarly, by \eqref{eq: C and S dif}, the canonical distance induced by $B$ obeys
\begin{align*}
d_B(s,t)^2 &= \var[B(s)-B(t)]
 \\ &=\sum_{j\in\cB_{mT,\ep}}\mu_{s.c.}(I_j \cup I_{-j}) \left( (C_j(t)-C_j(s))^2+(S_j(t)-S_j(s))^2\right)
 \\ &\le |\cB_{mT,\ep}|\cdot 2\max_{j\in\cB_{mT,\ep}} \mu_{s.c.}(I_j) \left((C_j(t)-C_j(s))^2 + (S_j(t)-S_j(s))^2\right)
 \\ & \le c\ep |t-s|^2,
\end{align*}
for an absolute constant $c>0$. 

Therefore, $N_B(x) \le\frac{\sqrt{c\ep}T}{x}$ and 
\[\text{diam}_B([0,T]) =\sup_{s,t\in [0,T]}d_B(s,t)\le \sqrt{2 \sup_{t\in[0,T]}\var B(t) }\le 2\sqrt{\pi\ep}.
\]
By Dudley's bound,
\[
\E \sup_{[0,T]} B_{mT,\ep} \le K\int_0^{2\sqrt{\pi\ep}} \sqrt{\log \left(\frac{\sqrt{c\ep}T}{x}\right)} \: dx \le \tilde{K} \sqrt{\ep\log T},
\]
for universal constants $K,\tilde{K}>0$.
By the Borell-TIS inequality, if $\tilde{K}\sqrt{\ep}<\delta/6$, then
\[
\P\Big( \sup_{[0,T]}|B_{mT,\ep}| >\frac{\delta}{3}\sqrt{\log T} \Big)\le 
\P\Big( \sup_{[0,T]}|B_{mT,\ep}| - \E\sup_{[0,T]}|B_{mT,\ep}| >\frac{\delta}{6}\sqrt{\log T} \Big)\le T^{-\delta^2/36}\le \frac 1 2,
\]
for $T>T_0(\delta)$. The result then follows by considering the complementary event.
\end{proof}

\begin{proof}[Proof of Claim~\ref{clm: R}]
In view of \eqref{eq: goal var} and \eqref{eq: goal dif}, the claim is again a routine application of Dudley's entropy bound on $\E \sup_{[0,T]} R_{mT}$ and the Borell-TIS inequality, and we omit the details.
\end{proof}

\subsection{Odd indices}
Now let us examine $\mu_{s.c.}$ when viewed at scales $T_{2n+1}^{-1}$. At this scale $\mu_{s.c.}$ `looks like' an absolutely continuous measure, in the sense that it is approximately uniform on intervals of width $s_{2n+1}^{-1} \gg T_{2n+1}^{-1}$. In particular, on balls of radius $T_{2n+1}$ the infimum of $f_\mu$ has moderate deviations of scale $\sqrt{2 \|\mu\| \log T_{2n+1}}$, just as if $\mu_{s.c.}$ were absolutely continuous (cf. Proposition \ref{p:m1}):

\begin{proposition}
\label{p:odd}
For every $\eps \in (0,\|\mu\|/2)$ there exists $\delta_0 > 0$ such that, for all $\delta \in (0, \delta_0)$ and $\beta \in [0,1)$, as $n \to \infty$,
\[ 
-\log \P \Big[ \inf_{x \in B(T_{2n+1})}  f_{\mu|_{B(\delta)^c}}(x) \ge - \sqrt{2 (\|\mu\|-\eps) ( 1 - \beta)   \log T_{2n+1}} \Big] \ge T_{2n+1}^{\beta + \frac{\eps}{2\|\mu\|}(1-\beta) + o(1)} . \]
\end{proposition}

We complete the proof of the second statement in \eqref{e:twolimits}. The proof is identical to the proof of the upper bound of Theorem~\ref{t:bounds} (see Section~\ref{s:ub}), except that we replace $m^-$ with $\|\mu\| $ and use %the conclusion of 
Proposition \ref{p:odd} rather than \eqref{e:mixing1} 
%the definition of $m^-$ 
to establish \eqref{e:mixing}. 
We deduce that, for every $\delta > 0$, there exist $\eps,\eps' > 0$ such that, as $n \to \infty$, eventually
\[   - \log \P \big[ \per(B(T_{2n+1}))   \big]   \ge  \min\big\{ T_{2n+1}^{\alpha + \eps' + o(1)},(\norm{\mu}(1-\alpha) -  \delta)(\log T_{2n+1}) \capa_f(B(T_{2n+1}-T_{2n+1}^{1-\eps}))  \big\}  .\]
Recalling \eqref{e:capexample}, this gives the result.

\begin{proof}[Proof of Proposition \ref{p:odd}]
For $m \ge 1$ define the absolutely continuous measure
\begin{align}
\label{e:muapprox}
  \nonumber  d\mu_E^m &= \sum_{k \in  E_m}  \mu_E\left([ k 2^{-m},(k+1)2^{-m}]\right) 2^m 
 \id_{[k2^{-m},(k+1)2^{-m}]} \, dx \\
    &= 2^{m-|J \cap [m]|}  \sum_{k \in  E_m}  \id_{[k2^{-m},(k+1)2^{-m}]}  \, dx ,
    \end{align}
    which `smooths' $\mu_E$ on scale $2^{-m}$, and let $\mu_{s.c.}^m$ be the symmetrisation of $\mu_E^m$. 
    Writing $W_2$ for the $2$-Wasserstein distance, we have
\begin{equation}\label{eq:W2}
    W_2(\mu_E, \mu_E^m) \le 2^{-m}.
    \end{equation}
    Note that if $m+1\in J$ then
    \begin{align*}
    d\mu_E^{m+1} & =
    2^{m+1-|J \cap [m+1]|}  \sum_{k \in  E_{m+1}}  \id_{[k2^{-(m+1)},(k+1)2^{-(m+1)}]}  \, dx 
    \\ & =2^{m-|J \cap [m]|}  \sum_{k \in  E_m}  \id_{[k2^{-m},(k+1)2^{-m}]}  \, dx =d\mu_E^{m},
    \end{align*}
    so that $\mu_{s.c.}^m = \mu_{s.c.}^{m+1}$.
    Recalling that $J$ includes all integers in $[s_{2n+1},s_{2n+2}]$, we have that  $\mu_{s.c.}^{s_{2n+1}} = \mu_{s.c.}^{s_{2n+2}}$. On the one hand this implies that $\tilde{\mu}_n := \mu_{s.c.}^{s_{2n+1}} + \mu_{a.c.}$ is a good approximation of $\mu$ in the sense that 
    \begin{equation}
        \label{e:wass}
    W_2 \big(\mu, \tilde{\mu}_n \big) =  W_2 \big(\mu_{s.c.}, \mu_{s.c.}^{s_{2n+2}} \big) \le  2^{-s_{2n+2}} \le T_{2n+1}^{-1/o(1)} , 
     \end{equation} 
  where the first inequality follows from~\eqref{eq:W2} and the second from~\eqref{eq:T and s}. On the other hand, for every $\eps' > 0$ and every $n \ge 1$ there exists a smooth measure $\mu'_n$ such that
\begin{equation}
\label{e:newdecomp} 
\|\mu_{s.c.}^{s_{2n+1}}-\mu'_n\| < \ep',
\end{equation} 
and $\mu'_n$ is a measure supported on $[-2,-1] \cup [1,2]$ which has a smooth density $w_n$ satisfying
\begin{equation}  
\label{e:denbound} \|w_n\|_{C^2(\R)} \le  c_{\eps'} 2^{3s_{2n+1}}. \end{equation}
 For instance one can obtain $\mu'_n$ by approximating each indicator in \eqref{e:muapprox} separately and summing them up. In particular \eqref{e:denbound} implies that the Fourier transform $\mathcal{F}[\mu'_n]$ satisfies
\begin{equation}
    \label{e:fourierbound}
    \|\F[\mu'_n]\|_{L^1(\Z)} \le  c \|w_n\|_{C^2(\R)} \le  c'_{\eps'} (2^{s_{2n+1}})^3 \le  T_{2n+1}^{{o(1)}}   . 
    \end{equation}

By \cite[Theorem 4.1]{bm22}, for every $\delta > 0$ and sufficiently large $n$ there exists a coupling of $f_{\mu|_{B(\delta)^c}}$ and $f_{\tilde{\mu}_n|_{B(\delta)^c}}$ such that
\[ \sup_{x \in B(T_{2n+1})} \max_{|k| \le 1} \textrm{Var}[ \partial^k(f_{\tilde{\mu}_n|_{B(\delta)^c}} - f_{\mu|_{B(\delta)^c}})(x) ] \le c T_{2n+1}^2 W^2_2(\mu, \tilde{\mu}_n)  \]
where $c > 0$ is an absolute constant. By Kolmogorov's theorem and the Borell-TIS inequality this implies that, for sufficiently large $n$
\begin{equation}
    \label{e:m5}
-\log \P \Big[ \sup_{x \in B(T_{2n+1})} | f_{\mu|_{B(\delta)^c}}(x) -  f_{\tilde{\mu}_{n}|_{B(\delta)^c}}(x)   | \ge 1 \Big]  \ge \frac{ c' T_{2n+1}^2}{ W^2_2(\mu, \tilde{\mu}_n) }  = T_{2n+1}^{1/o(1)} 
\end{equation}
where the final inequality used \eqref{e:wass}.

By applying the second statement of Lemma~\ref{l:m1} with the decomposition 
\[
(\mu_{s.c.}^{s_{2n}+1} + \mu_{a.c.})|_{B(\delta)^c}  =(\mu_n'+\mu_{a.c.})|_{B(\delta)^c}+\rho|_{B(\delta)^c} ,
\]
and using the bounds in \eqref{e:newdecomp}--\eqref{e:fourierbound}, we obtain the following: for every $\eps \in (0,\|\mu\|/2)$ there exists $\delta_0 > 0$ such that, for all $\delta \in (0, \delta_0)$ and $\beta \in [0,1)$, as $n \to \infty$
\begin{align}
\notag
-\log \P \Big[ \inf_{x \in B(T_{2n+1})}  &f_{\tilde{\mu}_n|_{B(\delta)^c}}(x) \ge - \sqrt{2 (\|\mu\|-\eps) ( 1 - \beta)   \log T_{2n+1}} \Big] \\ &
\notag
\ge \tfrac{ c \|\mu\|}{ \|\mathcal{F}[\mu_n']\|_{L^1(\Z)} + \|\mathcal{F}[\mu_{a.c.}|_{B(\delta)^c}]\|_{L^1(\Z)}}T^{\beta + \frac{\eps}{\|\mu\|}(1-\beta) -c' \sqrt{\frac{\|\rho\|}{\|\mu\|} } + o(1)}\\ &
\ge T_{2n+1}^{\beta + \frac{\eps}{2\|\mu\|}(1-\beta) + o(1)} 
\label{e:m4}
\end{align}
where $c,c' > 0$ are absolute constants, and the second inequality is by taking $\eps'$ sufficiently small. The result follows by combining \eqref{e:m5} and \eqref{e:m4} and adjusting constants.
\end{proof}

%%%%%%%%
\medskip
\appendix

\section{Basic properties of the capacity}
\label{a:capacity}
In this section we establish the dual representations of the capacity in \eqref{e:cap2} and \eqref{e:cap1}. While this result is classical (see, e.g., \cite{fug60}), we did not find the exact statement in the literature.  We work in a more general set-up of positive definite kernels, not necessarily stationary, and follow closely the presentation in \cite{ams14} which considered the one-dimensional case.

Let $K = K(x,y)$ be a symmetric positive definite kernel on either $\Z^d$ or $\R^d$, and in the latter case assume that $K$ is continuous. For $\nu,\eta \in \mathcal{M}$ define the energy
\[  E[ \nu, \eta ] = \int \int K(x,y) d\nu(x) d\eta(y) \ , \quad E[\nu] = E[\nu,\nu] . \]
Denote by $\mathcal{M}_{<\infty} \subset \mathcal{M}$ the subset of measures of finite energy $E[\nu] < \infty$. The reproducing kernel Hilbert space (RKHS) associated to $K$ is the set
\[ H = \left\{ h(\cdot) = \int K(\cdot,y) d\nu(y) :\: \nu \in \mathcal{M}_{<\infty}  \right\} \]
equipped with the inner product inherited from $E$, i.e.\
\[ \Big\langle \int K(\cdot,y) d\nu(y) ,  \int K(\cdot,y) d\eta(y) \Big\rangle_H  = E[\nu,\eta] . \] 
The RKHS has the following `reproducing property': for all $h \in H$,
\begin{equation}
\label{e:rep}
\langle h(\cdot), K(\cdot, y) \rangle_H = h(y) .
\end{equation}

\begin{proposition}
\label{p:basiccap}
For every compact $D$,
   \begin{equation}
\label{e:cap3}
\capa(D) := \Big( \inf_{ \nu \in \mathcal{P}(D) } \int_D \int_D K(x,y) d\nu(x) d\nu(y) \Big)^{-1}  = \inf \big\{ \|h\|^2_H : h \in H, h|_D  \ge 1 \big\} .
\end{equation}
Moreover if $\capa(D) < \infty$, then optimisers on both sides of \eqref{e:cap3} exist, and the optimiser on the right-hand side of \eqref{e:cap3} is unique and can be represented as
\[ h = \capa(D) \int K(\cdot,y) \, d\nu(y) , \]
where $\nu$ is any of the optimisers on the left-hand side.
\end{proposition}

\begin{proof}
Suppose first that the set $\{ h \in H : h|_D \ge 1\}$ is non-empty. Define the cone $P = \{ z  \in C^0(D) : z(x) \le 0\}$, where $C^0(D)$ is the set of continuous functions on $D$. The dual cone $P^\ast$ can be identified with the set of measures on $D$. Then by Lagrangian duality (see, e.g., \cite[Theorem 1, page 224]{lu69}) we have
\begin{align*}
      \inf \{ \|h\|_H  :  h \in H, h|_D \ge 1 \}    & =   \inf \{ \|h\|_H  :  h \in H,  (1 - h)|_D \in P \}     \\ 
      &  =   \sup_{\nu \in   \mathcal{P}(D) } \sup_{a \ge 0}  \inf_{h \in H} \Big(  \|h\|_H + a \int_D (1-h)(y) d\nu(y)   \Big)   
    \end{align*}
    and an optimiser $\nu$ in this expression exists. Observe that, for every $\nu \in \mathcal{P}(D)$ and $a \ge 0$,
    \begin{align*}
     \inf_{h \in H} \Big(  \|h\|_H + a \int_D (1-h)(y) d\nu(y)   \Big)  &  =  a +   \inf_{b \ge 0}   b \Big(  1 - a \sup_{ h \in H: \|h\|_H = 1} \int_D h(y) d\nu(y)   \Big)    \\
& = \begin{cases}
   -\infty & \text{if } a > \big( \sup_{ h \in H: \|h\|_H = 1} \int_D h(y) d\nu(y)   \Big)    \big)^{-1}  ,\\
  a  &  \text{if } a \le  \big( \sup_{ h \in H: \|h\|_H = 1} \int_D h(y) d\nu(y)   \Big)    \big)^{-1} ,\\
\end{cases}
\end{align*}
from which it follows that
\[  \inf \{ \|h\|_H  :  h \in H, h|_D \ge 1 \}    = \Big( \inf_{\nu \in   \mathcal{P}(D) }   \sup_{ h \in H: \|h\|_H = 1} \int_D h(y) d\nu(y)   \Big)^{-1}. \]
By the reproducing property \eqref{e:rep} we also have, for any $\nu \in \mathcal{P}(D)$,
\begin{align*}
 \sup_{ h \in H : \|h\|_H = 1}  \int_D h(y) d\nu(y)   & =   \sup_{ h \in H : \|h\|_H = 1}  \int_D \langle h(\cdot) , K(\cdot, y) \rangle_H  d \nu(y)   \\
  &  =  \sup_{ h \in H : \|h\|_H = 1}  \Big\langle h(\cdot) , \int_D  K(\cdot,y) d \nu(y)  \Big\rangle_H    \\
&  = ( E(\nu) )^{1/2} ,
\end{align*}
where the supremum in the previous expression is achieved for $h$ a constant multiple of $\int_D K(\cdot,y) d\nu(y)$. This proves that 
\begin{equation}
    \label{e:cap4}
  \inf \{ \|h\|_H  :  h \in H, h|_D \ge 1 \}   =     ( \inf_{\nu \in \mathcal{P}(D)}E(\nu) )^{-1/2}  =: \capa(D)^{1/2}  , 
  \end{equation}
and that the optimiser on the left-hand side of \eqref{e:cap4} is achieved by a constant multiple of $\int K(\cdot,y) d\nu(y)$ for any $\nu$ which minimises $E(\nu)$; one then checks that the relevant constant must be $\capa(D)$. The claimed uniqueness of the optimiser is a consequence of the convexity of $\|\cdot\|_H^2$ (see Claim \ref{c:convex}). This establishes the result in the case that $\capa(D) < \infty$. 

Suppose instead that $\{ h \in H :  h_D \ge 1 \}$ is empty; we need to prove that in this case $ \capa(D) = \infty$. Consider $\nu \in \mathcal{P}(D)$ which minimises~$E(\nu)$. If $\int K(\cdot,y) \, d\nu(y) \in H$ does not vanish, then by compactness it is bounded away from zero on $D$, and so after multiplying with a suitable constant, it exceeds $1$ on $D$, a contradiction. Hence there exists a $u \in D$ such that $\int K(u,y) \, d\nu(y)  = 0$. For every $\eps \in [0,1]$, define the measure $\nu_\eps = (1-\eps)\nu + \eps \delta_{u}$, where $\delta_u$ is a Dirac mass at $u$. One can check that this satisfies
\[ \frac{d}{d \eps} E(\nu_\eps) \Big|_{\eps = 0} = \frac{d}{d \eps}  \Big( ( 1 - \eps)^2 E(\nu) + \eps^2 K(u,u)  \Big) \Big|_\eps = - 2 E(\nu) ,  \]
where the first equality used that $\int K(u,y) \, d\nu(y)  = 0$. By the assumption of minimality we conclude that $E(\nu) = 0$, and so $ \capa(D) = \infty$.
\end{proof}

\smallskip
\section{Regular varying spectral singularities}
\label{a:taub}

In this section we give sufficient conditions for the spectral singularity to be regularly varying. This is a slight generalisation of known Tauberian-type results that are usually stated only for SGFs which are continuous and isotropic (see e.g.\ \cite[Theorem 3]{lo13}). Let $A_{\alpha,d}$ and $B_{\alpha,d}$ be constants defined as in Section \ref{ss:uni}.

\begin{proposition}
\label{p:taub}
Let $\alpha \in [0,d)$, let $w$ be slowly varying at infinity, and let $f$ be an SGF on $\R^d$ or $\Z^d$ satisfying one of the following:
\begin{enumerate}[(i)]
\item\label{itemtaub:1} As $x \to \infty$,
\[ K(x) \sim  B_{\alpha,d} |x|^{-\alpha} w(|x|) . \] 
\item\label{itemtaub:2} $\alpha>0$, and $\mu$ has density $\rho$ in a neighbourhood of the origin such that, as $\lambda \to 0$,
\[ \rho(\lambda) \sim A_{\alpha,d} |\lambda|^{\alpha-d} w(1/|\lambda|). \]
\end{enumerate}
Then, as $T \to \infty$,
\[\mu[B(1/T)] \sim T^{-\alpha} w(T) \]
\end{proposition}

\begin{remark}
\label{r:taub}
As a consequence of Proposition \ref{p:taub}, if $K$ is a regularly varying kernel on $\R^d$ or $\Z^d$ with index $\alpha \in [0,d)$, then it is asymptotic to a monotone decreasing function. For $\alpha \in (0,d)$ this is a general property of regularly varying functions \cite[Theorem 1.5.3]{bgt87}, whereas if $\alpha = 0$ it relies on positive definiteness, since it is not true of every slowly varying function.
\end{remark}
\begin{proof}[Proof of Proposition \ref{p:taub}]
We use dominated convergence arguments similar to those in Section \ref{s:rv}. We begin by working under the condition $(i)$. Let $\mathcal{S}_d$ denote the Fourier transform of the uniform probability measure on $B(T)$. Fix $v \in (0,1)$ and recall the pair $(\xi,\zeta)$ from Proposition~\ref{c:1} (we do not use $\phi$ from \eqref{eq:stf}, since here we use $(\xi,\zeta)$ on opposites sides of the Fourier transform). For a locally finite measure $\nu$ and $\theta > 0$, define $\nu^\theta = \nu * \xi^\theta$, where $\xi^\theta(\cdot) =  \theta^{-d} \xi( \cdot \theta) $. Since $\xi^\theta \ge 0$, $\|\xi^\theta\| = 1$, and $\xi^\theta$ is supported on $B(\theta)$, for every $\delta > 0$ we have
\begin{equation}
    \label{e:nubounds}
\nu^\theta[B(\delta -\theta)] \le \nu[B(\delta)] \le \nu^\theta[B(\delta +\theta)] . 
\end{equation}
Now fix $\theta > 0$. By \eqref{e:nubounds}, Parseval's theorem, and a change of variables
\begin{align*}
  \frac{ \mu[B(1/T)] }{T^{-\alpha} w(T) }  & \le \frac{ \mu^{\theta/T}[B( (1+\theta)/T)] }{T^{-\alpha} w(T) } \\
&  = (1+\theta)^{-d}   \int  |x|^{-\alpha}  \zeta(x/\theta) \mathcal{S}_d( x/(1+\theta)) \frac{ K(xT) }{ |xT|^{-\alpha} w(|x|T) }   \frac{ w(|x|T) }{w(T)} \, dx  .
 \end{align*}
Using condition $(i)$, and since $\zeta$ is bounded and rapidly decaying at infinity and $\mathcal{S}_d$ is bounded, a dominated convergence argument (as in \eqref{e:unikerdom}) gives that, as $T \to \infty$,
\[ \frac{ \mu^{\theta/T}[B( (1+\theta)/T)] }{T^{-\alpha} w(T) } \to  (1+\theta)^{-d}  B_{\alpha,d} \int  |x|^{-\alpha}  \zeta(x/\theta) \mathcal{S}_d( x/(1+\theta))   \, dx  . \]
By Parseval's theorem and \eqref{e:nubounds} again,
\[  (1+\theta)^{-d}   B_{\alpha,d} \int  |x|^{-\alpha}  \zeta(x/\theta) \mathcal{S}_d( x/(1+\theta))   \, dx   = \mu_{\alpha}^{\theta}[ B(1+\theta) ] \le \mu_\alpha[B(1 + 2\theta) ]   = (1+2\theta)^\alpha, \]
and so
\[ \limsup_{T \to \infty} \frac{ \mu[B(1/T)] }{T^{-\alpha} w(T) }   \le \lim_{\theta \to 0}  (1+2\theta)^\alpha  = 1 . \]
A similar argument shows that
\[ \liminf_{T \to \infty} \frac{ \mu[B(1/T)] }{T^{-\alpha} w(T) }   \ge \lim_{\theta \to 0} (1-2\theta)^\alpha   = 1 , \]
and combining gives the result.

The proof under condition $(ii)$ is simpler.  Let $T$ be sufficiently large so that $\mu|_{B(1/T)}$ has density $\rho$. By a change of variables,
\begin{align*}
&  \frac{ \mu[B(1/T)] }{T^{-\alpha} w(T)}  =    \int_{B(1)}  |\lambda|^{-\alpha} \frac{ \rho(|\lambda/T|) }{ |\lambda/T|^{\alpha-d} w(|T/\lambda|) }  \frac{ w(|T/\lambda|) }{w(T)} \, d\lambda .
 \end{align*}
Then using $(ii)$, a dominated convergence argument (as in \eqref{e:unidendom}) gives that, as $T \to \infty$,
\begin{align*} 
&   \int_{B(1)}  |\lambda|^{-\alpha} \frac{ \rho(|\lambda/T|) }{ |\lambda/T|^{\alpha-d} w(|T/\lambda|) }  \frac{ w(|T/\lambda|) }{w(T)} \, d\lambda   \to   A_{\alpha,d} \int_{B(1)}  |\lambda|^{-\alpha}   \, d\lambda  = \mu_\alpha[B(1)] = 1. \qedhere
\end{align*}
\end{proof}

\smallskip

\section{The Riesz capacity and equilibrium potential of the unit ball}

\label{a:explicit}

For completeness we record the explicit forms for the capacity $c_\alpha$ and equilibrium potential $h_\alpha$ of $B(1)$ with respect to the $\alpha$-Riesz kernel $k_\alpha$, which appear in our universality results stated in Section \ref{ss:uni}:

\begin{equation*}
c_\alpha = \begin{cases}  1 & \alpha = 0 ,\\  \frac{2^{2+\alpha-d} \pi^{\alpha+1/2} \Gamma( (2d-\alpha)/2 - 1 ) \Gamma((d-\alpha)/2) }{ \Gamma( (d-\alpha-1)/2) \Gamma(1+\alpha/2) \Gamma(d/2)^2 } & \alpha \in (0,d-2) , \\   
\frac{\pi^{d-2}}{\Gamma(d/2)^2} & \alpha = d-2 ,\\
\frac{\pi^\alpha}{\Gamma(1+\alpha/2)^2  } &  \alpha \in (d-2,d) , \end{cases}  
\end{equation*}
and 
\begin{equation*}
h_\alpha(x) = \begin{cases}  1 & \alpha = 0, \\  (\int_{\mathbb{S}^{d-1} } |1-y|^{-\alpha} \, dy )^{-1} \int_{\mathbb{S}^{d-1} } |x-y|^{-\alpha} \, dy  & \alpha \in (0,d-2) , \\ 
 \min\{1, |x|^{2-d}\} & \alpha = d-2, \\
  (\int_{B(1)}  |1-y|^{-\alpha}  \, d\nu_\alpha(y) )^{-1} \int_{B(1)}  |x-y|^{-\alpha}  \,  d\nu_\alpha(y)&  \alpha \in (d-2,d) , \end{cases}  
\end{equation*}
where $d \nu_\alpha(y) =  (1 - |y|^2)^{-(d-\alpha)/2}  \, dy$. A proof can be found in \cite[p.163]{lan72} (note the different normalisation of $k_\alpha$ used therein).

%%%%%%%%%
\smallskip
\bibliographystyle{halpha-abbrv}
\bibliography{persist}

\end{document}